\documentclass[11pt,leqno]{article} 

\usepackage{lmodern}
\providecommand{\keywords}[1]{\textbf{\small{Key words:}} #1}
\providecommand{\AMS}[1]{\textbf{\small{AMS subject classifications:}} #1}
\usepackage[utf8]{inputenc} 
\usepackage[T1]{fontenc}    
\usepackage{hyperref}       
\usepackage{url}            
\usepackage[english]{babel}
\usepackage{amsfonts}       
\usepackage{graphicx}

\usepackage[width=16cm, top=2.5cm]{geometry}
\graphicspath{{figures/}}
\usepackage{amsthm} 
\usepackage{float} 
\usepackage{physics} 
\usepackage{caption}
\usepackage{subcaption}

\newcommand{\ABS}[2]{[#1]_{{#2}_{_-}}^{{#2}_{_+}}}

\def\calO{\mathcal{O}}
\def\calM{\mathcal{M}}
\def\calR{\mathcal{R}}
\def\tcalR{\widetilde{\mathcal{R}}}

\def\N{\mathbb{N}}
\def\R{\mathbb{R}}

\def\D{\mathrm{d}}
\def\pp{p_{_+}}
\def\pn{p_{_-}}
\def\Pp{]_{\pn}^{\pp}}
\def\rmc{\mathrm{c}}
\def\rmi{\mathrm{i}}
\def\rme{\mathrm{e}}
\def\tu{{\tilde u}}
\def\tA{\tilde A}
\def\eps{\varepsilon}
\def\ogamma{{\overline{\gamma}}}
\def\odelta{{\overline{\delta}}}

\def\w{w}
\def\u{\textbf{u}}
\def\cA{\Gamma_3}
\def\cB{\Gamma_2}
\def\tcA{\widetilde{\Gamma}_3}
\def\tcB{\widetilde{\Gamma}_2}
\def\bT{\textbf{T}}
\def\c{R}
\def\s{S}

\DeclareMathOperator{\sgn}{sgn}

\theoremstyle{plain}
\newtheorem{theorem}{Theorem}[section] 

\newtheorem{proposition}[theorem]{Proposition} 
\newtheorem{corollary}[theorem]{Corollary} 
\newtheorem{lemma}[theorem]{Lemma} 

\theoremstyle{definition}
\newtheorem{remark}[theorem]{Remark} 
\newtheorem{hypothesis}[theorem]{Hypothesis} 

\numberwithin{equation}{section} 

\title{Lyapunov coefficients for Hopf bifurcations in systems with piecewise smooth nonlinearity}

\date{January 19, 2023} 
\author{Miriam Steinherr Zazo\thanks{University of Bremen, Germany, Department 3 -- Mathematics, m.steinherr@uni-bremen.de}  \and Jens D.M. Rademacher
	\thanks{University of Hamburg, Germany, Department of Mathematics, jens.rademacher@uni-hamburg.de}}

\hypersetup{
pdftitle={Lyapunov coefficients for Hopf bifurcations in systems with piecewise smooth nonlinearity},
pdfsubject={},
pdfauthor={Miriam Steinherr Zazo, Jens D.~M.~Rademacher},
pdfkeywords={First keyword, Second keyword, More},
}

\begin{document}
\maketitle

\begin{abstract}
	Motivated by models that arise in controlled ship maneuvering, we analyze Hopf bifurcations in systems with piecewise smooth nonlinear part. In particular, we derive explicit formulas for the generalization of the first Lyapunov coefficient to this setting. This generically determines the direction of branching (super- versus subcriticality), but in general this differs from any fixed smoothing of the vector field. We focus on nonsmooth nonlinearities of the form $u_i|u_j|$, but our results are formulated in broader generality for systems in any dimension with piecewise smooth nonlinear part. In addition, we discuss some codimension-one degeneracies and apply the results to a model of a shimmying wheel.
\end{abstract}

\keywords{	{\small degenerate Andronov-Hopf bifurcation, nonsmooth systems, normal form,\\
		\hspace*{2.6cm} invariant manifolds}
}

\medskip
\AMS{ {\small
	34C23,  
	37G15,  
	74H60,   
	70K42   
}}

\section{Introduction}\label{s:intro}
The identification of characteristic parameters for bifurcations is one of the key goals of bifurcation theory. For the Andronov-Hopf bifurcation from equilibria to periodic orbits, the most relevant characteristic parameter is the first Lyapunov coefficient, $\sigma_s\in\R$. If it is nonzero, it determines the scaling and the direction of bifurcation relative to real part of the critical eigenvalues of the linearization at the equilibrium. It is well known that the truncated normal form of the radial component on the center manifold reads
\begin{equation}
	\dot{u} = \mu u +\sigma_s u^3,
	\label{Npitchfork}
\end{equation}
with parameter $\mu\in\R$. {An excellent exposition of Andronov-Hopf bifurcation theory and applications can be found in \cite{Marsden76}, see also \cite{Guckenheimer, Kuznetsov}.}
In Figure~\ref{f:Hopf}(a,b) we plot the associated pitchfork bifurcation for different signs of $\sigma_s$.
\begin{figure}
	[bt]
	\begin{tabular}{cccc}
		\includegraphics[width= 0.15\textwidth]{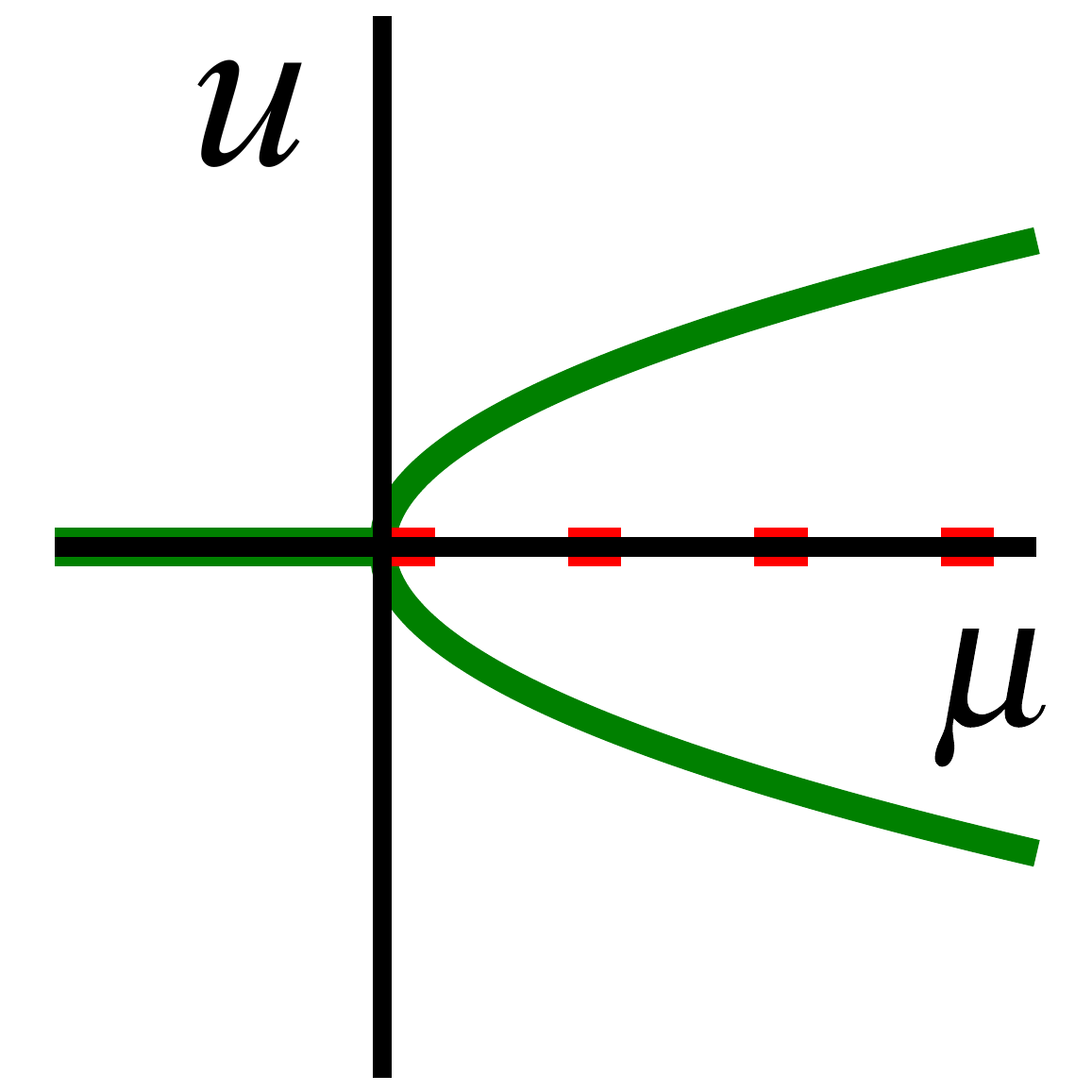}\hspace*{1cm}
		&\includegraphics[width= 0.15\textwidth]{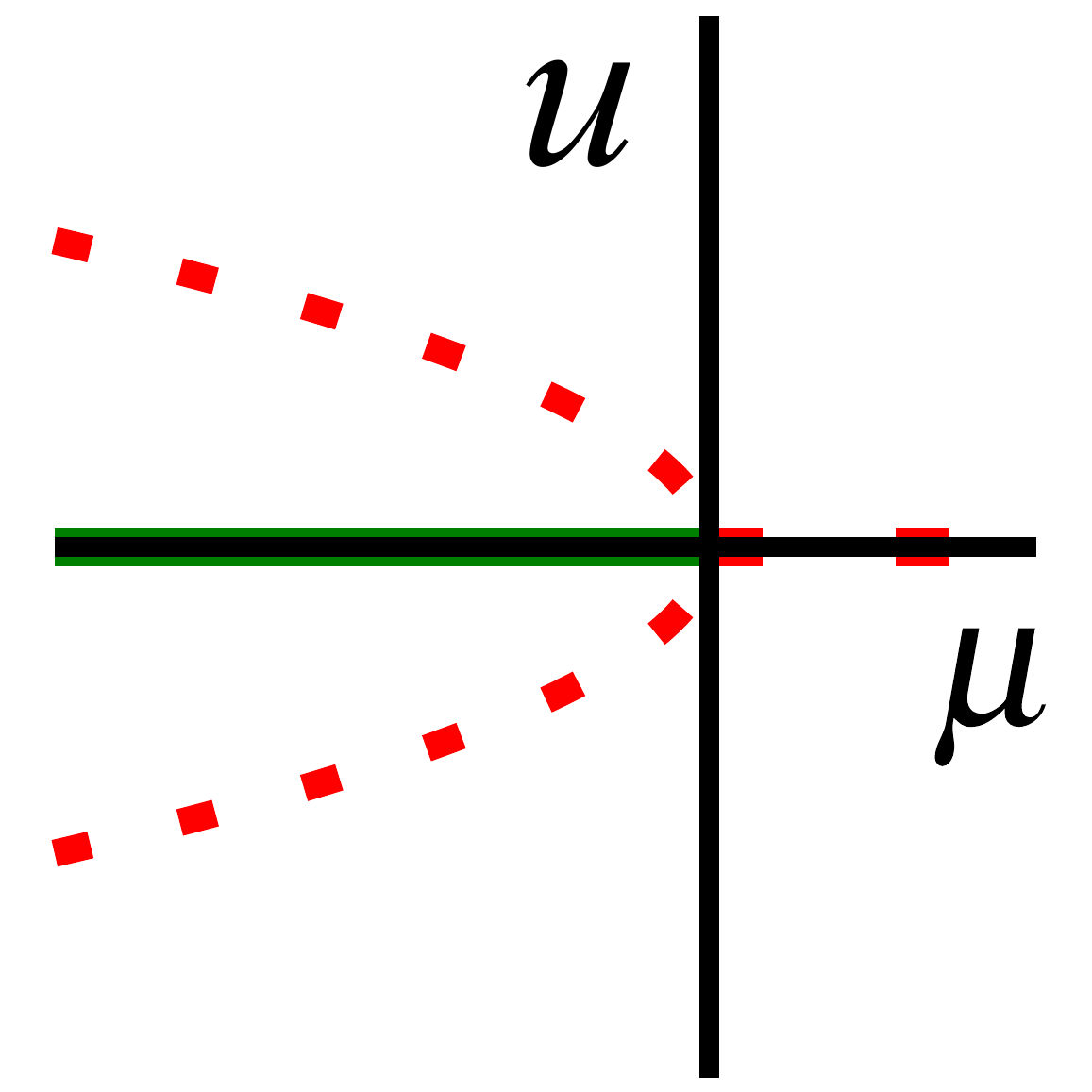}\hspace*{1.5cm}
		&\includegraphics[width= 0.15\textwidth]{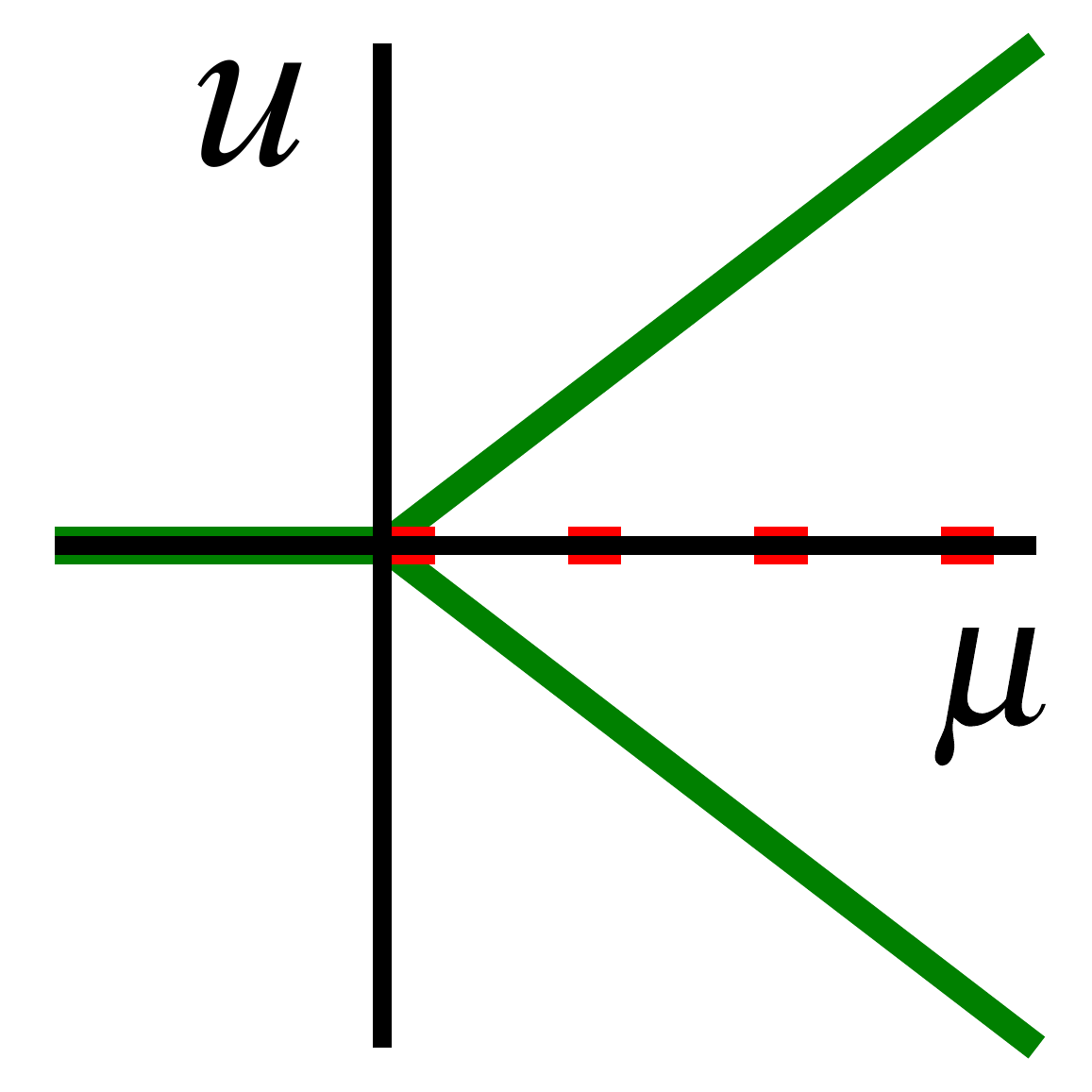}\hspace*{1cm}
		& \includegraphics[width= 0.15\textwidth]{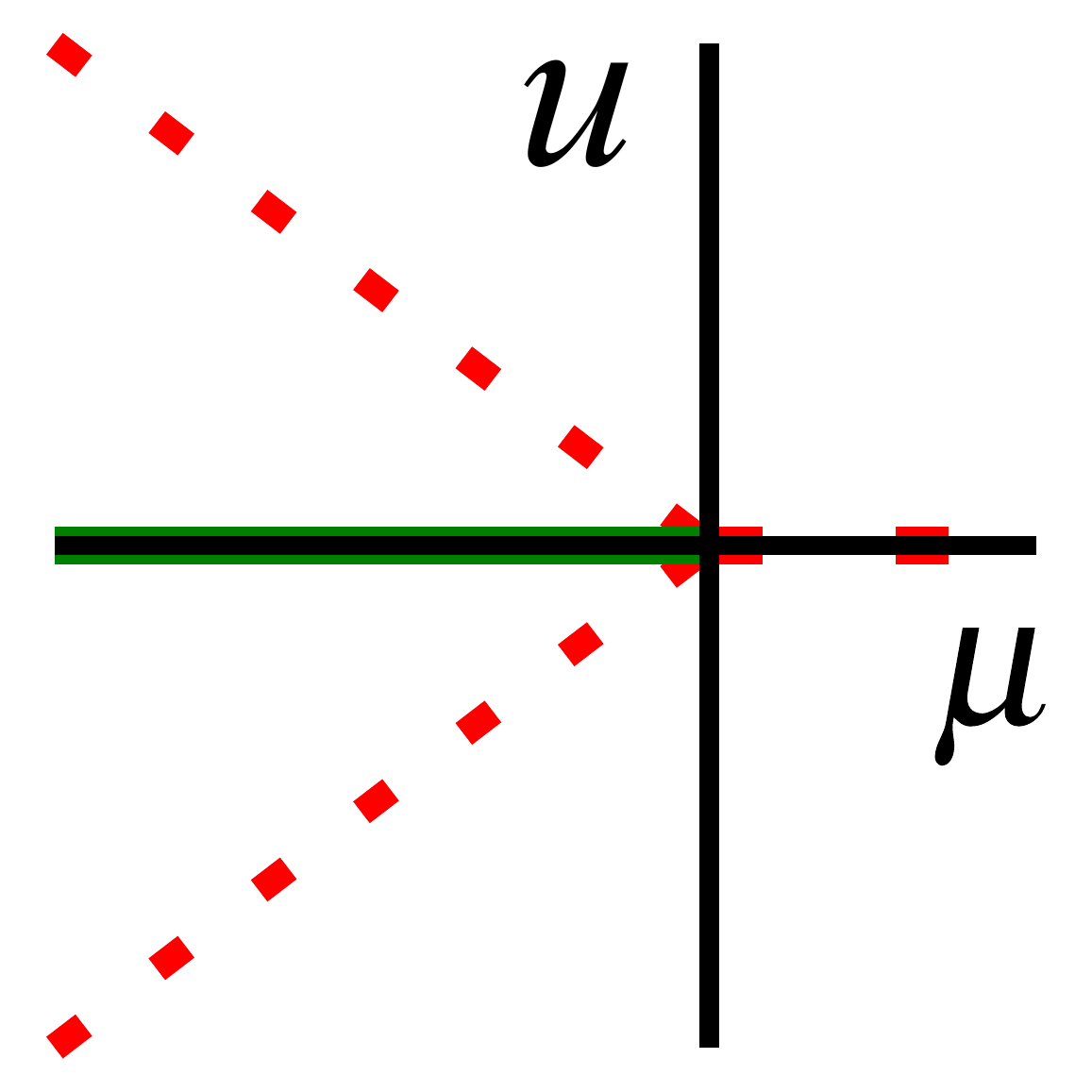}\\
		(a) & (b) &(c) &(d)
	\end{tabular}
	\caption{(a) Supercritical, $\sigma_s=-1$, and (b) subcritical, $\sigma_s=1$, pitchfork bifurcation of \eqref{Npitchfork} with stable (green) and unstable (red dashed) equilibria. In (c,d) we plot the analogous `degenerate' pitchforks for the nonsmooth case \eqref{Dpitchfork}, with $\sigma_{_\#}=-1 ,1$, respectively. }
	\label{f:Hopf}
\end{figure}
Generically, $\sigma_s\neq0$ and the bifurcating periodic orbits either coexist with the more unstable equilibrium, the supercritical case $\sigma_s<0$, or with the more stable equilibrium, the subcritical case $\sigma_s>0$. This distinction is relevant for applications since the transition induced by the bifurcation is a `soft' first order phase transition in the supercritical case, while it is `hard' in the subcritical one. Indeed, the transition is `safe' in a control sense in the supercritical case and `unsafe' in the subcritical one, where the local information near the equilibrium is insufficient to determine the dynamics near the unstable equilibrium.

Therefore, a formula for the first Lyapunov coefficient is important from a theoretical as well as applied viewpoint. In generic smooth bifurcation theory, such a formula is well known in terms of quantities derived from the Taylor expansion to order three in the equilibrium at bifurcation, {e.g.,} \cite{Kuznetsov}. However, this cannot be applied in nonsmooth systems. Nonsmooth terms appear in models for numerous phenomena and their study has gained momentum in the past decades,
as illustrated by the enormous amount of literature, see \cite{BookAlan,KuepperHoshamWeiss2013,NonsmoothSurvey2012,TianThesis} and the references therein to hint at some; below we discuss literature that relates to our situation. 

In this paper, we provide explicit formulas for the analogues of the first Lyapunov coefficient in systems with regular linear 
term and Lipschitz continuous, but only piecewise smooth nonlinear terms, with jumps in derivatives across switching surfaces. We also discuss codimension-one degeneracies and the second Lyapunov coefficient. To the best of our knowledge, this analysis is new. 
Such systems can be viewed as mildly nonsmooth, but occur in various models, e.g., for ship maneuvering 
\cite{InitialPaper,FossenHandbook,ToxopeusPaper}, which motivated the present study. Here the hydrodynamic drag force 
at high-enough Reynolds number is a nonsmooth function of the velocity $u$. More specifically, a dimensional and 
symmetry analysis with $\rho$ being the density of water, $C_D$ the drag coefficient and $A$ the effective drag area, yields  
$$F_D = -\frac{1}{2}\rho C_D A u\abs{u}.$$
Effective hydrodynamic forces among velocity components $u_i, u_j$, $1\leq i,j\leq n$, with $n$ depending on the model type, are often likewise modeled by second order modulus terms: $u_i\abs{u_j}$, cf.\ \cite{FossenHandbook}.
For illustration, let us consider the corresponding nonsmooth version of \eqref{Npitchfork},
\begin{equation}
	\dot{u} = \mu u +\sigma_{_\#} u\abs{u},
	\label{Dpitchfork}
\end{equation}
where the nonlinear term has the odd symmetry of the cubic term in \eqref{Npitchfork} and is once continuously differentiable, but not twice. We note that in higher dimensions, the mixed nonlinear terms $u_i\abs{u_j}$ for $i\neq j$, are differentiable at the origin only.

In Figure \ref{f:Hopf}(c,d) we plot the resulting bifurcation diagrams. Compared with \eqref{Npitchfork},
the amplitude scaling changes from $\sqrt{\mu}$ to $\mu$ and the rate of convergence to the bifurcating state changes from $-2\mu$ to $-\mu$.
Indeed, in this scalar equation case, the singular coordinate change $u=\tilde u^2$ transforms \eqref{Dpitchfork} into \eqref{Npitchfork} up to time rescaling by $2$. However, there is no such coordinate change for general systems of equations with nonsmooth terms of this kind.

More generally, we consider $n$-dimensional systems of ordinary differential equations (ODE) of the form
\begin{equation}\label{e:abstract0}
	\dot \u= A(\mu)\u+G(\u),
\end{equation}
with matrix $A(\mu)$ depending on a parameter $\mu\in\R$, and Lipschitz continuous nonlinear $G(\u)=\calO(|\u|^2)$. We shall assume the nonlinearity is smooth away from the smooth hypersurfaces $H_j$, $j=1,\ldots,n_H$, the \emph{switching surfaces}, which intersect pairwise transversally at the equilibrium point $\u_*=0$. We assume {further that} the smoothness of $G$ extends to the boundary within each component of the complement of $\cup_{j=1}^{n_H} H_j\subset \R^n$.

The bifurcation of periodic orbits is ---as in the smooth case--- induced by the spectral structure of $A(\mu)$, which is (unless stated otherwise) hyperbolic except for a simple complex conjugate pair that crosses the imaginary axis away from the origin as $\mu$ crosses zero.

\medskip
Our main results may be summarized informally as follows.

We infer from the result in \cite{IntegralManifold} that the center manifold of the smooth case is replaced by a Lipschitz invariant manifold (Proposition \ref{prop:inv_man}), and directly prove that a unique branch of periodic orbits emerges at the bifurcation (Theorem \ref{t_per_orb}). Moreover, we prove that the quadratic terms of $G$ are of generalized second order modulus type if $G$ is piecewise $C^2$ smooth (Theorem \ref{t:abstractnormal}). Here the absolute value in the above terms is replaced by  
\begin{align}
	[u\Pp = 
	\begin{cases}
		\pp u, & u\geq 0, \\
		\pn u, & u<0,
	\end{cases}
	\label{gen_abs_val}
\end{align}
where $\pn,\pp\in\R$ are general different slopes left and right of the origin, respectively. 

This already allows to express the first Lyapunov coefficient in an integral form, but its explicit evaluation is somewhat involved, so that we defer it to \S\ref{Gen_linear_part}. Instead, we start with the simpler case when $A$ is in block-diagonal form, in normal form on the center eigenspace, and of pure second order modulus form ($\pp=-\pn=1$). For the planar situation, we derive a normal form of the bifurcation equation with rather compact explicit coefficients using averaging theory (Theorem \ref{t_averaging}). Beyond the first Lyapunov coefficient $\sigma_{_\#}$, this includes the second Lyapunov coefficient $\sigma_2$, which becomes relevant when $\sigma_{_\#}=0$, and which explains how the smooth quadratic and cubic terms interact with the nonsmooth ones in determining the bifurcation's criticality. 

For refinement and generalization, and to provide direct self-contained proofs, we proceed using the Lyapunov-Schmidt reduction for the boundary value problem of periodic solutions, and refer to this as the `direct method' (\S\ref{s:direct}). We also include a discussion of the Bautin-type bifurcation in this setting, when $\sigma_{_\#}=0$. Concluding the planar case, we {generalize} the results to arbitrary $\pp, \pn$ (\S\ref{s:appplanar}). 

These results of the planar case readily generalize to higher dimensions, $n>2$, with additional hyperbolic directions (\S\ref{s:3D}, \ref{s:nD}). In addition, we apply the direct method to the situation with an additional nonhyperbolic direction {in the sense that the linearization at bifurcation has three eigenvalues on the imaginary axis: one zero eigenvalue and a complex conjugate pair.}
{In this case we show that either no periodic solutions bifurcate or two curves bifurcate (Corollaries \ref{c:3D}, \ref{c:nD}), depending on the sign of a particular combination of coefficients of the system that is assumed to be nonzero.}

{Concluding the summary of main results,} in \S\ref{Gen_linear_part}, we discuss the modifications in case the linear part is not in normal form. 

\medskip
For illustration, we consider the planar case with linear part in normal form, so \eqref{e:abstract0} with $\u=(v,\w)$ reads
\begin{equation}\label{e:planar0}
	\begin{cases}
		\dot{v} = \mu v - \omega \w + f\left( v, \w \right), \\
		\dot{\w} = \omega v + \mu \w + g\left( v, \w \right),
	\end{cases}
\end{equation}
and the case of purely quadratic nonlinearity with second order modulus terms gives
\begin{equation}\label{e:quadnonlin}
	\begin{aligned}
		f\left( v, \w \right) &= a_{11}v\abs{v} + a_{12}v\abs{\w} + a_{21}\w\abs{v} + a_{22}\w\abs{\w},\\
		g\left( v, \w \right) &= b_{11}v\abs{v} + b_{12}v\abs{\w} + b_{21}\w\abs{v} + b_{22}\w\abs{\w},
	\end{aligned}
\end{equation}
where $a_{ij}$, $b_{ij}$, $1\leq i,j\leq 2$, are real parameters. In this simplest situation, our new first Lyapunov coefficient reads
\begin{equation}\label{sigma1}
	\sigma_{_\#} = 2a_{11}+a_{12}+b_{21}+2b_{22}
\end{equation}
and we plot samples of bifurcation diagrams in Figure~\ref{f:auto} computed by numerical continuation with the software \texttt{Auto} \cite{auto}. {For these} we used numerically computed Jacobians and avoided evaluation too close to the nonsmooth point by choosing suitable step-sizes and accuracy.

\begin{figure}
	\centering
	\begin{tabular}{cc}
		\includegraphics[width= 0.4\linewidth]{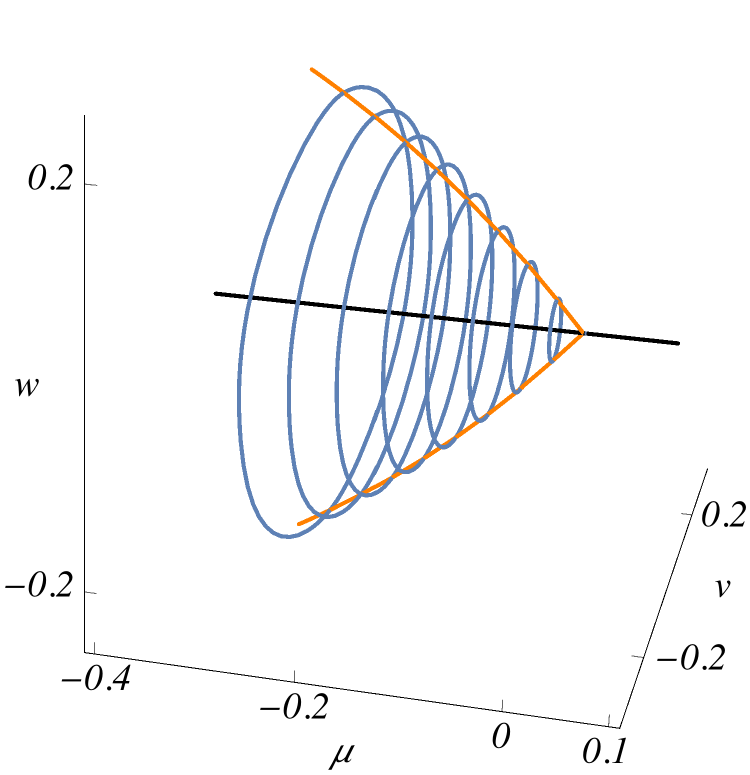} &
		\includegraphics[width= 0.4\linewidth]{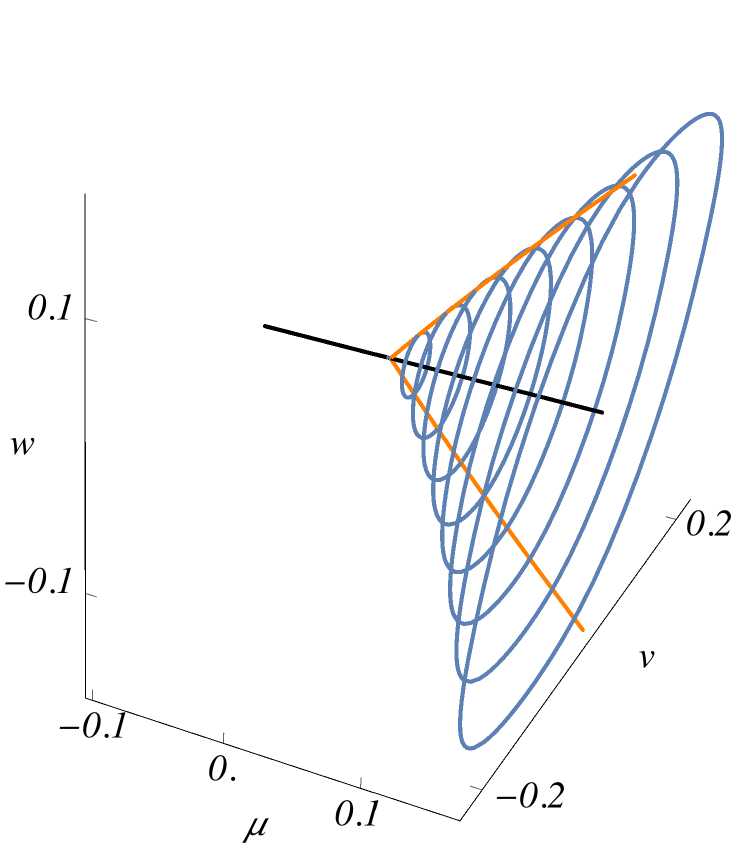}\\
		(a) & (b)
	\end{tabular}
	\caption{Plotted are bifurcation diagrams of \eqref{e:planar0} with \eqref{e:quadnonlin} computed by numerical continuation. Blue curves are periodic orbits, black lines the equilibrium at the origin, and orange the extrema in $w$, showing the nonsmooth bifurcation. In (a) we use $a_{ij}=b_{ij}=1$ except
		$b_{21}=-1$, so that $\sigma_{_\#}=4>0$ (subcritical). In (b) $b_{22}=-3$, so that $\sigma_{_\#}=-4$ (supercritical).}
	\label{f:auto}
\end{figure}

In comparison, the classical first Lyapunov coefficient for purely cubic nonlinearity, i.e., $|\cdot |$ replaced by $(\cdot)^2$, reads
\[
\sigma_s = 3a_{11}+a_{12}+b_{21}+3b_{22}.
\]
The leading order expansion of the radius $r_{_\#}$ and $r_s$ of bifurcating periodic solutions in these cases read, respectively, 
\[
r_{_\#}(\mu) = -\frac{3\pi}{2\sigma_{_\#}}\mu + \mathcal{O}\left(\mu^2\right), \qquad
r_s(\mu) = 2\sqrt{-\frac{2}{\sigma_s}\mu} + \mathcal{O}\left(\mu\right).
\]
We show that for $\sigma_{_\#}=0$ the bifurcation takes the form 
\[
r_0=\sqrt{-\frac{2\pi\omega}{\sigma_2}\mu}+\mathcal{O}\left(\mu\right),
\]
analogous to the smooth case, but with second Lyapunov coefficient in this setting given by
\begin{equation}\label{sigma2}
	\begin{aligned}
		{\sigma}_2 &= \frac{1}{9}\Big[
		11( 2 a_{11}a_{22} - a_{12}b_{11} + a_{22}b_{21} - 2 b_{11}b_{22} ) 
		+ 13 a_{11}b_{12} - 13 a_{21}b_{22} - 2 a_{12}a_{21} + 2 b_{12}b_{21} \Big] \\
		&\quad+ \frac{\pi}{4}( 2a_{22}b_{22}+a_{11}a_{21}+a_{12}a_{22}-b_{11}b_{21}- b_{12}b_{22}-2a_{11}b_{11} ).
	\end{aligned}
\end{equation}
In presence of smooth quadratic and cubic terms, the latter is modified with the classical terms, as we present in \S\ref{AV_S}. 

Despite the similarity of $\sigma_{_\#}$ and $\sigma_s$, it turns out that there is no fixed smoothing of the absolute value function that universally predicts the correct criticality of Hopf bifurcations in these systems (\S\ref{s:smooth}).
For exposition of this issue, consider the $L^\infty$-approximations, with regularization parameter $\eps> 0$, of the absolute value function $f_1(x) = \abs{x}$, given by $f_2(x)=\frac{2}{\pi}\arctan\left(\frac{x}{\eps}\right)x$ (cf.\ \cite{Leine}), and $f_3(x)=\frac{2}{\pi}\arctan\left(\frac{x}{\eps}(x-1)(x+1)\right)x$, a convex and a nonconvex approximation, respectively.

This last function approximates the absolute value for large (absolute) values of $x$. We plot the graphs in Figure \ref{3_Cases}(a) and the bifurcation diagrams for $\dot{x}=\mu x-f_i(x)x$, $i\in\{1,2,3\}$, in Figure \ref{3_Cases}(b). In particular,  $f_3$ gives a `microscopically' wrong result, which is nevertheless correct `macroscopically'.

\begin{figure}
	\begin{subfigure}{.47\textwidth}
		\centering
		\includegraphics[width= 0.6\linewidth]{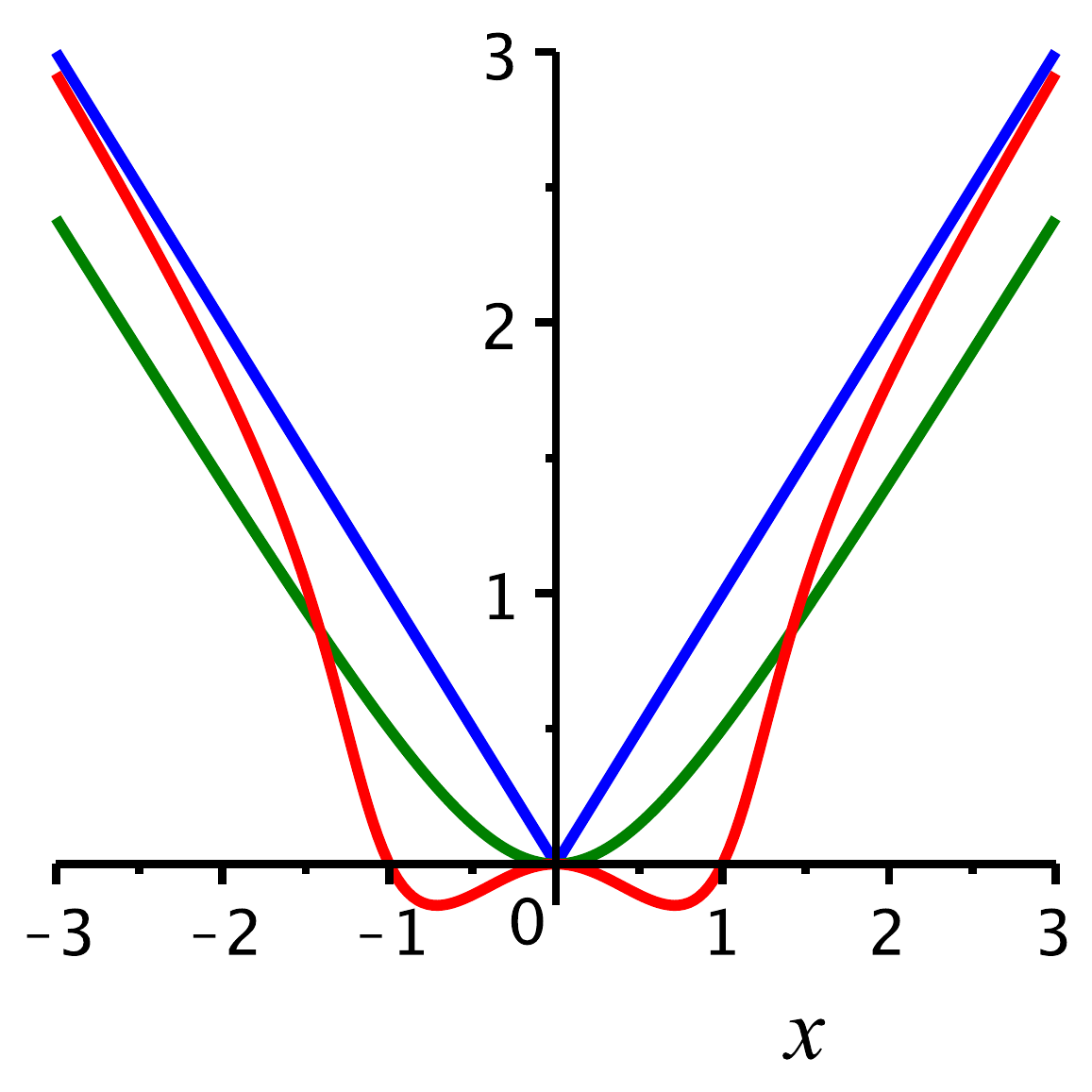}
		\subcaption{In blue $f_1(x)$, in green $f_2(x)$ and in red $f_3(x)$.}
	\end{subfigure}
	\begin{subfigure}{.47\textwidth}
		\centering
		\includegraphics[width= 0.6\linewidth]{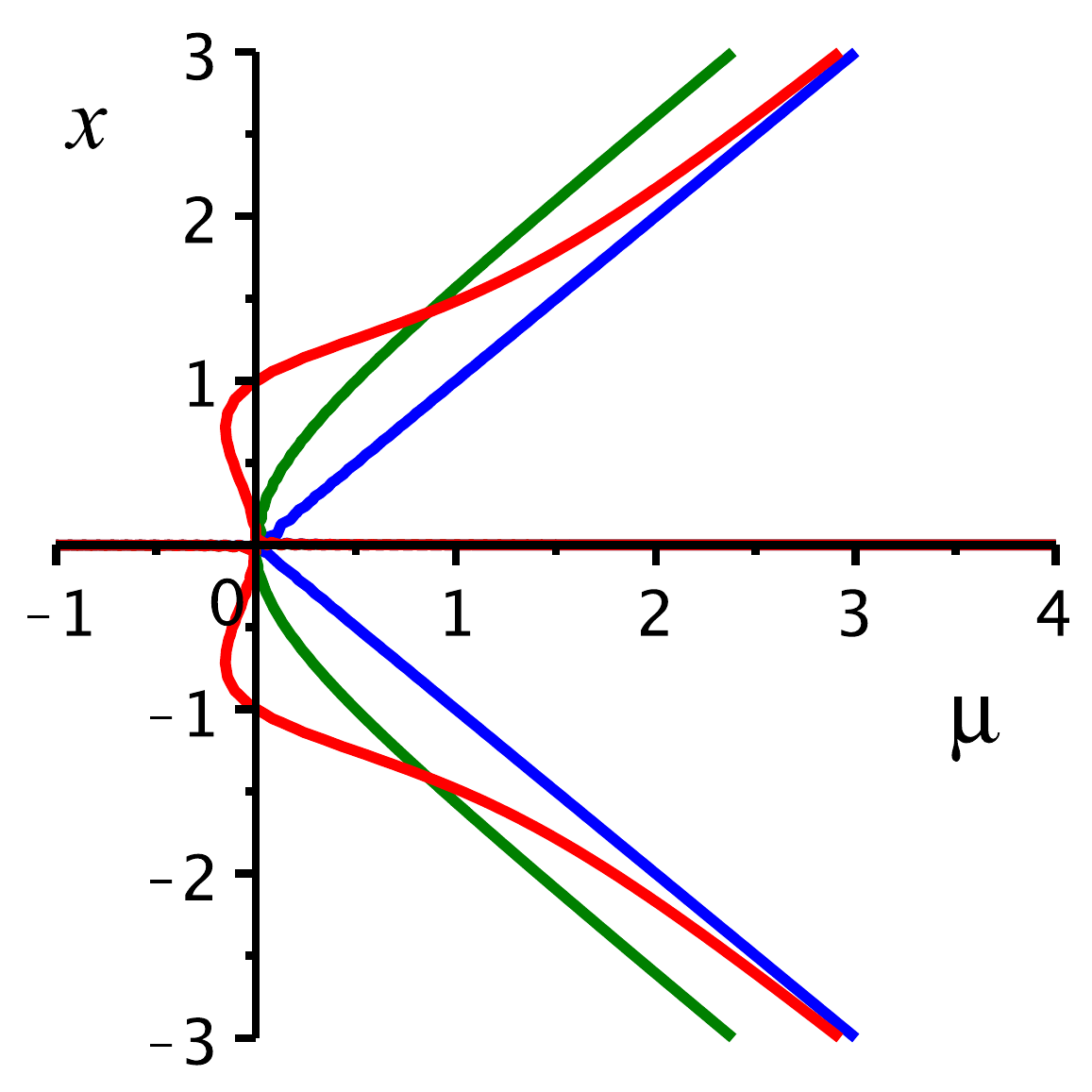}
		\subcaption{Bifurcation diagrams respect to $\mu$.}
	\end{subfigure}
	\caption{Comparison of bifurcation diagrams for $f_1, f_2$ and $f_3$ as in the text.}
	\label{3_Cases}
\end{figure}

Indeed, nonsmooth terms in models typically stem from passing to a macro- or mesoscopic scale such that  microscopic and smooth information is lost. Hence, the bifurcations in such models carry a macroscopic character and it is not surprising that an arbitrary smoothing changes this nature microscopically: a macroscopically supercritical bifurcation might show a subcritical behaviour on the microscopic level. However, the relevant information for the model is the macroscopic character, and ---for the class of models considered--- this is given by our newly derived Lyapunov coefficients.

The basic idea of proof is to change coordinates to a nonautonomous system for which the lack of smoothness is in the time variable only, so that averaging and the `direct method' can be applied.
We remark that in standard ODE literature on existence and bifurcations, smoothness of the time variable is often assumed, for example \cite{ChowHale}, but it is not needed in parts relevant for us. Indeed, merely continuity in time is for instance considered in \cite{CoddLev,Hartman,BookRasmussen}.

\medskip
{In order to demonstrate how to apply our method in a concrete case, we discuss in \S\ref{s:shim} the 3D model of a shimmying wheel from \cite{SBeregi}. This systems is of the form \eqref{e:abstract0} with pure second order modulus nonlinearity,
	but linear part not in normal form, though it has a nonzero real eigenvalue as well as a pair of complex conjugate eigenvalues that crosses the imaginary axis upon parameter change. 
	We fully characterize the resulting bifurcations in Theorem \ref{t:shym}.}

\medskip
We briefly discuss related literature. As mentioned, piecewise smooth vector fields have been widely investigated in many different applications as well as from a theoretical point of view, leading to a broad analysis in terms of bifurcation theory, cf.\ \cite{ReviewAlan,KuepperHoshamWeiss2013,NonsmoothSurvey2012}. 
Herein theory of continuous as well as discontinuous vector fields, e.g., \cite{Filippov1988, Kunze2000}, is used and further developed. A major distinction between our case and the systems studied in the literature is that we assume a separation a priori of a linear part and a nonsmooth nonlinear part. Broadly studied are the more general switching differential systems that are discontinuous across a switching surface or piecewise linear. These have been analyzed in various different forms, and we refer to \cite{TianThesis} for an exhaustive list of references; a typical case of discontinuity across the switching manifolds arises from the Heaviside step functions in biology neural models, e.g., \cite{Amari1977,Coombes2005,Harris2015}.
In analogy to center manifolds, the existence of invariant manifolds and sets 
has been investigated in \cite{IntegralManifold} for Carath\'eodory vector fields, and in \cite{KuepperHosham2010,KuepperHoshamWeiss2012} for vector fields with one switching surface. 
The bifurcation of periodic orbits in planar vector fields with one axis as the switching line has been studied in \cite{CollGasullProhens,GasTorr2003} via one-forms, and characteristic quantities have been determined, though the aforementioned Lyapunov coefficients are not included. Planar Hopf bifurcations for piecewise linear systems have been studied via return maps for one switching line in \cite{KuepperMoritz}, for several switching lines meeting at a point in \cite{BuzMedTor2018,Simpson2019,ZouKuepper2005}, 
and for nonintersecting switching manifold using Li\'enard forms in \cite{LlibrePonce}. Higher dimensional Filippov-type systems with a single switching manifold are considered in \cite{ZouKuepperBeyn}, which allows to abstractly study the occurrence of a Hopf bifurcation also for our setting; see also \cite{KuepperHoshamWeiss2013}. 
An approach via averaging with focus on the number of bifurcating periodic orbits for discontinuous systems is discussed in \cite{LlibreEtAl2017}. 
Nevertheless, we are not aware of results in the literature that cover our setting and our results on the explicit derivation of Lyapunov coefficients and the leading order analysis of bifurcating periodic solutions.

\medskip
This paper is organized as follows. In \S\ref{s:abstract} we discuss the abstract setting and provide basic results for the subsequent more explicit analysis. This is conducted in \S\ref{Planar_Section} for the planar case with linear part in normal form and nonlinear part with pure second order modulus terms for the nonsmooth functions, together with quadratic and cubic smooth functions. In \S\ref{s:general} we generalize the absolute value to arbitrary slopes, the system to higher space dimensions, and consider the linear part not being in normal form. Finally, in \S\ref{s:shim} we illustrate the application of our method and results to a concrete model.

\section{Abstract viewpoint}\label{s:abstract}
In this section we discuss the abstract starting point for our setting and motivate the specific assumptions used in the following sections. We consider an $n$-dimensional system of autonomous ODEs in an open set $U\subset\R^n$, with $0\in U$, of the form
\begin{equation}\label{e:abstract}
	\dot \u= A(\mu)\u+G(\u),
\end{equation}
with matrix $A(\mu)$ depending on a parameter $\mu\in\R$, and Lipschitz continuous nonlinear $G(\u)$.

We are interested in a detailed analysis of Hopf-type bifurcations at the equilibrium point $\u_*=0$. This requires control over the linear part, which is separated a priori in \eqref{e:abstract} from the potentially nondifferentiable nonlinear part ---note that $G$ is differentiable at $\u_*$ but not necessarily elsewhere. As usual for Hopf bifurcations, we assume that a pair of simple complex conjugate eigenvalues of $A(\mu)$ crosses the imaginary axis upon moving $\mu\in\R$ through zero. We collect the structural hypotheses on $A$ and $G$ without further loss of generality to our leading order analysis.

\begin{hypothesis}\label{h:AG}
	The eigenvalues of $A(\mu)$ are given by $\mu\pm \rmi\omega(\mu)$ with smooth nonzero $\omega(\mu)\in\R$  and all other eigenvalues have nonzero real part at $\mu=0$. The nonlinearity $G$ is Lipschitz continuous and satisfies $G(\u)=\calO(|\u|^2)$.
\end{hypothesis}

We denote by $E^\rmc$ the center eigenspace of $A(0)$ of the eigenvalues $\pm\rmi\omega(0)$, and first note the following result on invariant manifolds due to \cite{IntegralManifold}, which corresponds to center manifolds in the smooth case.

\begin{proposition}\label{prop:inv_man}
	Under Hypothesis~\ref{h:AG}, for $0\leq |\mu|\ll1$ there exist $2$-dimensional Lipschitz continuous invariant manifolds $\calM_\mu$ in an open neighborhood $U_*\subset U$ of $\u_*$, which contain $\u_*$ and all solutions that stay in $U_*$ for all time. Furthermore, if at $\mu=0$ all eigenvalues other than $\pm i\omega(0)$ have strictly negative real part, then each $\calM_\mu$ is (transversally) exponentially attractive. In addition, each $\calM_\mu$ is a Lipschitz continuous graph over $E^\rmc$ that depends Lipschitz continuously on $\mu$.
\end{proposition}

\begin{proof}
	The statements follow directly from \cite{IntegralManifold} upon adding a trivial equation for the parameter, as usual in center manifolds.  As for center manifolds, the proof relies on cutting off the vector field near $\u_*$, cf.\ \cite[Remark 6.2]{IntegralManifold}, and we infer the existence of $\calM_\mu$ from \cite[Corollary 6.4]{IntegralManifold}. The assumptions are satisfied since $G$ is of quadratic order, which means the Lipschitz constant of $G$ becomes arbitrarily small on small balls centered at $\u_*$. The stability statement follows from \cite[Corollary 6.5]{IntegralManifold}. 
\end{proof}
More refined stability information and estimates can be found in  \cite{IntegralManifold}.

Next, we present a variant of the standard Andronov-Hopf bifurcation theorem, cf.\ \cite{ChowHale}, which does not use any additional smoothness assumption. Here the uniqueness part relies on Proposition~\ref{prop:inv_man}, but the existence is independent of it. As mentioned, in case of a single switching surface, the abstract bifurcation of periodic solutions without smoothness statement concerning the branch follows from the results in \cite{ZouKuepperBeyn}, see also \cite{KuepperHoshamWeiss2013}.

\begin{theorem}
	\label{t_per_orb}
	Assume Hypothesis~\ref{h:AG}. 
	A locally unique branch of periodic solutions to \eqref{e:abstract} bifurcates from $\u_*=0$ at $\mu=0$. Specifically, there is a neighborhood $V\subset U$ of $\u_*$, such that for $0<|\mu|\ll1$ periodic solutions to \eqref{e:abstract} in $V$ are given (up to phase shift) by a Lipschitz continuous one-parameter family of $\tilde\omega(a)$-periodic solutions $\u_{\rm per}(t;a)$, $\mu=\mu(a)$ for $0\leq a\ll1$, 
	$\tilde\omega(0)=\omega(0)$, $\mu(0)=0$, whose projections into $E^\rmc$ have the complexified form 
	$a \rme^{\rmi \tilde\omega(a) t} + o(|a|)$. Moreover, we have the estimate $\mathrm{dist}(\u_{\rm per}(\cdot;a),E^{\rmc}) =\calO(a^{2})$.
\end{theorem}

This bifurcation is typically `degenerate' compared to the generic smooth Hopf bifurcation as in the example \eqref{Dpitchfork}, where {the bifurcating branch} is not $C^1$ through $u=0$. 
\begin{proof}
	We change coordinates such that $A(\mu)$ is in block-diagonal form with upper left 2-by-2 block for the eigenspace $E^\rmc$ having diagonal entries $\mu$ and anti-diagonal $\pm\omega(\mu)$, and remaining lower right $(n-2)$-dimensional block invertible at $\mu=0$; the modified $G$ remains of quadratic order and is Lipschitz continuous. Upon changing to cylindrical coordinates with vertical component $u=(u_3,\ldots,u_n)$, where $u_j$ are the scalar components of $\u$, we obtain 
	\begin{equation}\label{e:cylindrical0}
		\begin{aligned}
			\dot{r} &= \mu r +  \calR_1(r,u;\mu),\\
			r\dot{\varphi} &= \omega(0)r + \calR_2(r,u;\mu),\\
			\dot u &= \tA u + \calR_3(r,u;\mu).
		\end{aligned}
	\end{equation}
	Here $\tA$ is the invertible right lower block at $\mu=0$ and we suppress the dependence on $\varphi$ of $\calR_j$, $j=1,2,3$. Due to the Hypothesis~\ref{h:AG} in these coordinates we have the estimates $\calR_1(r,u;\mu) = \calO(r^2 +|\mu|(r^2 + |u|) + |u|^2)$, $\calR_{j}(r,u;\mu) = \calO(r^2 +|\mu|(r + |u|) + |u|^2)$, $j=2,3$. We seek initial conditions $r_0,u_0,\varphi_0$ and a parameter $\mu$ that permit a periodic solution near the trivial solution $r=u=0$. 
	By Proposition~\ref{prop:inv_man} any such periodic orbit is a Lipschitz graph over $E^\rmc$ so that there is a periodic function $\tu$ with $u=r\tu$. 
	Let $T>0$ denote the period and suppose $r(t)=0$ for some $t\in[0,T]$. Then $u(t)=0$ and therefore $\u(t)=\u_*$, so that $\u=\u_*$ is the trivial solution. Hence, we may assume that $r$ is nowhere zero and thus $\tu$ solves
	\[
	\dot \tu = \tA\tu + \tcalR_3(r,\tu;\mu),
	\]
	where $\tcalR_3(r,\tu;\mu) = \calO\big(r+|\mu|+|\tu|(|\mu|+r|\tu|)\big)$. By variation of constants we solve this for given $r, \varphi$ as
	\begin{equation}\label{e:tv}
		\tu(t) = e^{\tA t}\tu_0 + \int_0^t e^{\tA(t-s)}\tcalR_3(r(s),\tu(s);\mu) ds,
	\end{equation}
	with initial condition $\tu(0)=\tu_0$. $T$-periodic solutions solve in particular the boundary value problem 
	\begin{align}
		0&= \tu(T)-\tu(0) = \int_0^{T} \dot \tu(s)ds 
		\nonumber\\
		&= \int_0^{T}\tA e^{\tA s}\tu_0 ds +  \int_0^{T}\left(\tA \int_0^s e^{\tA(s-\tau)}\tcalR_3(r(\tau),\tu(\tau);\mu) d\tau + \tcalR_3(r(s),\tu(s);\mu) \right)ds \nonumber\\
		&= \left(e^{\tA T} - \mathrm{Id}\right)\tu_0 + \tcalR_4(r,\tu;\mu), \label{e:bvp0}
	\end{align}
	where $e^{\tA T} - \mathrm{Id}$ is invertible since  $\tA$ is invertible. 
	
	We have $\tcalR_4(r,\tu;\mu)= \calO\big(r_\infty +|\mu| + \tu_\infty(|\mu| + r_\infty \tu_\infty)\big)$ with $r_\infty = \sup\{r(t)\;|\; t\in[0,T]\}$, $\tu_\infty = \sup\{|\tu(t)|\;|\; t\in[0,T]\}$ and by \eqref{e:tv} there is a $C>0$ depending on $T$ with 
	\[
	\tu_\infty \leq C \big(|\tu_0| + r_\infty + |\mu| + \tu_\infty(|\mu|+r_\infty \tu_\infty)\big)
	\;\Leftrightarrow\;\big(1-C(|\mu|+r_\infty \tu_\infty)\big)\tu_\infty \leq C(|\tu_0| + r_\infty + |\mu|),
	\]
	so that for $0\leq |\tu_0|, r_\infty, |\mu|\ll1$ it follows $\frac{1}{2}\leq \big(1-C(|\mu|+r_\infty \tu_\infty)\big)$ and therefore we obtain $\tu_\infty \leq 2C(|\tu_0| + r_\infty + |\mu|)$. Thus,
	\[
	\tcalR_4(r,\tu;\mu)= \calO(r_\infty +|\mu| + |\tu_0|(|\mu| + r_\infty|\tu_0|)).
	\]
	Based on this, the uniform Banach contraction principle applies upon rewriting \eqref{e:bvp0} as 
	\[
	\tu_0 = \left(e^{\tA T} - \mathrm{Id}\right)^{-1}\tcalR_4(r,\tu;\mu), 
	\]
	which yields a locally unique Lipschitz continuous solution $\tu_0(r,\varphi;\mu) = \calO(r_\infty+|\mu|)$. 
	Note that together with the aforementioned, this implies the estimate $\tu_{\infty}=\calO(r_{\infty}+|\mu|)$.
	
	Substituting ${u}(t) = r(t)\tu(t)$ with initial condition $\tu_0(r,\varphi;\mu)$ for $\tu$ 
	into the first two equations of \eqref{e:cylindrical0} gives
	\begin{equation}\label{e:cylindrical01}
		\begin{aligned}
			\dot{r} &= \mu r +  \calR_5(r;\mu),\\
			\dot{\varphi} &= \omega(0) + \calR_6(r;\mu),
		\end{aligned}
	\end{equation}
	where we have divided the equation for $\varphi$ by $r$, since we look for nonzero solutions, and 
	\[
	\calR_5(r;\mu)=r\calO\big(r+|\mu|r+|\tu|(|\mu|+r|\tu|)\big)=r\calO(r_\infty + |\mu|r_\infty)=r\calO(r_\infty), 
	\quad \calR_6(r;\mu) = \calO(r_\infty+|\mu|).
	\]
	Since $\omega(0)\neq 0$, for $0\leq r,|\mu|\ll 1$ we may normalize the period to $T=2\pi$ and obtain
	\begin{equation}\label{e:absper}
		\frac{d r}{d\varphi} = \frac{\mu r +  \calR_5(r;\mu)}{\omega(0) + \calR_6(r;\mu)}
		= r\left(\frac{\mu}{\omega(0)}  +  \calR_7(r;\mu)\right),
	\end{equation}
	where $\calR_7(r;\mu) = \calO(r_\infty + |\mu|r_\infty) = \calO(r_\infty)$ follows from direct computation.
	Analogous to $\tu$ above, the boundary value problem $r(2\pi)=r(0)$ can be solved by the uniform contraction principle, which yields a locally unique and Lipschitz continuous solution $\mu(r_0)= \calO(r_0)$. Since $\varphi$ is $2\pi$-periodic, any periodic solution has a period $2\pi m$ for some $m\in \N$, and the previous computation gives a unique solution for any $m$, from which we took the one with minimal period, i.e., $m=1$.
	
	Finally, the statement of the form of periodic solutions directly proceeds with $a=r_0$ from changing back to the original time scale and coordinates. Notice that $r_{\infty}=\calO(r_{0})$ holds true since we are integrating an ODE over a bounded interval, such that the ratio between $r_\infty$ and $r_0$ is a bounded quantity, which is uniform because the vector field goes to zero when $r$ goes to zero. Therefore, and together with $\mu(r_0)= \calO(r_0)$, the previous estimate $\tu_{\infty}=\calO(r_{\infty}+|\mu|)$ becomes $\tu_{\infty}=\calO(r_0)$. Moreover, applying the supremum norm on both sides of $u=r\tu$ one gets $u_\infty = r_\infty\tu_\infty$, which is precisely of order $\calO(r_0^2)$, as we wanted to prove.
\end{proof}

While this theorem proves the existence of periodic orbits, it does not give information about their location in parameter space, scaling properties and stability; the problem is to control the leading order part of $\calR_7$ in \eqref{e:absper}, which ---in contrast to the smooth case--- turns out to be tedious. Consequently, we next aim to identify a suitable setting analogous to the center manifold reduction, and normal form transformations for a smooth vector field. In particular, we seek formulas for the analogue of the first Lyapunov coefficient from the smooth framework, whose sign determines whether the bifurcation is sub- or supercritical. 

\medskip
In order to specify a setting that allows for such an analysis, and is also relevant in applications, we will assume additional regularity away from sufficiently regular hypersurfaces $H_j$, $j=1,\ldots,n_H$, and denote $H:=\cup_{j=1}^{n_H} H_j$. We refer to these hypersurfaces as \emph{switching surfaces} and assume these intersect pairwise transversally at the equilibrium point $\u_*=0$.

\begin{hypothesis}\label{h:Ck}
	The switching surfaces $H_j$, $j=1,\ldots,n_H$, are $C^k$ smooth, $k\geq1$ and intersect transversally at $\u_*=0$. In each connected component of $U\setminus H$ the function $G$ is $C^k$ smooth and has a $C^k$ extension to the component's boundary. 
\end{hypothesis}

For simplicity, and with applications in mind, we consider only two switching surfaces, i.e., $n_H=2$. In order to facilitate the analysis, we first map $H_1, H_2$ locally onto the axes by changing coordinates. 

\begin{lemma}\label{l:cyl}
	Assume Hypotheses~\ref{h:AG} and \ref{h:Ck} and let $n_H=2$. There is a neighborhood $V\subset U$ of $\u_*$ and a diffeomorphism $\Psi$ on $V$ such that $\Psi(H_j\cap V) = \{ u_j=0\}\cap \Psi(V)$, $j=1, 2$; in particular $\Psi(\u_*)=0$.
	In subsequent cylindrical coordinates $(r,\varphi,u)\in \R_+ \times[0,2\pi)\times\R^{n-2}$ with respect to 
	the $(u_1,u_2)$-coordinate plane,
	the vector field is $C^k$ with respect to $(r,u)$.
\end{lemma}
\begin{proof}
	The smoothness of $H_j$, $j=1, 2$, and their transverse intersection allow for a smooth change of coordinates that straighten $H_1$, $H_2$ locally near $\u_*$ and maps these onto the coordinate hypersurfaces $\{u_1=0\}$, $\{u_2=0\}$, respectively. The assumed smoothness away from the switching surfaces implies the smoothness in the radial direction. 
\end{proof}

A concrete analysis of the nature of a Hopf bifurcation requires additional information on the leading order terms in $G$.  As shown next, a sufficient condition to identify the structure of the quadratic terms is Hypothesis~\ref{h:Ck} with $k=2$, where we use the notation $\ABS{\cdot}{p}$ as defined in \eqref{gen_abs_val}.

\begin{theorem}
	\label{t:abstractnormal}
	Assume Hypotheses~\ref{h:AG} and \ref{h:Ck}
	for $k\geq 2$ and let $n_H=2$. In the coordinates of Lemma~\ref{l:cyl}, the nonsmooth quadratic order terms in a component $G_j$, $j=1,\ldots, n$, of $G$ are of the form $u_\ell[u_i\Pp$, $1\leq \ell\leq n$, $i=1,2$, where $\pp,\pn\in \R$ depend on $i,\ell,j$ and are the limits of second derivatives of $G$ on the different connected components of $\R^n\setminus H$.
\end{theorem}
\begin{proof}
	Consider a coordinate quadrant and let $\widetilde{G}$ be the extension of $G$ to its closure. By assumption, we can Taylor expand  $\widetilde{G}(\u)= \frac 1 2 D^2\widetilde{G}(0)[\u,\u] + o(|\u|^2)$ since $G(0)=0$ as well as $D\widetilde{G}(0)=0$. However, for different coordinate quadrants the second order partial derivates may differ. By the form of $H$ in Lemma~\ref{l:cyl}, one-sided derivatives transverse to the coordinate axes might be distinct only for the $u_1, u_2$ axes. Hence, at $\u=0$ second order derivatives involving $u_1, u_2$ may differ, and we denote by $p_{j\ell i_\pm}$ the partial derivatives $\frac{\partial^2}{\partial u_i \partial u_\ell}G_j(0)$, $1\leq \ell\leq n$, that are one-sided with respect to $i=1,2$ as indicated by the sign. The functions $\ABS{u_i}{p_{j\ell i}}$ thus provide a closed formula for the quadratic terms of $G_j$ as claimed.
\end{proof}

Even with explicit quadratic terms in these coordinates, an analysis based on the coordinates of Lemma~\ref{l:cyl} remains a challenge. 

\begin{remark}\label{e:arrangement}
	In cylindrical coordinates relative to $E^\rmc$, cf.\ \eqref{e:cylindrical0}, the vector field is generally not smooth in the radial direction. In general, smoothness cannot be achieved by changing coordinates as this typically modifies $H$ to be nonradial. In particular, we cannot assume, without loss of generality, that the linear part in the coordinates of Lemma~\ref{l:cyl} is in block-diagonal form or in Jordan normal form as in  \eqref{e:cylindrical0}. 
\end{remark}

For exposition, we consider the planar situation $n=2$,  
where $H_1, H_2$ are the $u_1$- and $u_2$-axes, respectively. In contrast to \eqref{e:planar0} (and \eqref{e:cylindrical0}), the linear part is generally not in normal form, i.e., we have
\begin{equation}\label{e:abstractplanar}
	\begin{pmatrix}
		\dot u_1\\
		\dot u_2
	\end{pmatrix} = 
	\begin{pmatrix}
		m_1 & m_2\\
		m_3 & m_4
	\end{pmatrix}\begin{pmatrix}
		u_1\\
		u_2
	\end{pmatrix}+\begin{pmatrix}
		f_1\left( u_1, u_2 \right)\\
		f_2\left( u_1, u_2 \right)
	\end{pmatrix},
\end{equation}
where $G=(f_1,f_2)$ is nonlinear.
Based on Hypothesis~\ref{h:AG} the linear part satisfies $\mu=\frac 1 2 (m_1+m_4)$, with $\mu=0$ at the bifurcation point, and the determinant at $\mu=0$ is positive so that we get together $m_1^2+m_2m_3<0$ and $m_2 m_3<0$. 
Upon changing to polar coordinates we obtain, generally different from \eqref{e:cylindrical01}, 
\begin{equation}
	\begin{cases}
		\dot{r} = M(\varphi)r+\chi_2(\varphi)r^2 + \calO(r^3),\\
		\dot{\varphi} = W(\varphi) + \Omega_1(\varphi)r + \calO(r^2),
	\end{cases}
	\label{Sys_Polar_NoNF}
\end{equation}
where $M, \chi_2, W, \Omega_1$ are $2\pi$-periodic in $\varphi$. Abbreviating $c:=\cos{\varphi}$ and $s:=\sin{\varphi}$, we have explicitly
\begin{align*}
	M(\varphi) &= m_1c^2 + (m_2+m_3)sc + m_4s^2, \\
	W(\varphi)&= m_3c^2 + (m_4-m_1)sc - m_2s^2,
\end{align*}
where $\chi_2, \Omega_1$ are continuous but in general nonsmooth in $\varphi$ as a combination of generalized absolute value terms \eqref{gen_abs_val}.
Due to the conditions at $\mu=0$ we have $W(\varphi)\neq 0$ for any $\varphi$ so that $\dot{\varphi}\neq 0$ for $0\leq r,|\mu|\ll1$. 
This allows to rescale time in \eqref{Sys_Polar_NoNF} analogous to \eqref{e:absper} and gives

\begin{equation}\label{new_time}
	{r}' := \frac{dr}{d\varphi} = \frac{M(\varphi)r+\chi_2(\varphi)r^2}{W(\varphi) + \Omega_1(\varphi)r} + \calO(r^3)
	= \frac{M(\varphi)}{W(\varphi)}r+\left(\frac{\chi_2(\varphi)}{W(\varphi)} - \frac{M(\varphi)\Omega_1(\varphi)}{W(\varphi)^2}\right) r^2 +\mathcal{O}(r^3).
\end{equation}
Using averaging theory, as it will be discussed in detail in \S\ref{AV_S}, periodic orbits of \eqref{new_time} are generically in 1-to-1 correspondence with equilibria of the averaged form of \eqref{new_time} given by
\begin{align}\label{r_bar}
	\bar{r}' &= {\Lambda}\bar{r}+{\Sigma}\bar{r}^2+\mathcal{O}(\bar{r}^3),
\end{align}
where $\Lambda, \Sigma\in\R$ are the averages of the linear and quadratic coefficients, respectively:
\begin{align}
	\Lambda &= \frac{1}{2\pi} \int_0^{2\pi}\frac{M(\varphi)}{W(\varphi)}\D \varphi=
	\frac{m_1+m_4}{\sqrt{-4m_2m_3-(m_1-m_4)^2}}, \label{check_mu} \\
	\Sigma &= \frac{1}{2\pi} \int_0^{2\pi}\frac{\chi_2(\varphi)}{W(\varphi)}
	- \frac{M(\varphi)\Omega_1(\varphi)}{W(\varphi)^2}\D \varphi.\label{check_sigma}
\end{align}
The explicit expression in \eqref{check_mu} follows from a straightforward but tedious calculation;
note that $\Lambda\in \R$ for $0\leq |\mu|\ll1$ due to the above conditions at bifurcation.

For $\Sigma\neq 0$, equilibria of \eqref{r_bar} are $\bar{r}=0$ and $\bar{r}=-\Lambda/{\Sigma}$, which gives a branch of non-trivial periodic orbits parameterized by $\Lambda$. The direction of branching, and thus the super- and subcriticality, is determined by the sign of $\Sigma$, which therefore is a generalized first Lyapunov coefficient. 
However, this is still unsatisfying as it does not readily provide an explicit algebraic formula for $\Sigma$ in terms of the coefficients of $A(\mu)$ and $G$.

In order to further illustrate this issue, let $f_1,f_2$ be purely quadratic and built from second order modulus terms as in \eqref{e:quadnonlin}. In this case we explicitly have
\begin{align}
	\chi_2(\varphi) &=  c\abs{c}(a_{11}c+b_{11}s) + c\abs{s}(a_{12}c+b_{12}s) + s\abs{c}(a_{21}c+b_{21}s) + s\abs{s}(a_{22}c+b_{22}s), \label{chi}\\
	\Omega_1(\varphi) &= -\Big[c\abs{c}(a_{11}s-b_{11}c) + c\abs{s}(a_{12}s-b_{12}c) + s\abs{c}(a_{21}s-b_{21}c) + s\abs{s}(a_{22}s-b_{22}c) \Big],\label{Omega}
\end{align}
which are continuous but not differentiable due to the terms involving $|c|,|s|$. 

Clearly, the building blocks of the integrals in \eqref{check_sigma} are rational trigonometric functions with denominator $W$ of degree $2$ and numerators of degree $3$ and $5$. However, explicit formulas based on this appear difficult to obtain, so that we instead change to linear normal form as discussed in \S\ref{Gen_linear_part}, with the caveat that the nonlinear terms are in general not smooth in the radius.

Indeed, in the normal form case $m_1=\mu,\, m_2=-\omega,\, m_3=\omega,\, m_4=\mu$, with $\omega >0$, the situation becomes manageable: in \eqref{Sys_Polar_NoNF} we have constant $M(\varphi)= \mu$ and $W(\varphi)=\omega(\mu)$, and we will show below that then $\Sigma=\frac{2}{3\pi\omega}\sigma_{_\#}$, with $\sigma_{_\#}$ as defined in \S\ref{s:intro}. 
Therefore, until \S\ref{Gen_linear_part} we will assume that the linear part is in normal form in the coordinates of Lemma~\ref{l:cyl}, which also occurs in applications as mentioned in \S\ref{s:intro}.

\section{Planar normal form case with absolute values}
\label{Planar_Section}

In this section we discuss two approaches to {prove}
existence and bifurcation of periodic orbits in our mildly nonsmooth setting. First, we provide details for the aforementioned approach by averaging, and second discuss a direct approach that provides a detailed unfolding by Lyapunov-Schmidt reduction, and that can also be used in some nongeneric cases.

While we focus here on the planar case, both methods readily generalize to higher dimensional settings. For averaging one needs normal hyperbolicity in general, and for the direct approach we present higher dimensional cases in upcoming sections. 
Without change in the leading order result, for simplicity we fix the imaginary part $\omega\neq 0$ independent of $\mu$.

To simplify the exposition in this section, we assume the linear part is in normal form and the nonsmooth terms are of second order modulus type, i.e., with absolute value $|\cdot| = [\cdot]_{-1}^1$. The general case will be discussed in \S\ref{s:general}.
With the linear part in normal form and including smooth quadratic and cubic terms we thus consider the form of \eqref{e:abstract} given by, cf.\ {\eqref{e:planar0}},
\begin{align}
	\begin{cases}
		\dot{v} &= \mu v - \omega \w + f\left( v, \w \right) + f_q\left( v, \w\right) + f_c\left( v, \w \right), \\
		\dot{\w} &= \omega v + \mu \w + g\left( v, \w \right) + g_q\left( v, \w \right) + g_c\left( v, \w \right),
	\end{cases}
	\label{General2D_AV}
\end{align}
where $f,g$ are as in \eqref{e:quadnonlin}, and 
\begin{align*}
	f_q\left( v, \w\right) &= a_1 v^2 + a_2 v\w + a_3 \w^2,
	&f_c\left( v, \w \right) &= c_{a1} v^3 + c_{a2} v\w^2 + c_{a3} v^2\w + c_{a4} \w^3,\\
	g_q\left( v, \w\right) &= b_1 v^2 + b_2 v\w + b_3 \w^2,
	&g_c\left( v, \w\right) &= c_{b1} v^3 + c_{b2} v\w^2 + c_{b3} v^2\w + c_{b4} \w^3,
\end{align*}
and $\mu, \omega \in \R$ with $\omega\neq 0$, and $a_{ij}, b_{ij}, a_k, b_k, c_{ah}, c_{bh},$ $\forall i,j\in \{1,2\}, \forall k\in\{1,2,3\}, \forall h\in\{1,2,3,4\}$ are real constants, all viewed as parameters. 

\subsection{Averaging}\label{AV_S}

We next show how to apply averaging theory to \eqref{General2D_AV} in polar coordinates. In addition to $\sigma_{_\#}, \sigma_2$ from \eqref{sigma1}, \eqref{sigma2}, the following expressions appear as normal form coefficients:
\begin{align*}
	S_q :=\, & a_1a_2 +a_2a_3 - b_1b_2 - b_2b_3 -2a_1b_1 + 2a_3b_3,\nonumber\\
	S_c :=\, & 3c_{a1} + c_{a2} + c_{b3} + 3c_{b4}.\nonumber
\end{align*}
Notice that $\sigma_{_\#}, \sigma_2$ depend only on $f, g$, i.e., the nonsmooth terms, while $S_q$ depends on the smooth quadratic terms $f_q, g_q$; and $S_c$ on the cubic ones $f_c,g_c$.

\begin{theorem}
	\label{t_averaging}
	For $0<|\mu| \ll1$ periodic solutions to \eqref{General2D_AV} are locally in 1-to-1 correspondence with equilibria of the averaged normal form in polar coordinates $v=r\cos{\varphi},\, w=r\sin{\varphi}$ of  \eqref{General2D_AV} given by 
	\begin{equation}
		{\bar{r}}' = \frac{\mu}{\omega} \bar{r} + \frac{2}{3\pi\omega}\sigma_{_\#} \bar{r}^2 + \left( \frac{1}{8\omega^2}S_q + \frac{1}{8\omega}S_c + \frac{1}{2\pi\omega^2}\sigma_2 \right) \bar{r}^3 +  \mathcal{O}\left(\bar{r}^4+\sigma_{_\#}\bar{r}^3+\abs{\mu}\bar{r}^2\right).
		\label{NormalFormAV}
	\end{equation}
\end{theorem}
Remark that in accordance with the smooth Hopf bifurcation, the quadratic term in $\bar r'$ vanishes for vanishing nonsmooth terms $f=g=0$, so that the leading order nonlinear term in the normal form is cubic. Before giving the proof we note and discuss an important corollary. For this recall the pitchfork bifurcation of \eqref{Dpitchfork} which is degenerate in that the bifurcating branch is nonsmooth.

\begin{corollary}
	\label{c_averaging}
	If $\sigma_{_\#}\neq 0$, then at $\mu=0$ \eqref{NormalFormAV} undergoes a degenerate pitchfork bifurcation in $\mu$, where non-trivial equilibria are of the form 
	\begin{equation}\label{periodic_orbit}
		r_0(\mu) = -\frac{3\pi}{2\sigma_{_\#}}\mu + \mathcal{O}\left(\mu^2\right).
	\end{equation}
	In this case,  \eqref{General2D_AV} undergoes a degenerate Hopf bifurcation in the sense that for $0<|\mu| \ll1$ periodic solutions to \eqref{General2D_AV} are locally in 1-to-1 correspondence with $r_0(\mu)$, which is also the expansion of the radial component of the periodic solutions. 
	In particular, this Hopf bifurcation is subcritical if {$\sgn(\sigma_{_\#})>0$} and supercritical if {$\sgn(\sigma_{_\#})<0$}.
	Moreover, the bifurcating periodic orbits of \eqref{General2D_AV} are of the same stability as the {corresponding} equilibria in \eqref{NormalFormAV}.
\end{corollary}

\begin{proof} (Corollary \ref{c_averaging})
	The bifurcation statement follows directly from Theorem \ref{t_averaging} and the statement about stability follows from \cite[Thm.\ 6.3.3]{Averaging}. Since $r_0\geq 0$ we must have $\frac{\mu}{\sigma_{_\#}}\geq 0$. Hence, the sign of $\sigma_{_\#}$ determines the criticality of the bifurcation.
\end{proof}

The radial components $r(\varphi;\mu)$ of the periodic orbits are in general not constant in $\varphi$, but this dependence is of order $\mu^2$. We thus consider \eqref{periodic_orbit} as the leading order amplitude of the periodic solutions.

\begin{remark}\label{1st_Lyap}
	Since the criticality of the Hopf bifurcation is given by the sign of $\sigma_{_\#}$, it is an analogue of the first Lyapunov coefficient in this nonsmooth case. For the smooth case $f=g=0$, where $\sigma_{_\#}=\sigma_2=0$, the classical first Lyapunov coefficient is $\sigma_s:=\frac{1}{8\omega}S_q + \frac{1}{8}S_c$. In \S\ref{s:smooth} we show that there is no canonical way to infer the sign of $\sigma_{_\#}$ from smoothing a priori.
\end{remark}

\begin{remark}\label{2nd_Lyap}
	In case $\sigma_{_\#}=0$ but nonzero cubic coefficient in \eqref{NormalFormAV}, the bifurcating branch is a quadratic function of $\mu$ to leading order. This readily gives an analogue of the second Lyapunov coefficient in this nonsmooth case. An explicit statement in absence of smooth terms is given in Theorem~\ref{2ndPart} below. Notably, in the smooth case, vanishing first Lyapunov coefficient, but nonzero second Lyapunov coefficient yields a quartic bifurcation equation. Hence, the scaling laws {for the radius} are $\mu^{1/j}$ with $j=1,2$ in the nonsmooth {case} and $j=2,4$ in the smooth case, respectively.
\end{remark}

Next we give the proof of Theorem \ref{t_averaging}.

\begin{proof} (Theorem \ref{t_averaging})
	Taking polar coordinates $(v,w)=(r\cos{\varphi},r\sin{\varphi})$ system \eqref{General2D_AV}, cf.\ \eqref{Sys_Polar_NoNF}, becomes
	\begin{align}
		\begin{cases}
			\dot{r} &= r\mu + r^2\chi_2(\varphi) + r^3\chi_3(\varphi), \\
			\dot{\varphi} &= \omega + r\Omega_1(\varphi) + r^2\Omega_2(\varphi),
		\end{cases}
		\label{System2DpolarNF}
	\end{align}
	where $\chi_2(\varphi)$ and $\Omega_1(\varphi)$ are as in \eqref{chi} and \eqref{Omega}, respectively, but adding now the contributions of the smooth quadratic terms of $f_q, g_q$:
	\begin{align*}
		\chi_2(\varphi) =&  c\abs{c}(a_{11}c+b_{11}s) + c\abs{s}(a_{12}c+b_{12}s) + s\abs{c}(a_{21}c+b_{21}s) + s\abs{s}(a_{22}c+b_{22}s) \\
		&+ (a_1-b_2-a_3)c^3 + (b_1+a_2-b_3)sc^2 + (b_2+a_3)c + b_3s, \\
		\Omega_1(\varphi) =& -\Big[c\abs{c}(a_{11}s-b_{11}c) + c\abs{s}(a_{12}s-b_{12}c) + s\abs{c}(a_{21}s-b_{21}c) + s\abs{s}(a_{22}s-b_{22}c) \Big] \\
		&+ (b_1+a_2-b_3)c^3 + (-a_1+b_2+a_3)sc^2 + (-a_2+b_3)c - a_3s,
	\end{align*}
	and $\chi_3(\varphi)$ and $\Omega_2(\varphi)$ are smooth functions of $\varphi$ and the coefficients of $f_c, g_c$:
	\begin{align*}
		\chi_3(\varphi) =& (c_{a1}-c_{a2}-c_{b3}+c_{b4})c^4 + (c_{a3}-c_{a4}+c_{b1}-c_{b2})sc^3 \\
		&+ (c_{a2}+c_{b3}-c_{b4})c^2 + (c_{a4}+c_{b2})sc + c_{b4}s^2,\\
		\Omega_2(\varphi) =& (c_{a3}-c_{a4}+c_{b1}-c_{b2})c^4 + (c_{a2}-c_{a1}+c_{b3}-c_{b4})sc^3 - c_{a3}c^2 \\
		&+ (c_{b4}-c_{a2})sc - c_{a4}s^2 + (c_{b2}+c_{a4})c^2.
	\end{align*}
	To simplify the notation we write, as before, $c:=\cos{\varphi}$, $s:=\sin{\varphi}$.
	
	Analogous to \eqref{new_time}, we change parametrization such that the return time to $\varphi=0$ is equal for all orbits starting on this half-axis with initial radius $r_0>0$ to get
	\begin{align*}
		{r}':= \dv{r}{\varphi} &= \frac{r\mu + r^2\chi_2(\varphi) + r^3\chi_3(\varphi)}{\omega + r\Omega_1(\varphi) + r^2\Omega_2(\varphi)}.
	\end{align*}
	Expanding the right-hand side of $r'$ in small $r$ and $\mu$ gives
	\begin{equation}
		{r}' =
		\frac{\mu}{\omega} r + \frac{\chi_2}{\omega}  r^2 + \left( \frac{\chi_3}{\omega} - \frac{\chi_2\Omega_1}{\omega^2} \right) r^3 + \mathcal{O}\left( r^4+\abs{\mu}r^2 \right).
		\label{DoNormalForm_r}
	\end{equation}
	In order to follow the method of averaging (e.g., \cite{Guckenheimer,Averaging}), we write $r=\epsilon x$ and $\mu=\epsilon m$ for $0<\epsilon\ll 1$, such that \eqref{DoNormalForm_r} in terms of $x$ and $m$ becomes
	\begin{align}
		x' &= \epsilon \left( \frac{m}{\omega} x + \frac{\chi_2}{\omega}  x^2 \right) + \epsilon^2\left( \frac{\chi_3}{\omega} - \frac{\chi_2\Omega_1}{\omega^2} \right) x^3 + \epsilon^2\mathcal{O}\left( \epsilon x^4+\abs{m}x^2 \right).
		\label{x_form_for_av}
	\end{align}
	Following \cite{Averaging}, there is a near-identity transformation which 
	maps solutions of the truncated averaged equation
	\begin{equation}
		y' = \epsilon\bar{f}(y) + \epsilon^2\bar{f}_{2}(y)+ \mathcal{O}(\epsilon^3)
		\label{truncated_av}
	\end{equation}
	to solutions of \eqref{x_form_for_av}, where its detailed derivation is given in Appendix~\ref{NIT}, as well as an explanation of the computation of the following functions:
	\begin{equation}
		\label{averaging_integrals}
		\begin{aligned}
			\bar{f}(y) &= 
			\frac{m}{\omega}y +\frac{2}{3\pi\omega}\sigma_{_\#}y^2, \\
			\bar{f}_{2}(y) &= 
			\left( \frac{1}{8\omega^2}S_q + \frac{1}{8\omega}S_c + \frac{1}{2\pi\omega^2}\sigma_2 \right)y^3 + 
			\calO\left( \sigma_{_\#}y^3 + \abs{m}y^2 \right).
		\end{aligned}
	\end{equation}
	
	We obtain the averaged equation \eqref{NormalFormAV} from \eqref{DoNormalForm_r} by the change of coordinates $y=\frac{\bar{r}}{\epsilon}$ and $m=\frac{\mu}{\epsilon}$ applied to \eqref{truncated_av} {with \eqref{averaging_integrals}}; this becomes \eqref{NormalFormAV} since all terms involving $\epsilon$ cancel out.
	
	Finally, from \cite[Thm.\ 6.3.2]{Averaging} the existence of a periodic orbit in the averaged system implies the existence of a periodic orbit in the original system.
\end{proof}

\subsection{Smoothing and the first Lyapunov coefficient}\label{s:smooth}
From Remarks \ref{1st_Lyap} and \ref{2nd_Lyap} on the first and second Lyapunov coefficients, it is natural to ask in what way the nonsmooth first Lyapunov coefficient 
\[
\sigma_{_\#} =  2a_{11}+a_{12}+b_{21}+2b_{22}
\]
from \eqref{sigma1} differs from the first Lyapunov coefficient of a smoothed version of \eqref{e:planar0}. 

More specifically, the question is whether one can smooth the vector field in such a way that the sign of the {resulting} first Lyapunov coefficient is the same as that of the nonsmooth one, $\sigma_{_\#}$, in all cases. We shall prove that this is not possible \emph{without using the formula for $\sigma_{_\#}$} ---with the help of this formula we can find suitable smoothing.

\medskip
Clearly, nonconvex approximations of the absolute value $|\cdot|$ can change criticality compared to the nonsmooth case (see Figure \ref{3_Cases}). More generally, we have the following.

\begin{lemma}
	\label{Convex_Approx}
	For any $f,g$ with a sign change in the coefficients $a_{11}, a_{12}, b_{21}, b_{22}$, there are smooth approximations $f_\eps, g_\eps$ with $(f_\eps,g_\eps)\to (f,g)$ in $L^\infty$ such that the criticality of the smoothed Hopf bifurcation is opposite that of the nonsmooth case. Moreover,  $f_\eps, g_\eps$ can be chosen as symmetric smooth convex approximations of the absolute values in $f,g$.
\end{lemma}
\begin{proof}
	Without loss of generality, we consider system \eqref{e:planar0}.  For given $f,g$ we can choose a smooth approximation of $|\cdot|$ in the terms with coefficients $a_{11}, a_{12}, b_{21}, b_{22}$ that have quadratic terms with positive coefficients of the form $\eps^{-1} \tilde a_{11}, \eps^{-1} \tilde a_{12}, \eps^{-1} \tilde b_{21}, \eps^{-1} \tilde b_{22}$, respectively. Then the (smooth) first Lyapunov coefficient reads
	\[
	\sigma_{s,\eps}:= \eps^{-1}\left(3\tilde a_{11} a_{11}+\tilde a_{12} a_{12}+\tilde b_{21} b_{21}+3\tilde b_{22} b_{22} \right),
	\]
	which is the same as $S_c$ in \S\ref{AV_S} when replacing accordingly coefficients of $f,g$ and $f_c,g_c$, respectively.
	
	Suppose now $\sigma_{_\#}<0$. In this case, the sign change within $(a_{11}, a_{12}, b_{21}, b_{22})$ allows to choose $(\tilde a_{11}, \tilde a_{12}, \tilde b_{21}, \tilde b_{22})>0$ such that $\sigma_{s,\eps}>0$. Likewise for $\sigma_{_\#}>0$ we can arrange $\sigma_{s,\eps}<0$.
\end{proof}

\begin{remark}
	If all of $a_{11}, a_{12}, b_{21}, b_{22}$ have the same sign, then any convex smoothing of the absolute value with nonzero quadratic terms will yield a first Lyapunov coefficient of the same sign as $\sigma_{_\#}\neq 0$.
	Moreover, having derived the formula for $\sigma_{_\#}$, we can ---a posteriori--- identify a smoothing that preserves {the} criticality for all $f,g$. With the notation of Lemma \ref{Convex_Approx} this is
	$\tilde a_{11} =  \tilde b_{22} =  2/3, \tilde a_{12} = \tilde b_{21} = 1$.
\end{remark}

\begin{lemma}
	There is no smooth approximation of the absolute value function with nonzero quadratic term that preserves the criticality of the nonsmooth case for all $f,g$.
\end{lemma}
\begin{proof}
	In contrast to Lemma \ref{Convex_Approx}, here all absolute value terms in $f,g$ are approximated in the same way so that in the notation of the proof of Lemma \ref{Convex_Approx} we have $\tilde a_{11}=\tilde a_{12}= \tilde b_{21}= \tilde b_{22} > 0$. Without loss of generality we can assume {these coefficients are all equal $1$} due to the prefactor $\eps^{-1}$, so that the first Lyapunov coefficient is  
	\[
	\sigma_{s,\eps}= \eps^{-1}\left(3 a_{11}+ a_{12}+ b_{21}+3 b_{22} \right),
	\]
	and we readily find examples of $(a_{11}, a_{12}, b_{21}, b_{22})$ such that the signs of $\sigma_{_\#}$ and $\sigma_{s,\eps}$ differ.
\end{proof}

The discrepancies shown here for the absolute value function readily carry over to the generalized absolute value function \eqref{gen_abs_val}.

\subsection{Direct method}\label{s:direct}
In Theorem \ref{t_averaging}, the conclusion for \eqref{General2D_AV} does not cover the bifurcation point $\mu=0$ so that we cannot infer uniqueness of the branch of bifurcating periodic orbits directly. 
In order to directly include $\mu=0$ in the bifurcation analysis and to facilitate the upcoming generalizations, we present a `direct' method for a general (possibly) nonsmooth planar system. This does not rely on the existence of an invariant manifold as in Proposition~\ref{prop:inv_man} or results from averaging theory.

The basic result is the following bifurcation of periodic solutions for a radial equation with quadratic nonlinear terms, which cannot stem from a smooth planar vector field, but occurs in our setting as in \eqref{Sys_Polar_NoNF}. 

\begin{proposition}
	\label{Thm_Gen}
	Consider a planar system in polar coordinates $(r,\varphi)\in \R_+ \times [0,2\pi)$ periodic in $\varphi$ {of the form}
	\begin{align}
		\begin{cases}
			\dot{r} &= r\mu + r^2\chi_2(\varphi), \\
			\dot{\varphi} &= \omega + r\Omega_1(\varphi),
		\end{cases}
		\label{System2Dpolar}
	\end{align}
	where $\mu\in\R$, $\omega\neq 0$ and continuous $\chi_2(\varphi),\Omega_1(\varphi)$ with minimal period $2\pi$.
	
	If $\int_0^{2\pi}\chi_2(\varphi)\D \varphi\neq 0$, then a locally unique branch of periodic orbits bifurcates at $\mu=0$. These orbits have period $2\pi + \calO(\mu)$ and constant radius satisfying
	\begin{equation} \label{General_Result}
		r_0=\frac{-2\pi}{\int_0^{2\pi}\chi_2(\varphi)\D\varphi}\mu+\mathcal{O}\left(\mu^2\right).
	\end{equation}
	In particular, since $r_0\geq 0$, the criticality of the bifurcation is determined by the sign of $\int_0^{2\pi}\chi_2(\varphi)\D \varphi$.
\end{proposition}

For later reference we present a rather detailed proof.

\begin{proof}
	As in the proof of Theorem \ref{t_averaging}, for small $r$ the radius satisfies
	\begin{align}
		{r}':= & \frac{r\mu + r^2\chi_2(\varphi)}{\omega + r\Omega_1(\varphi)} =: \Psi(r,\varphi).
		\label{System2Dparam}\end{align}
	
	We fix the initial time at $\varphi_0=0$ and for any initial $r(0)=r_0$, a unique local solution is guaranteed from the Picard-Lindel\"of theorem with continuous time dependence, e.g., \cite{Hartman}. 
	This also guarantees existence on any given time interval for sufficiently small $|\mu|, r_0$. Moreover, the solution $r(\varphi;r_0)$ can be Taylor expanded with respect to $r_0$ due to the smoothness of $\Psi(r,\varphi)$ in $r$ and continuity in the time component using the uniform contraction principle for the derivatives, cf.\ \cite{Hartman}.
	
	On the one hand, we may thus expand $r(\varphi)= r(\varphi; r_0)$ as
	\begin{equation*}
		r(\varphi)=\alpha_1(\varphi)r_0 + \alpha_2(\varphi)r_0^2 + \mathcal{O}\left(r_0^3\right),
	\end{equation*}
	and differentiate with respect to $\varphi$, 
	\begin{equation}
		r'(\varphi)=\alpha_1'(\varphi)r_0 + \alpha_2'(\varphi)r_0^2 + \mathcal{O}\left(r_0^3\right),
		\label{Expansion_r'}
	\end{equation}
	where $\alpha_1(0)=1$ and $\alpha_2(0)=0$ since $r(0)=r_0$.
	
	On the other hand, we Taylor expand $\Psi(r,\varphi)$ in $r=0$ from \eqref{System2Dparam}, using $\Psi(0,\varphi)=0$, as
	\begin{align}
		r' &= \Psi(r,\varphi) = \Psi(0,\varphi) + \partial_r\Psi(0,\varphi)r + \frac{1}{2}\partial^2_r\Psi(0,\varphi) r^2 + \mathcal{O}\left(r^3\right)
		= k_1 r + k_2 r^2 + \mathcal{O}\left(r^3\right),
		\label{Expansion2_r'}
	\end{align}
	where we denote $\partial^i_r\Psi(0,\varphi)=\frac{\partial^i\Psi(r,\varphi)}{\partial r^i}\big\rvert_{r=0}$, $i\in\mathbb{N}$, and set
	\begin{equation}\label{ks}
		k_1 := \partial_r\Psi(0,\varphi) = \frac{\mu}{\omega}, \hspace*{4mm} k_2(\varphi) := \frac{1}{2}\partial^2_r\Psi(0,\varphi) = \frac{\omega\chi_2(\varphi) - \mu\Omega_1(\varphi)}{\omega^2}.
	\end{equation}
	Matching the coefficients of $r_0$ and $r_0^2$ in \eqref{Expansion_r'} and \eqref{Expansion2_r'} gives the ODEs $\alpha_1' = k_1 \alpha_1$ and $\alpha_2' = k_1 \alpha_2 +k_2\alpha_1^2$.
	The solutions with $\alpha_1(0)=1$ and $\alpha_2(0)=0$ read
	\begin{align*}
		\alpha_1(\varphi) &= e^{k_1 \varphi}, 
		&\alpha_2(\varphi) &= 
		\int_0^\varphi e^{k_1(\varphi+s)}k_2(s)\D s.
	\end{align*}
	Periodic orbits necessarily have period $2\pi m$ for some $m\in\N$, which yields the condition
	\begin{equation}\label{per_orb_r}
		0 = r(2\pi m)-r(0) = \int_0^{2\pi m} r' \mathrm{d}\varphi = r_0\int_0^{2\pi m} \alpha_1'(\varphi)\D \varphi + r_0^2\int_0^{2\pi m} \alpha_2'(\varphi)\D \varphi + \mathcal{O}\left(r_0^3\right).
	\end{equation}
	Using the series expansion of $e^{2\pi m k_1}$ in $\mu=0$ we have
	$$ \int_0^{2\pi m} \alpha_1'(\varphi)\D \varphi = \alpha_1(2\pi m)-\alpha_1(0) = e^{2\pi m k_1}-1 = 2\pi m k_1 + \mathcal{O}\left(\mu^2\right), $$
	and similarly,
	\begin{align*}
		\int_0^{2\pi m} \alpha_2'(\varphi)\D \varphi &= \alpha_2(2\pi m)-\alpha_2(0) =  e^{2\pi m k_1}\int_0^{2\pi m} e^{k_1\varphi}k_2(\varphi)\D \varphi - 0 \\
		&= (1+2\pi m k_1)\int_0^{2\pi m}k_2(\varphi)\D\varphi + k_1\int_0^{2\pi m}\varphi k_2(\varphi)\D\varphi + \mathcal{O}\left(\mu^2\right).
	\end{align*}
	For non-trivial periodic orbits, $r_0\neq 0$, we divide \eqref{per_orb_r} by $r_0$, which provides the bifurcation equation
	\begin{equation*}
		0 = 2\pi m k_1 + r_0\left( (1+2\pi m k_1)\int_0^{2\pi m}k_2(\varphi)\D\varphi + k_1\int_0^{2\pi m}\varphi k_2(\varphi)\D\varphi \right) + \mathcal{O}\left(\mu^2\right),
	\end{equation*}
	where the factor of $r_0$ is nonzero at $\mu=0$ by assumption. Hence, the implicit function theorem applies and gives a unique solution. Since the solution for $m=1$ is a solution for any $m$, this is the unique periodic solution. Solving the bifurcation equation for $m=1$ yields
	\begin{equation*}
		r_0 = \frac{- 2\pi \mu}{(\omega+2\pi  \mu)\int_0^{2\pi }k_2(\varphi)\D\varphi + \mu\int_0^{2\pi m}\varphi k_2(\varphi)\D\varphi} + \mathcal{O}\left(\mu^2\right),
	\end{equation*}
	whose expansion in $\mu=0$ gives the claimed \eqref{General_Result} and in particular the direction of branching.
	
	Finally, the exchange of stability between the trivial equilibrium and the periodic orbit follows from the monotonicity of the $1$-dimensional Poincar\'e Map on an interval that contains $r=0$ and $r=r_0(\mu)$ by uniqueness of the periodic orbit. 
\end{proof}

We next note that higher order perturbations do not change the result to leading order.

\begin{corollary}\label{hot2D}
	The statement of Proposition \ref{Thm_Gen} holds for a planar system in polar coordinates $(r,\varphi)\in \R_+ \times [0,2\pi)$ periodic in $\varphi$, of the form
	\begin{align}
		\begin{cases}
			\dot{r} &= r\mu + r^2\chi_2(\varphi) + r^3\chi_3(r,\varphi), \\
			\dot{\varphi} &= \omega + r\Omega_1(\varphi) + r^2\Omega_2(r,\varphi),
		\end{cases}
		\label{System_Polar_General}
	\end{align}
	where $\mu\in\R$, $\omega\neq 0$ and $\chi_{j+1}$, $\Omega_j$, $j=1,2$, are continuous in their variables. 
\end{corollary}

Note that system \eqref{System_Polar_General} is a generalization of \eqref{System2DpolarNF} in which $\chi_3$ and $\Omega_2$ depend now on $r$.

\begin{proof}
	Following the proof of Proposition \ref{Thm_Gen} we write system \eqref{System_Polar_General} analogous to \eqref{System2Dparam} with $$\Psi(r,\varphi) = \frac{r\mu + r^2\chi_2(\varphi) + r^3\chi_3(r,\varphi)}{\omega + r\Omega_1(\varphi) + r^2\Omega_2(r,\varphi)}.$$
	Upon subtracting the leading order part of \eqref{Expansion2_r'}, a direct computation produces a remainder term of order $\calO(r^3)$, which leads to the claimed result.
\end{proof}

Next we show how these results can be directly used to determine the Hopf bifurcation and its super- or subcriticality. 
Starting with the simplest model, we return to system \eqref{General2D_AV} with $f_q, g_q, f_c, g_c \equiv 0$, i.e., \eqref{e:planar0}. 
Recall $\sigma_{_\#}=2a_{11}+a_{12}+b_{21}+2b_{22}$ from \eqref{sigma1} was identified as determining the criticality in Corollary \ref{c_averaging}. With the direct method we obtain the following.

\begin{theorem}
	\label{1stPart}
	If $\sigma_{_\#}\neq 0$, then there exists an interval $I$ around $\mu=0$ such that at $\mu=0$ system \eqref{e:planar0} with $f,g$ from \eqref{e:quadnonlin} undergoes a degenerate Hopf bifurcation in $\mu$ where the leading order amplitudes of the locally unique periodic orbits is given by \eqref{periodic_orbit}. 
	In particular, the unique bifurcating branch of periodic solutions emerges subcritically if {$\sgn(\sigma_{_\#})>0$} and supercritically if {$\sgn(\sigma_{_\#})<0$}. Moreover, the bifurcating periodic orbits have exchanged stability with the equilibrium at $r=0$, i.e., are stable if they exist for $\mu>0$ and unstable if this is for $\mu<0$.
\end{theorem}

\begin{proof}
	Taking polar coordinates $(v,\w)=(r\cos{\varphi},r\sin{\varphi})$ for system \eqref{e:planar0} gives \eqref{System2Dpolar}, where $\chi_2(\varphi)$ and $\Omega_1(\varphi)$ are as in \eqref{chi} and \eqref{Omega}, respectively.
	Applying Proposition \ref{Thm_Gen}
	and computing the integral of $\chi_2$ 
	in each quadrant as in the proof of Theorem~\ref{t_averaging}, we obtain 
	\eqref{periodic_orbit}, and the criticality follows as in Corollary~\ref{c_averaging}.
	
	Finally, the exchange of stability is due to the monotonicity of the $1$-dimensional Poincar\'e Map.
\end{proof}

We next note that, proceeding as for Corollary \ref{hot2D}, the coefficients from the quadratic and cubic terms do not affect the bifurcation to leading order.

\begin{corollary}
	\label{thm_cubic}
	If $\sigma_{_\#}\neq 0$, then the statement of Theorem \ref{1stPart} holds for the more general system \eqref{General2D_AV}. In particular, $f_q$, $g_q$, $f_c$, $g_c$ do not affect $\sigma_{_\#}$ and the leading order bifurcation.
\end{corollary}

Having investigated $\sigma_{_\#}\neq 0$, we next consider the degenerate case $\sigma_{_\#}= 0$. For that, recall  Remark~\ref{2nd_Lyap} and {$\sigma_2$ from} \eqref{sigma2}.

\begin{theorem}
	\label{2ndPart}
	If $\sigma_{_\#}= 0$ and $\sigma_2\neq 0$, then there exists an interval $I$ around $\mu=0$ such that at $\mu=0$ system \eqref{General2D_AV} undergoes a degenerate Hopf bifurcation in $\mu$ where the leading order amplitude of the locally unique periodic orbit is given by 
	\begin{equation}\label{r_2nd_HB}
		r_0=\sqrt{-\frac{2\pi\omega}{\sigma_2}\mu}+\mathcal{O}\left(\mu\right).
	\end{equation}
	Particularly, the unique bifurcating branch of periodic solutions results subcritically if 
	$\sgn(\omega\sigma_2)>0$ and supercritically if $\sgn(\omega\sigma_2)<0$. Moreover, the bifurcating periodic orbits have exchanged stability with the equilibrium at $r=0$, i.e., are stable if they exist for $\mu>0$ and unstable if this is for $\mu<0$.
\end{theorem}

\begin{proof}
	Proceeding as before, we write \eqref{General2D_AV} in polar coordinates $(v,\w)=(r\cos{\varphi},r\sin{\varphi})$ and change the time parametrization to obtain the form \eqref{System2Dparam} for the radial equation. 
	
	On the one hand, we expand the solution $r(\varphi)=r(\varphi;r_0)$ with $r(0)=r_0$ and differentiate it with respect to $\varphi$ as
	\begin{equation}
		r'(\varphi)=\alpha_1'(\varphi)r_0 + \alpha_2'(\varphi)r_0^2 + \alpha_3'(\varphi)r_0^3 + \mathcal{O}\left(r_0^4\right),
		\label{2Expansion_r}
	\end{equation}
	where $\alpha_1(0)=1$ and $\alpha_2(0)=\alpha_3(0)=0$.
	
	On the other hand, we compute the Taylor expansion of $r'$, from \eqref{System2Dparam}, up to third order in $r=0$ as
	\begin{equation*}
		r' = \Psi(r,\varphi) = k_1 r + k_2 r^2 + k_3 r^3 + \mathcal{O}\left(r^4\right),
	\end{equation*}
	where we use $\Psi(0,\varphi)=0$ and the notation \eqref{ks} as well as
	\begin{equation*}
		k_3(\varphi) := \frac{1}{3!}\partial^3_r\Psi(0,\varphi)= \frac{-\omega\chi_2(\varphi)\Omega_1(\varphi) + \mu\Omega_1(\varphi)^2}{\omega^3}.
	\end{equation*}
	Analogous to the proof of Proposition \ref{Thm_Gen}, using \eqref{2Expansion_r} and its derivate, and comparing coefficients, we obtain the ODEs
	\begin{align*}
		\alpha_1' &= k_1\alpha_1,
		& \alpha_2' &= k_1\alpha_2 + k_2\alpha_1^2,
		& \alpha_3' &= k_1\alpha_3 + 2k_2\alpha_1\alpha_2 + k_3\alpha_1^3.
	\end{align*}
	We solve these by variation of constants using $\alpha_1(0)=1$ and $\alpha_2(0)=\alpha_3(0)=0$ as
	\begin{align*}
		\alpha_1(\varphi) &= e^{k_1 \varphi}, \\
		\alpha_2(\varphi) &= \int_0^\varphi e^{k_1(\varphi+s)}k_2(s)\D s, \\
		\alpha_3(\varphi) &= e^{k_1\varphi}\left[2\int_0^\varphi k_2(s)\alpha_2(s)\D s + \int_0^\varphi e^{2k_1s}k_3(s)\D s \right].
	\end{align*}
	
	Periodic orbits are the solutions with $r_0\neq 0$ of
	\begin{equation}\label{e:2ndper}
		0 = r(2\pi)-r(0) = r_0\int_0^{2\pi}\alpha_1'(\varphi)\D\varphi + r_0^2\int_0^{2\pi}\alpha_2'(\varphi)\D\varphi + r_0^3\int_0^{2\pi}\alpha_3'(\varphi)\D\varphi + \mathcal{O}\left(r_0^4\right).
	\end{equation}
	Straightforward computations give $\alpha_1(2\pi)-\alpha_1(0) = \frac{2\pi}{\omega}\mu + \mathcal{O}\left(\mu^2\right)$. Furthermore, we obtain 	
	$\alpha_j(2\pi)-\alpha_j(0) = \Gamma_j + \mathcal{O}\left(\mu\right)$, $j=2,3$, 
	where $\cB = \frac{4}{3\omega}\sigma_{_\#}$ and $\cA = \frac{1}{\omega^2}\sigma_2$.
	Substitution into \eqref{e:2ndper} for periodic orbits and dividing out $r_0\neq 0$, yield the bifurcation equation
	\begin{equation}\label{Thm2_eq_r0}
		0 = \frac{2\pi}{\omega}\mu + \frac{4}{3\omega}\sigma_{_\#}r_0 + \frac{1}{\omega^2}\sigma_2 r_0^2 + \mathcal{O}\left(\mu^2 + \mu r_0 + r_0^3\right).
	\end{equation}
	Since $\sigma_{_\#}=0$, this equation becomes 
	$0 = \frac{2\pi}{\omega}\mu + \frac{1}{\omega^2}\sigma_2 r_0^2 + \mathcal{O}\left(\mu^2 + \mu r_0 + r_0^3\right)$.
	Here the implicit function theorem applies a priori to provide a unique branch $\mu(r_0)$ with
	\begin{equation}\label{e:bifeq2}
		\mu = 
		-\frac{\sigma_2}{2\pi\omega}r_0^2 + \mathcal{O}\left(r_0^3\right).
	\end{equation} 
	Solving this for $r_0$ provides \eqref{r_2nd_HB}, where the square root to be real requires $\mu\omega\sigma_2 < 0$, which gives the claimed sub/supercriticality.
\end{proof}

This last theorem readily extends to the analogue of the so-called Bautin bifurcation for smooth vector fields, also called generalized Hopf bifurcation, which unfolds from zero first Lyapunov coefficient and identifies a curve of fold points. From \eqref{Thm2_eq_r0} we directly derive the loci fold points in the $(\mu, \sigma_{_\#})$-parameter plane as 
\[
\mu = \frac{2\omega}{9\pi\sigma_2} \sigma_{_\#}^2
\]
to leading order with respect to $\sigma_{_\#}$. 
Notably, the loci of fold points for the smooth Bautin bifurcation also lies on a quadratic curve in terms of the first Lyapunov coefficient. This last similarity is due to the fact that the ODE of the smooth case has no even terms in the radial component, $\dot{r} = \mu r + \sigma_s r^3 + \sigma_l r^5$, leading to $\mu = \frac{\sigma_s^2}{4\sigma_{l}}$, for $\frac{\sigma_s}{\sigma_l}<0$. In the  $(\mu,\sigma_{_\#})$-parameter plane, the origin corresponds to the Bautin point and the vertical axis, $\mu=0$, to the sub- and supercritical Hopf bifurcations for positive- and negative values of $\sigma_{_\#}$, respectively.

\section{Generalizations}\label{s:general}
In this section we discuss analogous bifurcation results for the generalization from the absolute value, \eqref{gen_abs_val}, and then turn to higher dimensional systems as well as general linear form of the linear part.

\subsection{Generalization from the absolute value}\label{s:appplanar}

Recall our notation for different left and right slopes \eqref{gen_abs_val}, and consider the generalized canonical equation 
\begin{equation}\label{e:genscalar}
	\dot{u} = \mu u+\sigma_{_\#} u^j [u\Pp,
\end{equation}
with left slope $p_-$, right slope $p_+$ and $j\in\N$ measuring the degree of smoothness such that the right-hand side is $C^j$ but not $C^{j+1}$ smooth. Sample bifurcation diagrams for $j=1$ and $j=2$ are plotted in Figure \ref{f:genslopes} for $\sigma_{_\#}=-1$. 

\begin{figure}
	\centering
	\begin{tabular}{cc}
		\includegraphics[width= 0.25\linewidth]{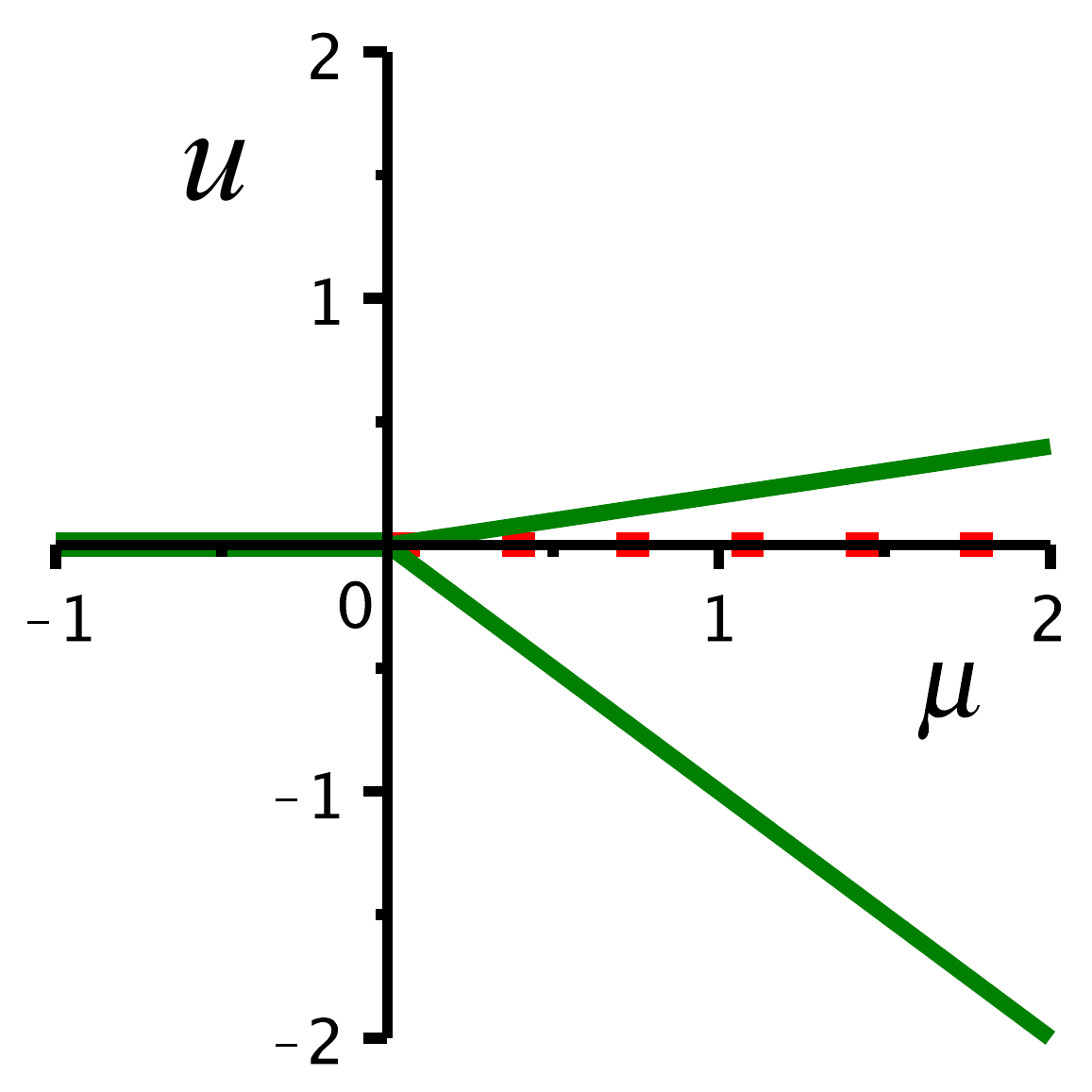}\hspace*{1cm}
		&\includegraphics[width= 0.25\linewidth]{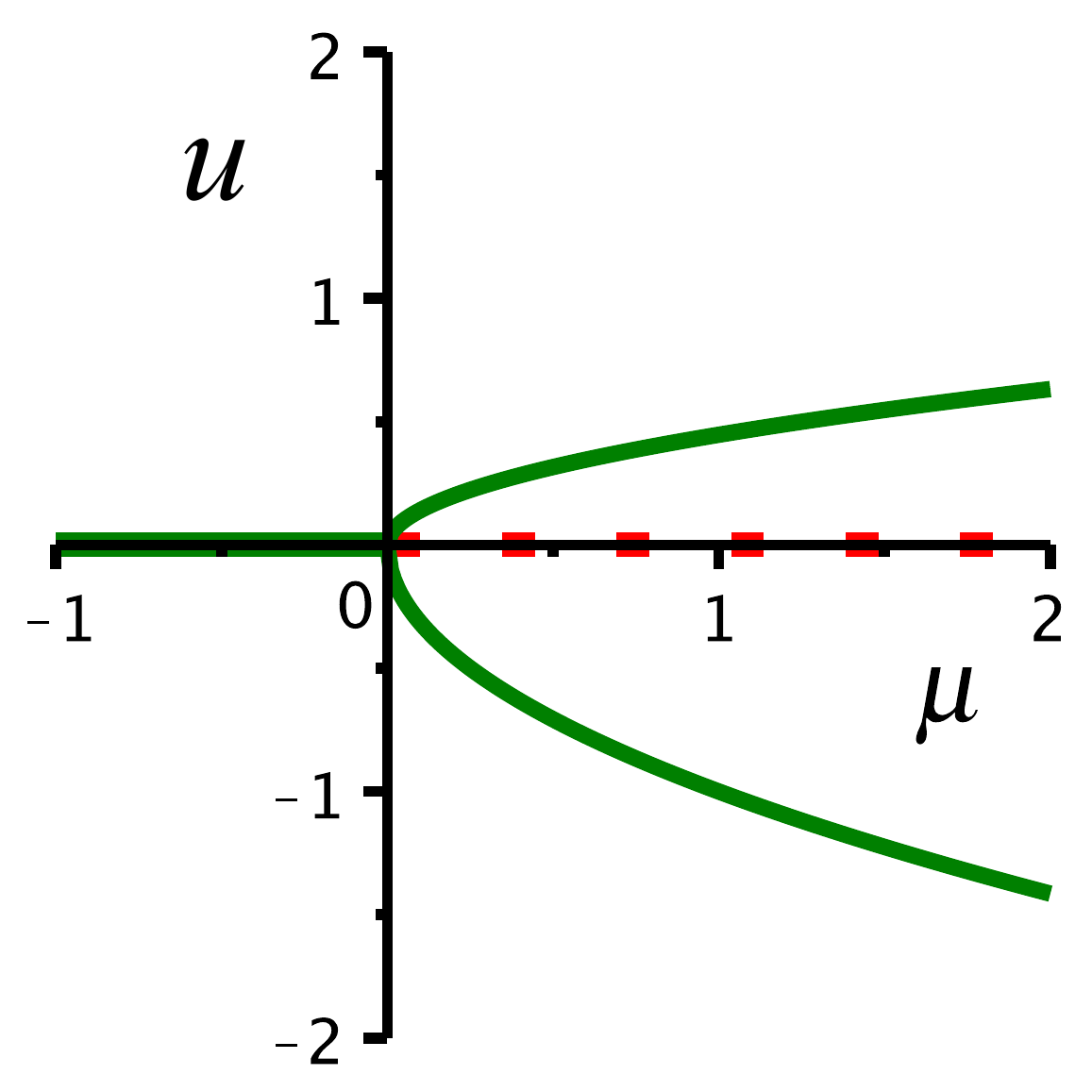}\\
		(a) & (b) 
	\end{tabular}
	\caption{Degenerated supercritical pitchfork bifurcation of \eqref{e:genscalar} for {$p_-=-1$, $p_+=5$} of degree $j=1$ (a) and $j=2$ (b).}\label{f:genslopes}
\end{figure}
The case $j=2$ highlights that also lack of smoothness in the cubic terms impacts the bifurcation in general. We do not pursue this further here, but analogous to the following discussion, it is possible to derive a modified normal form coefficient $S_c$.

For the Hopf bifurcation analysis, we analogously replace the absolute value in \eqref{e:quadnonlin} by \eqref{gen_abs_val}, and thus replace $f,g$ in \eqref{e:planar0} by
\begin{align}
	f\left( v, \w; \alpha \right) &= a_{11}v\ABS{v}{\alpha_1} + a_{12}v\ABS{\w}{\alpha_2} + a_{21}\w\ABS{v}{\alpha_3} + a_{22}\w\ABS{\w}{\alpha_4},\label{f_general}\\
	g\left( v, \w; \beta \right) &= b_{11}v\ABS{v}{\beta_1} + b_{12}v\ABS{\w}{\beta_2} + b_{21}\w\ABS{v}{\beta_3} + b_{22}\w\ABS{\w}{\beta_4},\label{g_general}
\end{align}
where $\alpha = (\alpha_{1_\pm},\alpha_{2_\pm},\alpha_{3_\pm},\alpha_{4_\pm})$, $\beta = (\beta_{1_\pm},\beta_{2_\pm},\beta_{3_\pm},\beta_{4_\pm}) \in\R^8$. 
This generalization leads to the generalized nonsmooth first Lyapunov coefficient given by
\begin{equation}\label{e:tildesig}
	\widetilde\sigma_{_\#}:=a_{11}(\alpha_{1_+}-\alpha_{1_-})+\frac 1 2 a_{12}(\alpha_{2_+}-\alpha_{2_-})+
	\frac 1 2 b_{21}(\beta_{3_+}-\beta_{3_-})+ b_{22}(\beta_{4_+}-\beta_{4_-}).
\end{equation}
Notably, in the smooth case, where the left- and right slopes coincide, we have $\widetilde\sigma_{_\#}=0$, and if left- and right slopes are $-1$ and $1$, respectively, we recover $\sigma_{_\#}$.

\begin{theorem}\label{Thm0_General}
	If $\widetilde\sigma_{_\#}\neq 0$, then the statement of Theorem~\ref{1stPart} holds true for \eqref{e:planar0} with $f,g$ from \eqref{f_general}, \eqref{g_general}, respectively, with $\sigma_{_\#}$ replaced by $\widetilde\sigma_{_\#}$.
\end{theorem}

\begin{proof}
	Taking polar coordinates 
	we obtain \eqref{System2Dpolar}, 
	where
	\begin{align*}
		\chi_2(\varphi) =& \; c^2\left( a_{11}\ABS{c}{\alpha_1}+a_{12}\ABS{s}{\alpha_2} \right) + s^2\left( b_{21}\ABS{c}{\beta_3}+b_{22}\ABS{s}{\beta_4} \right)\\ &+ sc\left( a_{21}\ABS{c}{\alpha_3} + a_{22}\ABS{s}{\alpha_4} + b_{11}\ABS{c}{\beta_1} + b_{12}\ABS{s}{\beta_2} \right),
	\end{align*}
	{again with $s:=\sin(\varphi)$, $c:=\cos(\varphi)$.}
	Applying Proposition \ref{Thm_Gen} 
	we compute $\int_0^{2\pi}\chi_2(\varphi)\D\varphi$, which gives $\frac{4}{3} \widetilde\sigma_{_\#} $. 
	Indeed, from \eqref{gen_abs_val} we obtain 
	\begin{align*}
		\int_0^{2\pi} c^2 a_{11}\ABS{c}{\alpha_1} \D\varphi &= \int_0^{\frac{\pi}{2}} c^3 a_{11}\alpha_{1_+} \D\varphi + \int_{\frac{\pi}{2}}^{\frac{3\pi}{2}} c^3 a_{11}\alpha_{1_-} \D\varphi + \int_{\frac{3\pi}{2}}^{2\pi} c^3 a_{11}\alpha_{1_+} \D\varphi \\
		&= \frac{4}{3}a_{11}\left( \alpha_{1_+}-\alpha_{1_-} \right), \\
		\int_0^{2\pi} c^2 a_{12}\ABS{s}{\alpha_2} \D\varphi &= \int_0^\pi c^2s\ a_{12}\alpha_{2_+} \D\varphi + \int_\pi^{2\pi} c^2s\ a_{12}\alpha_{2_-} \D\varphi \\
		&= \frac{2}{3}a_{12}\left( \alpha_{2_+}-\alpha_{2_-} \right),
	\end{align*}
	and similarly for the other terms. Note that the integral of the {third term on the} right-hand side of $\chi_2$ vanishes due to the symmetry of $sc$.
	Thus, we get \eqref{periodic_orbit} with $\sigma_{_\#}$ replaced by $\widetilde\sigma_{_\#}$.
\end{proof}

\subsection{$3$D system}\label{s:3D}

In this section we extend the previous results to higher dimensional systems. Recall that Proposition~\ref{prop:inv_man} and Theorem~\ref{t_per_orb} rely on hyperbolicity of the spectrum of $A(0)$ from \eqref{e:abstract} except for a simple pair of complex conjugate eigenvalues. Analogously, averaging theory can be used in this setting to obtain a normal form as in Theorem~\ref{t_averaging}. Here we follow the `direct method' and obtain bifurcation results also without normal hyperbolicity.

To simplify the exposition, we start with the absolute value $\abs{\cdot}$ and consider an extension of the planar quadratic case \eqref{e:planar0}, \eqref{e:quadnonlin}, motivated by the example in \cite{InitialPaper}, which is a simplification of a model used for ship maneuvering. As discussed in \S\ref{s:abstract}, we first assume the linear part is in normal form ---a general linear part will be considered in \S\ref{Gen_linear_part}--- which gives
\begin{equation}
	\begin{pmatrix}
		\dot{u}\\
		\dot{{v}}\\
		\dot{{\w}}\\
	\end{pmatrix} =
	\begin{pmatrix}
		c_1 u + c_2 u^2 + { c_{3} uv + c_{4} u\w } + c_5v\w  + h\left( v, \w \right)\\
		\mu v - \omega \w + c_6 uv + c_7 u\w + f\left( v, \w \right)\\
		\omega v + \mu \w + c_{8} uv + c_{9} u\w + g\left( v, \w \right)
	\end{pmatrix},
	\label{3DAbstractSystem}
\end{equation}
where $f,g$ are as in \eqref{e:quadnonlin}, $h\left( v, \w \right) = h_{11}v\abs{v} + h_{12}v\abs{\w} + h_{21}\w\abs{v} + h_{22}\w\abs{\w}$ 
and $h_{ij}, c_k$, $\forall i,j\in\{1,2\}$, $\forall k\in\{1,\ldots 9\}$, are real constants, all viewed as parameters. Again we assume $\omega\neq 0$ and take $\mu$ as the bifurcation parameter.

With linear part in normal form in the coordinates of Lemma \ref{l:cyl}, the vector field is actually smooth in the additional variable $u$. It turns out that in the generic case $c_1\neq 0$, this additional smoothness will not be relevant for the leading order analysis, while we make use of it in the degenerate case $c_1=0$.

We define the following quantities that appear in the upcoming results:
\begin{equation}\begin{aligned}
		\ogamma_{10} &= e^{\frac{2\pi c_1}{\omega}}-1, 
		&\ogamma_{20} &= \frac{e^{\frac{2\pi c_1}{\omega}}\left(e^{\frac{2\pi c_1}{\omega}}-1\right)}{c_1}  \left( c_2-\frac{c_1\rho_2}{\omega(c_1^2+4\omega^2)} \right),\\[1em]
		\ogamma_{02}  &= \frac{1}{\omega}e^{\frac{2\pi c_1}{\omega}} \int_0^{2\pi} e^{s\frac{2\mu-c_1}{\omega}}\Upsilon(s) \D s, 
		&\ogamma_{11} &= \frac{2}{3\omega^2}e^{\frac{2\pi c_1}{\omega}}
		\left[ c_1\left(2\tau_2 + \frac{P}	
		{3\omega}\mu\right) {-3\pi c_4\mu} \right] 
		+ \calO(\mu^2),\\[1em]
		\odelta_{01} &= \frac{2\pi\mu}{\omega} + \calO(\mu^2), 
		&\odelta_{02} &= \frac{2}{3\omega}\left[ \sigma_{_\#}\left(2 + \frac{6\pi}{\omega}\mu\right) + \frac{Q
		}{3\omega}\mu \right] + \calO(\mu^2),\\[1em]
		&&\odelta_{11} &= \frac{e^{\frac{2\pi c_1}{\omega}}-1}{c_1}\frac{1}{\omega(c_1^2+4\omega^2)} \big[ \omega\rho_1+R\mu \big] + \calO(\mu^2),
	\end{aligned}\label{ogammas_odeltas}\end{equation}
where we shortened the notation by lumping the weighted sums of coefficients from $f,g$, and from the smooth quadratic terms, respectively, given by 
\begin{align*}
	\tau_1&=4a_{22}+5a_{21}-5b_{12}-4b_{11}, \quad
	\tau_2=2a_{22}+a_{21}-b_{12}-2b_{11}, \quad \tau_3=a_{11}-a_{12}-b_{21}+b_{22},\\
	P& =3\pi(2\tau_2-a_{11}+b_{21})+4\tau_3,  \quad Q =-3\pi(b_{11}+a_{21})+2\tau_1, \quad\quad R=2\pi\rho_1-\rho_2,\\
	\rho_1&=c_6c_1^2-c_7c_1\omega +2c_6\omega^2-c_8c_1\omega+2c_9\omega^2,
	\quad\rho_2=c_8c_1^2-c_9c_1\omega+2c_8\omega^2+c_6c_1\omega-2c_7\omega^2, 
\end{align*}
as well as the $h$-dependent 
\[
\Upsilon(\varphi) = c_5cs+h_{11}c\abs{c}+h_{12}c\abs{s}+h_{21}s\abs{c}+h_{22}s\abs{s}.
\]
The explicit form of $\ogamma_{02}$ can be found in Appendix \ref{3D_gammas_deltas}.

\begin{theorem}\label{thm3D}
	In cylindrical coordinates $(u,v,\w)=(u,r\cos{\varphi},r\sin{\varphi})$, up to time shifts, periodic solutions to \eqref{3DAbstractSystem} with $r(0)=r_0, u(0)=u_0$ for $0\leq |\mu| \ll 1$ near $r=u=0$ are in 1-to-1 correspondence with solutions to the algebraic equation system
	\begin{align}
		0 &= \ogamma_{10} u_0 + \ogamma_{20} u_0^2 + \ogamma_{02} r_0^2 + \ogamma_{11} u_0r_0 + \mathcal{O}\left(3\right),\label{Periodic_R-a}\\
		0 &=  \odelta_{01} r_0 + \odelta_{02} r_0^2+ \odelta_{11} u_0r_0 + \mathcal{O}\left(3\right),\label{Periodic_R-b}
	\end{align}
	where $\calO(3)$ are terms of at least cubic order in $u_0,r_0$. 
\end{theorem}

\begin{proof}
	In cylindrical coordinates $(u,v,\w)=(u,r\cos{\varphi},r\sin{\varphi})$ system \eqref{3DAbstractSystem} becomes
	\begin{align}
		\label{polar_system_3D}
		\begin{cases}
			\dot{u} &= c_1 u + c_2 u^2 + { (c_3 c u + c_4 s u)r } + \Upsilon(\varphi)r^2,\\
			\dot{r} &= \left(\mu+\chi_1(\varphi)u\right)r + \chi_2(\varphi) r^2,\\
			\dot{\varphi} &= \omega + \Omega_0(\varphi)u + \Omega_1(\varphi) r,
		\end{cases}
	\end{align}	
	where 
	$\chi_1(\varphi) = c_6c^2+(c_7+c_8)cs+c_9s^2$, $\Omega_0(\varphi) = c_8c^2+(c_9-c_6)cs-c_7s^2$, 
	and the nonsmooth functions $\chi_2(\varphi)$ and $\Omega_1(\varphi)$ are as in \eqref{chi} and \eqref{Omega}, respectively.
	
	Upon rescaling time the equations for $u$ and $r$ of the previous system become
	\begin{align}\label{eq_u_r}
		\begin{cases}
			{u}' &= \frac{du/dt}{d\varphi /dt} = \frac{c_1 u + c_2 u^2 + { (c_3 c u + c_4 s u)r } + \Upsilon(\varphi)r^2}{\omega + \Omega_0(\varphi)u + \Omega_1(\varphi) r} =:\Psi_u(u,r,\varphi),\\[10pt]
			{r}' &= \frac{dr/dt}{d\varphi /dt} = \frac{\left(\mu+\chi_1(\varphi)u\right)r + \chi_2(\varphi) r^2}{\omega + \Omega_0(\varphi)u + \Omega_1(\varphi) r} =:\Psi_r(u,r,\varphi).
		\end{cases}
	\end{align}
	
	Taylor expansion of $u'$ and $r'$ in $(u,r)=(0,0)$ up to third order gives:
	\begin{subequations}
		\begin{align} 
			\begin{split} \label{UP-a}
				u' =& \Psi_u(0,0,\varphi) + \partial_u\Psi_u(0,0,\varphi)u + \partial_r\Psi_u(0,0,\varphi)r \\
				&+ \frac{1}{2}\partial_u^2\Psi_u(0,0,\varphi)u^2 + \frac{1}{2}\partial_r^2\Psi_u(0,0,\varphi)r^2 + \partial_{ur}^2\Psi_u(0,0,\varphi)ur + \mathcal{O}\left(3\right),
			\end{split}\\[5pt]
			\begin{split} \label{UP-b}
				r' =&  \Psi_r(0,0,\varphi) + \partial_u\Psi_r(0,0,\varphi)u + \partial_r\Psi_r(0,0,\varphi)r \\
				&+ \frac{1}{2}\partial_u^2\Psi_r(0,0,\varphi)u^2 + \frac{1}{2}\partial_r^2\Psi_r(0,0,\varphi)r^2 + \partial_{ur}^2\Psi_r(0,0,\varphi)ur + \mathcal{O}\left(3\right).
			\end{split}
		\end{align}
	\end{subequations}
	
	On the other hand, and similarly to the procedure of the $2$-dimensional case, we write $u(\varphi)$ and $r(\varphi)$ as the following expansions with  coefficients $\gamma_{ij}, \delta_{ij}$:
	\begin{equation}\begin{aligned}
			u(\varphi) &= \gamma_{10}(\varphi)u_0 + \gamma_{20}(\varphi)u_0^2 + \gamma_{01}(\varphi)r_0 + \gamma_{02}(\varphi)r_0^2 + \gamma_{11}(\varphi)u_0r_0 + \mathcal{O}\left(3\right), \\
			r(\varphi) &= \delta_{10}(\varphi)u_0 + \delta_{20}(\varphi)u_0^2 + \delta_{01}(\varphi)r_0 + \delta_{02}(\varphi)r_0^2 + \delta_{11}(\varphi)u_0r_0 + \mathcal{O}\left(3\right),
			\label{UE_RE}
	\end{aligned}\end{equation}
	with the initial conditions $u(0)=u_0$ and $r(0)=r_0$, which imply $\gamma_{10}(0)=\delta_{01}(0)=1$ and the rest zero.
	
	Substituting (\ref{UE_RE}) into (\ref{UP-a}) and (\ref{UP-b}) and matching the coefficients of the powers of $u_0$ and $r_0$ we get to solve a set of ODEs in order to obtain the expressions for $\gamma_{ij}$ and $\delta_{ij}$ (see Appendix \ref{3D_gammas_deltas} for the details). Using these, the system of boundary value problems $0 = u(2\pi) - u(0)$, $0 = r(2\pi) - r(0)$ for periodic solutions precisely yields
	\eqref{Periodic_R-a}, \eqref{Periodic_R-b}, where $\ogamma_{ij}=\gamma_{ij}(2\pi)-\gamma_{ij}(0)$ and $\odelta_{ij}=\delta_{ij}(2\pi)-\delta_{ij}(0)$.
\end{proof}

The solution structure of \eqref{Periodic_R-a}, \eqref{Periodic_R-b} strongly depends on whether $c_1=0$ or not. If not, then the transverse direction is hyperbolic and Theorem~\ref{1stPart} implies a locally unique branch of periodic solutions. In the nonhyperbolic case the situation is different and we note that if $c_1=0$, then 
with 
$\gamma_{_\#} :=  2h_{21} + c_5 + \pi h_{22}$, we have
\begin{equation}\label{e:c20}
	\ogamma_{10}=0, \hspace*{0.5cm} \ogamma_{11}={-\frac{2\pi c_4}{\omega^2}\mu+}\calO(\mu^2), \hspace*{0.5cm} \ogamma_{20}=\frac{2\pi c_2}{\omega}, \hspace*{0.5cm} 
	\ogamma_{02} = -\frac{\pi\gamma_{_\#}}{\omega^2}\mu + \calO(\mu^2).
\end{equation}

\begin{corollary}\label{c:3D}
	Consider system \eqref{3DAbstractSystem} in cylindrical coordinates $(u,v,\w)=(u,r\cos{\varphi},r\sin{\varphi})$. If $c_1\neq0$, then $u=u(\varphi;\mu) =  \mathcal{O}\left(\mu^2\right)$ and the statement of Theorem~\ref{1stPart} holds true.
	If $c_1=0$ and $\omega c_2\gamma_{_\#} \mu>0$, then precisely two curves of periodic solutions bifurcate at $\mu=0$ for $\mu\sigma_{_\#}\leq 0$, each in the sense of Theorem \ref{t_per_orb}, and their initial conditions $r(0)=r_0$, $u(0)=u_0^\pm$ satisfy
	\begin{align}
		u_0^\pm &= u_0^\pm(\mu) = 
		\mp \frac{3\pi}{2\sigma_{_\#}}\sqrt{\frac{\gamma_{_\#}}{2\omega c_2} \mu^3} + \calO(\mu^{2}) =\calO(|\mu|^{3/2}),\label{e:u3D} \\
		r_0 &= r_0(\mu)= -\frac{3\pi}{2\sigma_{_\#}}\mu + \calO(|\mu|^{3/2}). \label{e:3Dr2}
	\end{align}
	In case $c_1=0$ and $\omega c_2\gamma_{_\#} \mu<0$, there is no bifurcation through $\mu$.
\end{corollary}

\begin{proof}
	In the (transversely) hyperbolic case $c_1\neq 0$ we have $\ogamma_{10}\neq 0$, and thus one may solve \eqref{Periodic_R-a} for $u_0$ by the implicit function theorem as $u_0=u_0(r_0) =  \calO(r_0^2)$.
	Substitution into \eqref{Periodic_R-b} changes the higher order term only, so that to leading order we obtain the same problem as in Theorem~\ref{1stPart} with solution given by \eqref{periodic_orbit}. The stability statement of Theorem~\ref{1stPart} holds true from the existence of a $2$-dimensional Lipschitz continuous invariant manifold given by Proposition \ref{prop:inv_man}.
	
	We now consider $c_1=0$. Using \eqref{e:c20} we can cast \eqref{Periodic_R-a},  \eqref{Periodic_R-b} as
	\begin{align}
		0 &= \frac{2\pi c_2}{\omega} u_0^2 - \frac{\pi\gamma_{_\#}}{\omega^2} \mu r_0^2  {-\frac{2\pi c_4}{\omega^2}\mu u_0r_0} + \mathcal{O}\left(\mu^2 r_0^2\right)  + \mathcal{O}\left(3\right),\label{Periodic_R-aa}\\
		0 &=  \odelta_{01} r_0 + \frac{4\sigma_{_\#}}{3\omega} r_0^2 + \mathcal{O}\big(|u_0r_0| + |\mu r_0| (|u_0|+ |r_0|)\big)  + \mathcal{O}\left(3\right),\label{Periodic_R-bb}
	\end{align}
	so that we may solve \eqref{Periodic_R-aa} to leading order as
	\begin{equation}
		\label{exp_u0_of_r0}
		u_0 = u_0^\pm(r_0;\mu) = \frac{c_4}{\omega c_2}\mu r_0 \pm r_0\sqrt{\frac{c_4^2}{4c_2^2}\mu^2+\frac{\gamma_{_\#}}{2\omega c_2} \mu} + \calO(|\mu|) = \pm r_0\sqrt{\frac{\gamma_{_\#}}{2\omega c_2} \mu} + \calO(|\mu|).
	\end{equation}
	Substitution into \eqref{Periodic_R-b} gives a factor $r_0$ corresponding to the trivial solution $u_0=r_0=0$. For non-trivial solutions we divide by $r_0\neq 0$ and solve the leading order part as  
	$$ r_0 = - \frac{\odelta_{01}}{\frac{4\sigma_{_\#}}{3\omega} + \calO(\sqrt{\mu})} =  - \frac{3\pi}{2\sigma_{_\#}}\mu + \calO(|\mu|^{3/2}).$$
	Next, we substitute this into \eqref{exp_u0_of_r0} and note that perturbation by the higher order terms yields \eqref{e:u3D}, \eqref{e:3Dr2}. These give positive $r_0$ in case $\mu\sigma_{_\#}<0$ and therefore real valued $u_0$ in case $\omega c_2\gamma_{_\#} \mu>0$. However, if $\omega c_2\gamma_{_\#} \mu<0$ then for any $0<|\mu|\ll 1$ either $r_0<0$ or $u_0$ is imaginary.
\end{proof}

We subsequently consider the degenerate case $\sigma_{_\#}=0$, but assume $c_1\neq 0$, 
which generalizes Theorem~\ref{2ndPart} to the present $3$-dimensional setting. 
We will show that the generalization of $\sigma_2$ is given by $\omega^2\tcA$, where
\begin{equation}\label{tilde_o}
	\tcA := \tilde{\delta}_{03} - \tilde{\delta}_{11}\frac{\tilde{\gamma}_{02}}{\ogamma_{10}},
\end{equation}
with $\ogamma_{10}$ from \eqref{ogammas_odeltas}, and 
\begin{equation*}
	\begin{aligned}
		\tilde{\delta}_{11} &:= \odelta_{11}|_{\mu=0} =\frac{e^{\frac{2\pi c_1}{\omega}}-1}{c_1}\frac{\rho_1}{c_1^2+4\omega^2}, \hspace*{1cm}  \tilde{\gamma}_{02} := \ogamma_{02}|_{\mu=0} =\frac{1}{\omega}e^{\frac{2\pi c_1}{\omega}} \int_0^{2\pi}e^{-s\frac{c_1}{\omega}}\Upsilon(s) \D s,\\
		\tilde{\delta}_{03} &:= \cA + \frac{1}{\omega^2}\int_0^{2\pi} \chi_1(s)\int_0^s e^{c_1\frac{s-\tau}{\omega}}\Upsilon(\tau) \D\tau\D s,
	\end{aligned}
\end{equation*}
where $\cA=\frac{2}{\omega^2}\int_0^{2\pi}\chi_2(s) \int_0^{s}\chi_2(\tau)\D\tau \D s-\frac{1}{\omega^2}\int_0^{2\pi}\chi_2(s)\Omega_1(s)\D s$, as for \eqref{Thm2_eq_r0}. 
Comparing $\tcA$ with $\cA$, we expect $\tcA \neq \cA$, as a results of the coupling with the additional variable $u$. We omit here the fully explicit approach for $\tcA$, since the expressions become too lengthy for practical uses. However, for illustration, we consider the simpler case $h=0$  in \eqref{3DAbstractSystem}, which yields 
\begin{equation*}
	\tilde{\delta}_{03} = \cA + \frac{ c_5\pi(c_6+c_9)}{c_1^2+4\omega^2}\left(e^{\frac{2\pi}{\omega}c_1}-1\right),  \hspace*{1cm}
	\tilde{\delta}_{11}\frac{\tilde{\gamma}_{02}}{\ogamma_{10}} =  \frac{c_5\rho_1\omega}{c_1(c_1^2+4\omega^2)^2} \left(e^{\frac{2\pi}{\omega}c_1}-1\right),
\end{equation*}
and thus,
\begin{equation*}
	\tcA = \cA + \frac{c_5\left(e^{\frac{2\pi}{\omega}c_1}-1\right)}{c_1^2+4\omega^2}\left[ \pi (c_6+c_9) - \frac{\rho_1\omega}{c_1(c_1^2+4\omega^2)} \right].
\end{equation*}

\begin{corollary}\label{Second:c:3D}
	Consider \eqref{3DAbstractSystem} in cylindrical coordinates $(u,v,\w)=(u,r\cos{\varphi},r\sin{\varphi})$ and $\sigma_{_\#}=0$. If $c_1\neq0$, then $u=u(\varphi;\mu) =  \mathcal{O}\left(\mu^2\right)$ and the statement of Theorem~\ref{2ndPart} holds true {with} $\sigma_2$ {replaced} by $\omega^2\tcA$.
\end{corollary}

\begin{proof}
	Upon rescaling time the equations for $u,r$ in cylindrical coordinates of \eqref{3DAbstractSystem} become \eqref{eq_u_r}. Similarly to the proof of Theorem \ref{thm3D}, we compute the Taylor expansion of $u'$ and $r'$ in $(u,r)=(0,0)$ up to forth order (see Appendix \ref{3D_gammas_deltas_second} for the details) and we write $u(\varphi)$ and $r(\varphi)$ as the following expansions:
	\begin{equation}\begin{aligned}	u(\varphi) =& \gamma_{10}(\varphi)u_0 + \gamma_{20}(\varphi)u_0^2 + \gamma_{30}(\varphi)u_0^3 + \gamma_{01}(\varphi)r_0 + \gamma_{02}(\varphi)r_0^2 + \gamma_{03}(\varphi)r_0^3 \\
			&+ \gamma_{11}(\varphi)u_0r_0 + \gamma_{21}(\varphi)u_0^2r_0 + \gamma_{12}(\varphi)u_0r_0^2 + \mathcal{O}\left(4\right), \\[0.5em]
			r(\varphi) =& \delta_{10}(\varphi)u_0 + \delta_{20}(\varphi)u_0^2 + \delta_{30}(\varphi)u_0^3 + \delta_{01}(\varphi)r_0 + \delta_{02}(\varphi)r_0^2 + \delta_{03}(\varphi)r_0^3 \\
			&+ \delta_{11}(\varphi)u_0r_0 + \delta_{21}(\varphi)u_0^2r_0 + \delta_{12}(\varphi)u_0r_0^2 + \mathcal{O}\left(4\right),
			\label{UE_RE_Second}
	\end{aligned}\end{equation}
	with the initial conditions $u(0)=u_0$ and $r(0)=r_0$, which imply $\gamma_{10}(0)=\delta_{01}(0)=1$ and the rest zero. With these expressions we compute, as before, the functions $\gamma_{ij}$ and $\delta_{ij}$ $\forall i,j\in\N_0$ such that $i+j=3$. Note that the others are the same as for Theorem \ref{thm3D}. The periodic solutions with $r(0)=r_0$, $u(0)=u_0$ for $0\leq |\mu| \ll 1$ near $r=u=0$ are in 1-to-1 correspondence with solutions to the algebraic equation system
	\begin{align}
		0 &= \ogamma_{10} u_0 + \ogamma_{20} u_0^2 + \ogamma_{30} u_0^3 + \ogamma_{02} r_0^2 + \ogamma_{03} r_0^3 + \ogamma_{11} u_0r_0 + \ogamma_{21} u_0^2r_0 + \ogamma_{12} u_0r_0^2 + \mathcal{O}\left(4\right),\label{Periodic_R-a3}\\
		0 &=  \odelta_{01} r_0 + \odelta_{02} r_0^2 + \odelta_{03} r_0^3 + \odelta_{11} u_0r_0 + \odelta_{21} u_0^2r_0 + \odelta_{12} u_0r_0^2 + \mathcal{O}\left(4\right),\label{Periodic_R-b3}
	\end{align}
	where $\calO(4)$ are terms of at least fourth order in $u_0$, $r_0$, and $\ogamma_{ij}=\gamma_{ij}(2\pi)-\gamma_{ij}(0)$, $\odelta_{ij}=\delta_{ij}(2\pi)-\delta_{ij}(0)$.  
	Moreover, since $c_1\neq 0$ we have $\ogamma_{10}\neq 0$. Therefore, we may solve \eqref{Periodic_R-a3} for $u_0$ by the implicit function theorem as $u_0=-\frac{\ogamma_{02}}{\ogamma_{10}}r_0^2+\calO\left(r_0^3\right)=\calO\left(r_0^2\right)$.
	
	Substitution into \eqref{Periodic_R-b3} and dividing out $r_0\neq 0$ yield 
	\begin{equation*}
		0 =  \odelta_{01} + \odelta_{02} r_0 + \left(\odelta_{03} - \odelta_{11}\frac{\ogamma_{02}}{\ogamma_{10}}\right)r_0^2 + \mathcal{O}\left(3\right),
	\end{equation*}
	which we rewrite, to leading order and similarly to \eqref{Thm2_eq_r0} in Theorem~\ref{2ndPart}, as
	\begin{equation}\label{Cor_deg}
		0 = \frac{2\pi}{\omega}\mu + \tcB r_0 + \tcA r_0^2 + \mathcal{O}\left(\mu^2 + \mu r_0 + r_0^3\right),
	\end{equation}
	where $\tcB=\odelta_{02}|_{\mu=0}$, which vanishes for $\sigma_{_\#}=0$ analogous to $\cB$ in \eqref{Thm2_eq_r0}, and $\tcA$ is as defined in \eqref{tilde_o}; the expression for $\odelta_{03}$ stems from \eqref{delta_03}.
	Hence, the solution for \eqref{Cor_deg} is given by \eqref{r_2nd_HB} replacing $\sigma_2$ by $\omega^2\tcA$, which is assumed to be nonzero.
	The stability statement of Theorem~\ref{1stPart} holds true from the existence of a $2$-dimensional Lipschitz continuous invariant manifold given by Proposition \ref{prop:inv_man}.
\end{proof}

Lastly, we use these results to extend system \eqref{3DAbstractSystem} to a higher order model with the generalized absolute value \eqref{gen_abs_val} as follows
\begin{equation}
	\begin{pmatrix}
		\dot{u}\\
		\dot{{v}}\\
		\dot{{\w}}\\
	\end{pmatrix} =
	\begin{pmatrix}
		c_1 u + c_2 u^2 + { c_3 uv + c_4 u\w } + c_5 v\w + h(v, \w; \gamma)\\
		\mu v - \omega \w + c_6 uv + c_7 u\w + f\left( v, \w; \alpha \right) + f_q\left( v, \w\right) + f_c\left( v, \w\right) \\
		\omega v + \mu \w + c_8 uv + c_9 u\w + g\left( v, \w; \beta \right) + g_q\left( v, \w\right) + g_c\left( v, \w \right) \\
	\end{pmatrix},
	\label{3DAbstractSystem_General}
\end{equation}
where $f\left( v, \w; \alpha \right)$ and $g\left( v, \w; \beta \right)$ are \eqref{f_general} and \eqref{g_general}, respectively, and the functions $f_q,g_q,f_c,g_c$ are as in system \eqref{General2D_AV}. The expression of $h$ is analogous to $f,g$. {We recall also $\widetilde\sigma_{_\#}$ from \eqref{e:tildesig}.}

\begin{corollary}
	\label{Generalization_3D}
	If $\widetilde\sigma_{_\#}\neq 0$, the statement of Corollary \ref{c:3D} for system \eqref{3DAbstractSystem_General} holds true with $\sigma_{_\#}$ replaced by $\widetilde\sigma_{_\#}$. 
\end{corollary}

\begin{proof}
	The proof follows from Theorems \ref{Thm0_General} and \ref{thm3D} and Corollary \ref{thm_cubic}.
\end{proof}

This concludes our analysis for the $3$-dimensional case, which paves the way for the $n$-dimensional case discussed thereafter.

\subsection{$n$D system}\label{s:nD}

We consider the $n$-dimensional generalization of system \eqref{3DAbstractSystem} with additional component $u=(u_1,\cdots,u_{n-2})\in\R^{ n-2}$ given by
\begin{equation}
	\begin{pmatrix}
		\dot{u}\\
		\dot{{v}}\\
		\dot{{\w}}\\
	\end{pmatrix} =
	\begin{pmatrix}
		\tA u + U(u,v,\w)\\
		\mu v - \omega \w + \sum_{i=1}^{n-2}({c_6}_i u_iv + {c_7}_i u_i\w) + \tilde f\left( v, \w \right)\\
		\omega v + \mu \w + \sum_{i=1}^{n-2}({c_8}_i u_iv + {c_9}_i u_i\w) + \tilde g\left( v, \w \right)
	\end{pmatrix},
	\label{nDAbstractSystem}
\end{equation}
where $\tA=({c_1}_{ij})_{1\leq i,j\leq n-2}$ is an $(n-2)\times(n-2)$ matrix 
and $U: \R^{n-2}\times\R\times\R \longrightarrow \R^{n-2}$ is a nonlinear function, smooth in $u$ and possibly nonsmooth in $v,\w$ with absolute values as in \eqref{3DAbstractSystem}. Hence, $U(u,v,\w) = \calO(2)$, 
where $\calO(2)$ are terms of at least second order in $u_i,v,\w$. The constants ${c_1}_{ij}, {c_6}_i, {c_7}_i, {c_8}_i, {c_9}_i$ are all real $\forall i,j\in\{1,\cdots, n-2\}$, and the functions $\tilde f, \tilde g$ are of the same form as the nonlinear part of system \eqref{General2D_AV}. 

We present now analogous results as before for this $n$-dimensional case.
However, we refrain from explicitly determining the coefficients involved.

\begin{theorem}
	\label{Thm_nD}
	Consider \eqref{nDAbstractSystem} in cylindrical coordinates $(u,v,\w)=(u,r\cos{\varphi},r\sin{\varphi})$ 
	analogous to Theorem \ref{thm3D} with $u\in\R^{n-2}$. Up to time shifts, periodic solutions to \eqref{nDAbstractSystem} with $r(0)=r_0$, $u(0)=u_0\in\R^{n-2}$, 
	for $0\leq |\mu|, r_0, |u_0| \ll 1$ are in 1-to-1 correspondence with solutions to the algebraic $(n-1)$-dimensional system given by equations analogous to \eqref{Periodic_R-a} and \eqref{Periodic_R-b}, where $\odelta_{01},\odelta_{02}$ are scalars and $\ogamma_{10}, \ogamma_{11}, \ogamma_{02}, \ogamma_{20}, \odelta_{11}$ are linear maps and quadratic forms in $n-2$ dimensions.
\end{theorem}

\begin{proof}
	The proof is analogous to that of Theorem \ref{thm3D}, now by setting up a boundary value problem with $n-2$ equations for $0=u(2\pi)-u(0)$ and one for $0=r(2\pi)-(0)$. This results in a system of $n-1$ equations formed by direct analogues to \eqref{Periodic_R-a} and \eqref{Periodic_R-b}, where $\calO(3)$ contains all terms of at least cubic order in ${u_0}_i, r_0$, 
	and $\ogamma_{20}u_0^2$ is a quadratic form in $n-2$ dimensions.
\end{proof}

Similar to the $3$-dimensional case, the solution structure of the $(n-1)$-dimensional system \eqref{Periodic_R-a}, \eqref{Periodic_R-b} depends on whether the matrix $\tA$ is hyperbolic (i.e., {the full linear part} $A$ satisfies Hypothesis~\ref{h:AG}) or not, as shown in the next result.

\begin{corollary}
	\label{c:nD}
	Consider \eqref{nDAbstractSystem} in cylindrical coordinates $(u,v,\w)=(u,r\cos{\varphi},r\sin{\varphi})$. If $\tA$ is hyperbolic, 
	then the solution vector $u=u(\varphi;\mu)$ is of order $\calO\left(\mu^2\right)$ and the statement of Theorem \ref{1stPart} holds true. If $\tA$ is not hyperbolic with $1$-dimensional generalized kernel, then there are constants $c_2$, $\gamma_{\#}$ such that the statements of Corollary \ref{c:3D} for $c_1=0$ hold true.
\end{corollary}

\begin{proof}
	From Theorem \ref{Thm_nD} we have the corresponding equations \eqref{Periodic_R-a}, \eqref{Periodic_R-b} for the $n$-dimensional system \eqref{nDAbstractSystem}, where $\ogamma_{20}u_0^2$ is a quadratic form in $n-2$ dimensions. If $\tA$ is hyperbolic, then the $(n-2)\times(n-2)$ matrix $\ogamma_{10}=e^{2\pi \tA/\omega}-\mathrm{Id}$ is invertible. Solving the $(n-1)$-dimensional system gives the same as in the proof of Corollary \ref{c:3D} to leading order.
	
	If $\tA$ is not hyperbolic, then by assumption it has a $1$-dimensional generalized kernel. In this case, we {change coordinates in the analogue of} \eqref{Periodic_R-a} such that the matrix $\ogamma_{10}$ is block-diagonal with the kernel in the top left, and an invertible $(n-3)\times(n-3)$ block $\ogamma'_{10}$ on the lower right of the matrix. Thus, we split \eqref{Periodic_R-a} into a scalar equation and a $(n-3)$-dimensional system. By the implicit function theorem we solve the equations corresponding to $\ogamma'_{10}$ and substitute the result into the other two equations: the one with the  $1$-dimensional kernel and the corresponding \eqref{Periodic_R-b} with $\odelta_{01} = \frac{2\pi\mu}{\omega} + \calO(\mu^2)$, $\odelta_{02} = \frac{4\sigma_{_\#}}{3\omega} + \calO(\mu)$. We obtain then two scalar equations of the same type as in Corollary \ref{c:3D} for the case $c_1=0$.
\end{proof}

We omit explicit formulas for $c_2, \gamma_{\#}$, but note that these can be provided in terms of data from $\tA$.
Before concluding this section, we note that these results directly extend to the more general nonsmooth terms \eqref{gen_abs_val} and to additional higher order functions as in \eqref{General2D_AV}.

\begin{corollary}\label{c:nD:gen_abs}
	Consider system \eqref{nDAbstractSystem} with $\tilde f, \tilde g$ as the nonlinear part of \eqref{General2D_AV}, but with $f,g$ as in \eqref{f_general}, \eqref{g_general}, respectively. If $\tA$ is hyperbolic and $\widetilde\sigma_{_\#}\neq 0$, cf.\ \eqref{e:tildesig}, then the statement of Corollary \ref{c:nD} holds true with $\sigma_{_\#}$ replaced by $\widetilde\sigma_{_\#}$.
\end{corollary}

Recall from \S\ref{s:abstract} that we have presented results for systems where the linear part is in block-diagonal form and normal form for the oscillatory part, while the nonlinear part is smooth in the radial direction. For completeness, we next discuss the case of general linear part, i.e., not necessarily in normal form.

\subsection{General linear part}\label{Gen_linear_part}

Here we show that our analysis also applies to systems with general linear part. First, we consider the planar case \eqref{e:abstractplanar} with
\begin{align*}
	f_1(u_1,u_2)&=a_{11}u_1|u_1|+a_{12}u_1|u_2|+a_{21}u_2|u_1|+a_{22}u_2|u_2| +\calO(3),\\
	f_2(u_1,u_2)&=b_{11}u_1|u_1|+b_{12}u_1|u_2|+b_{21}u_2|u_1|+b_{22}u_2|u_2| +\calO(3).
\end{align*} 
Under Hypothesis~\ref{h:AG}, changing the linear part of \eqref{e:abstractplanar} to normal form by the associated matrix $\bT$, 
i.e., $\bT\cdot(v_1,v_2)^\mathsf{T}=(u_1,u_2)^\mathsf{T}$, the system becomes
\begin{equation}\label{e:abstractplanar:2}
	\begin{pmatrix}
		\dot v_1\\
		\dot v_2
	\end{pmatrix} = 
	\begin{pmatrix}
		\mu & -\omega\\
		\omega & \mu
	\end{pmatrix}\begin{pmatrix}
		v_1\\
		v_2
	\end{pmatrix}+\bT^{-1}\begin{pmatrix}
		g_1\left( v_1, v_2 \right)\\
		g_2\left( v_1, v_2 \right)
	\end{pmatrix},
\end{equation} where $g_i(v_1,v_2)=f_i\left(\bT\cdot (v_1,v_1)^\mathsf{T}\right)$ for $i\in\{1,2\}$ and with $\bT=(z_{ij})_{1\leq i,j\leq 2}$, {as well as the shorthand $[[\cdot]]:=\cdot|\cdot|$}, 
we have 
\begin{align*}
	g_1(v_1,v_2)=&a_{11} [[z_{11}v_1+z_{12}v_2]] +a_{12}(z_{11}v_1+z_{12}v_2)|z_{21}v_1+z_{22}v_2|\\
	&+a_{21}(z_{21}v_1+z_{22}v_2)|z_{11}v_1+z_{12}v_2|+a_{22} [[z_{21}v_1+z_{22}v_2]] +\calO(3),\\[0.5em]
	g_2(v_1,v_2)=&b_{11} [[z_{11}v_1+z_{12}v_2]] +b_{12}(z_{11}v_1+z_{12}v_2)|z_{21}v_1+z_{22}v_2|\\
	&+b_{21}(z_{21}v_1+z_{22}v_2)|z_{11}v_1+z_{12}v_2|+b_{22} [[z_{21}v_1+z_{22}v_2]]  +\calO(3).
\end{align*}
{We use} polar coordinates for $(v_1,v_2)=(r\cos(\varphi),r\sin(\varphi))$ {as before,} and
\[
(z_{11},z_{12})=(C\cos(\phi),C\sin(\phi)), \hspace*{1cm} (z_{21},z_{22})=(D\cos(\vartheta),D\sin(\vartheta)),
\] where $C,D\in\R$, $\phi,\vartheta\in[0,2\pi)$ are fixed constants. 
System \eqref{e:abstractplanar:2} can be written as
\begin{equation}
	\begin{cases}
		\dot{r} = \mu r+\chi_2(\varphi)r^2 + \calO(r^3),\\
		\dot{\varphi} = \omega + \Omega_1(\varphi)r + \calO(r^2),
	\end{cases}
	\label{e:abstract:polar}
\end{equation} 
where, using trigonometric identities, we have
{\begin{align*}
		\chi_2(\varphi) =& \frac{1}{\det(\bT)}\Big( [[\cos(\varphi-\phi)]]C\abs{C}(a_{11}\c+b_{11}\s) + \cos(\varphi-\phi)|\cos(\varphi-\vartheta)|C\abs{D}(a_{12}\c+b_{12}\s) \\
		&+ \cos(\varphi-\vartheta)|\cos(\varphi-\phi)|\abs{C}D(a_{21}\c+b_{21}\s) + [[\cos(\varphi-\vartheta)]]D\abs{D}(a_{22}\c+b_{22}\s) \Big).
\end{align*}}
{with $\c:=D\sin(\vartheta-\varphi)$, $\s:=C\sin(\varphi-\phi)$.} 
By assumption, $\omega\neq 0$ so that 
rescaling time in \eqref{e:abstract:polar} analogous to \eqref{e:absper} gives \eqref{new_time} with $M(\varphi)=\mu$ and $W(\varphi)=\omega$.
Following the approach described in \S\ref{s:abstract}, for the analogue of \eqref{r_bar} we obtain
\begin{align}
	\Lambda &= \frac{1}{2\pi} \int_0^{2\pi}\frac{\mu}{\omega}\D \varphi= \frac{\mu}{\omega}, \\
	\Sigma &= \frac{1}{2\pi} \int_0^{2\pi}\frac{\chi_2(\varphi)}{\omega} \D\varphi, 
	\label{check_sigma_nnf}
\end{align}
where we set $\mu=0$ in \eqref{check_sigma_nnf} (unlike in \eqref{check_sigma})
and the expression for $\Sigma$ can be determined explicitly. For instance, the first term of $\chi_2(\varphi)$ can be integrated as
\begin{equation*}
	\frac{C{\abs{C}D}}{{\det(\bT)}}a_{11}\int_0^{2\pi} [[\cos(\varphi-\phi)]] {\sin(\vartheta-\varphi)} \D\varphi = \frac{8C{\abs{C}D}}{3{\det(\bT)}}a_{11}{\sin(\vartheta-\phi)=\frac{8}{3}\abs{C}a_{11}},
\end{equation*}
{with last equality due to $\det(\bT)=CD\sin(\vartheta-\phi)$.} Computing the integral of $\chi_2(\varphi)$, {equation \eqref{check_sigma_nnf} turns into} 
{\begin{equation}\begin{aligned}\label{generalized_sigma}
			\Sigma 
			=\frac{2}{3\pi\omega}&\Big[ 2\abs{C}a_{11} + \abs{D}a_{12} +\abs{C}b_{21} + 2\abs{D}b_{22} \\
			&+\cos(\vartheta-\phi)\big(\sgn(C)Da_{21}+\sgn(D)Cb_{12}\big)\Big].
\end{aligned}\end{equation}}
In case $\phi=0$ and $\vartheta=\frac{\pi}{2}$,
{we have $\cos(\vartheta-\phi)=0$ so that the last few terms in  \eqref{generalized_sigma} vanish}
and for $C=D=1$ the same expression as in \eqref{averaging_integrals} is obtained, i.e., $\Sigma=\frac{2}{3\pi\omega}\sigma_{_\#}$. Notice that this set of parameters gives $z_{11}=z_{22}=1$, $z_{12}=z_{21}=0$, i.e., $\bT$ is the identity.

Moreover, we can
{derive} the analogue of \eqref{generalized_sigma} for the generalized nonsmooth function \eqref{gen_abs_val} and compute the integrals involved in the generalized $\chi_2(\varphi)$ as in the proof of Theorem \ref{Thm0_General}. For instance, some of them read, omitting the factor $\det(\bT)^{-1}$,
\begin{align*}
	C{\abs{C}}\int_0^{2\pi} \cos(\varphi-\phi)&\left( \ABS{\cos(\varphi-\phi)}{\alpha_1} {D\sin(\vartheta-\varphi)}a_{11}+\ABS{\cos(\varphi-\phi)}{\beta_1}{C\sin(\varphi-\phi)}b_{11}\right) \D\varphi \\
	=\frac{4}{3}C{\abs{C}} &{D\sin(\vartheta-\phi)}a_{11}\left(\alpha_{1_+}-\alpha_{1_-}\right),
\end{align*}
\begin{align*}
	C{\abs{D}}\int_0^{2\pi} \cos(\varphi-\phi)&\left(\ABS{\cos(\varphi-\vartheta)}{\alpha_2}{D\sin(\vartheta-\varphi)}a_{12}+\ABS{\cos(\varphi-\vartheta)}{\beta_2}{C\sin(\varphi-\phi)}b_{12}\right) \D\varphi \\
	= \frac{1}{3}C{\abs{D}}\bigg[ &2{D\sin(\vartheta-\phi)}a_{12}\left(\alpha_{2_+}-\alpha_{2_-}\right)
	+{C\sin\big(2(\vartheta-\phi)\big)}b_{12}\left(\beta_{2_+}-\beta_{2_-}\right) \bigg].
\end{align*} 
{The full expression can be simplified to }
{\begin{equation*}
		\begin{aligned}
			\widetilde\Sigma :=
			\frac{2}{3\pi\omega} \bigg[&\abs{C}a_{11}\left(\alpha_{1_+}-\alpha_{1_-}\right)+\frac{\abs{D}}{2}a_{12}\left(\alpha_{2_+}-\alpha_{2_-}\right)\\
			&+\frac{\abs{C}}{2}b_{21}\left(\beta_{3_+}-\beta_{3_-}\right)
			+\abs{D}b_{22}\left(\beta_{4_+}-\beta_{4_-}\right) \\
			&+ \frac{\cos(\vartheta-\phi)}{2}\Big(\sgn(C)Da_{21}\left(\alpha_{3_+}-\alpha_{3_-}\right)+\sgn(D)Cb_{12}\left(\beta_{2_+}-\beta_{2_-}\right)\Big) \bigg].
		\end{aligned}
	\end{equation*}
	As above, for $\phi=0$ and $\vartheta=\frac{\pi}{2}$ the last few terms vanish, and for $C=D=1$ we have $\widetilde\Sigma = \frac{2}{3\pi\omega}\widetilde\sigma_{_\#}$ with $\widetilde\sigma_{_\#}$ from \eqref{e:tildesig}.}

\medskip
Furthermore, we can extend these results for the case $n>2$ in the form of a coupled system similar to \eqref{e:cylindrical0} using the approach presented in the proof of Theorem~\ref{t_per_orb}. This gives an integral expression for the generalized first Lyapunov coefficient which provides an explicit algebraic formula for an adjusted $\widetilde\sigma_{_\#}$. We comprise this in the following result.

\begin{theorem}\label{Thm_Gen_Lin_Part}
	Consider system \eqref{e:abstract} with general linear part $A(\mu)$, and 
	satisfying the hypotheses of Theorem \ref{t:abstractnormal}.
	The statement of Corollary \ref{c:nD:gen_abs} 
	holds true with $\widetilde\sigma_{_\#}$ replaced by $\frac{3\pi\omega}{2}\widetilde{\Sigma}$.
\end{theorem}

In particular, this theorem covers system \eqref{e:abstractplanar} with general  matrix $A=(m_{ij})_{1\leq i,j\leq 2}$. We also remark that the system considered here is neither of the form of \eqref{3DAbstractSystem} nor \eqref{nDAbstractSystem} in terms of smoothness of the $u$ variable.
\begin{proof}
	We proceed as before to get the analogue of system \eqref{e:abstractplanar:2}, i.e., transforming the linear part into a block-diagonal matrix and normal form in the center eigenspace $E^\rmc$. From Theorem \ref{t:abstractnormal} the nonlinear terms are second order modulus terms, which in this case are of the form $(L_i(v,w)+K_i(u))\ABS{L_j(v,w)+K_j(u)}{p}$, $i,j\in\{1,2\}$, where the functions $L_i(v,w)$ are linear combinations of $v,w$; $K_i(u)$ are linear combinations of the components of the vector $u$, i.e., $u_l$, $\forall l\in\{1,\cdots,n-2\}$; and $p_+,p_-\in\R$ are as in \eqref{gen_abs_val}. Note that $L_1, K_1$ are not necessarily equal to $L_2, K_2$, respectively.
	The previous product can be expanded as 
	\begin{equation}\label{gen_expansion}
		(L_i(v,w)+K_i(u))\ABS{L_j(v,w)+K_j(u)}{p} = L_i\ABS{L_j}{p} + \calO(L_iK_j+L_jK_i+K_iK_j),
	\end{equation}
	since the error term $p_+ L_i L_j - L_i\ABS{L_j}{p}$ (resp. $p_- L_i L_j - L_i\ABS{L_j}{p}$ ) is of order $|u|^2$, i.e., contained in the higher order terms of \eqref{gen_expansion}. More precisely, consider the case $L_j+K_j \geq 0$. Then, the error term is $p_+L_i L_j - L_i\ABS{L_j}{p}$, which is zero for $L_j \geq 0$, and otherwise $(p_+ - p_-)L_iL_j$. However, in order to have both $L_j+K_j \geq 0$ and $L_j<0$, the signs of $L_j$ and $K_j$ have to differ, which happens only if these magnitudes are comparable. Hence, $\calO(L_j)=\calO(K_j)$. For the case $L_j+K_j < 0$ we proceed analogously.
	
	In particular, $\calO(L_iK_j+L_jK_i+K_iK_j) = \calO(K(\check{K}+L))$, where $K, \check{K}$ are linear combinations of the components of $u$, and $L$ of $v, w$.
	
	Following the proof of Theorem \ref{t_per_orb}, we write $u=r\tu$ and, together with the change of polar coordinates from above, $L_i=r\cos(\varphi-\zeta_i)$ (where $\zeta_i$ is either $\phi$ or $\vartheta$), so that  
	$$(L_i(v,w)+rK_i(\tu))\ABS{L_j(v,w)+rK_j(\tu)}{p} = r^2\cos(\varphi-\zeta_i)\ABS{\cos(\varphi-\zeta_j)}{p} + r^2\calO(\tu).$$
	From Theorem \ref{t_per_orb} we have $\tu=\calO(r_0)$ and thus $r^2\calO(\tu)$ is of higher order. We can then integrate explicitly the leading order as done for \eqref{check_sigma_nnf}.
\end{proof}

We implement now these results to an applied $3$-dimensional model in the field of land vehicles.

\section{A 3D example: shimmying wheel}\label{s:shim}
For illustration of the theory and its practice, we consider as an example the model of a shimmying wheel with contact force analyzed in \cite{SBeregi}, where a towed caster with an elastic tyre is studied.
The equations of motion of the towed wheel can be written as follows:
\begin{equation}
	\begin{pmatrix}
		\dot{\Omega}\\
		\dot{\psi}\\
		\dot{q}
	\end{pmatrix}=
	\mathbf{J}\begin{pmatrix}
		{\Omega}\\
		{\psi}\\
		{q}
	\end{pmatrix} + {\tilde{c}_4}
	\begin{pmatrix}
		q\abs{q}\\
		0\\
		0
	\end{pmatrix},\quad
	\textbf{J}:=
	\begin{pmatrix}
		\tilde{c}_1 & \tilde{c}_2 & \tilde{c}_3\\
		1 & 0 & 0\\
		\tilde{c}_5 & \tilde{c}_6 & \tilde{c}_7
	\end{pmatrix},
	\label{System1Beregi}
\end{equation}
where $\psi$ is the yaw angle, $q$ is the deformation angle of the tyre due to the contact with the ground and $\Omega=\dot{\psi}$, and the parameters $\tilde{c}_i\in\R$ are constants determined by the system. We can readily see that there is only one switching surface in this case, namely $\{q=0\}$. Here $\mathbf{J}$ is the Jacobian matrix at the equilibrium point $(\Omega,\psi,q)=(0,0,0)$.

The system is of the form \eqref{e:abstract} and suitable parameter choices yield a pair of complex conjugate eigenvalues crossing the imaginary axis, as well as one nonzero real eigenvalue. The resulting bifurcations were studied in \cite{SBeregi} and termed `dynamic loss of stability'. Here we expound how our approach applies to this system. 

Clearly, Theorem~\ref{t_per_orb} applies for any Hopf bifurcation eigenvalue configuration, which proves that a unique branch of periodic solutions bifurcates. In order to identify the direction of bifurcation, we first aim to apply the results of \S\ref{s:3D} and therefore attempt to bring the nonlinear part into a second order modulus form, while also bringing the linear part into Jordan normal form. 

We thus suppose the parameters are such that the Jacobian matrix has a pair of complex conjugate eigenvalues $\lambda_{\pm}=\mu\pm i\omega$, where $\mu,\omega,\lambda_3\in\R$, $\omega,\lambda_3\neq 0$, with the corresponding eigenvectors $\textbf{s}_1=\textbf{u}+i\textbf{v}$, $\textbf{s}_2=\textbf{u}-i\textbf{v}$ and $\textbf{s}_3$, where $\textbf{u},\textbf{v},\textbf{s}_3\in\R^3$. Such parameter choices are possible as it can be seen from inspecting the characteristic equation with the Routh-Hurwitz criterion; we omit details and refer to  \cite{SBeregi}. 
The transformation $\textbf{T}=(\textbf{u} | \textbf{v} |\textbf{s}_3)$ 
with the new state variables 
$(\xi_1,\xi_2,\xi_3)^\mathsf{T}=\textbf{T}^{-1}(\Omega,\psi,q)^\mathsf{T}$ turns \eqref{System1Beregi} into 
\begin{equation}
	\begin{pmatrix}
		\dot{\xi_1}\\
		\dot{\xi_2}\\
		\dot{\xi_3}
	\end{pmatrix}=
	\mathbf{A}
	\begin{pmatrix}
		\xi_1\\
		\xi_2\\
		\xi_3
	\end{pmatrix}+\textbf{h}_2(\xi_1,\xi_2,\xi_3), \quad 
	\mathbf{A} = \begin{pmatrix}
		\mu & \omega & 0\\
		-\omega & \mu & 0\\
		0 & 0 & \lambda_3
	\end{pmatrix},
	\label{system_xi}
\end{equation}
where $\textbf{h}_2$ contains the quadratic terms and reads, using the shorthand $[[\cdot]]:=\cdot|\cdot|$,  
\begin{equation}\textbf{h}_2(\xi_1,\xi_2,\xi_3)=
	\left(
	\tilde{T_{1}},
	\tilde{T_{2}},
	\tilde{T_{3}}
	\right)^\mathsf{T}
	[[u_3\xi_1+v_3\xi_2+s_3\xi_3]],
	\label{h_2_xi}
\end{equation}
{where $u_j, v_j, s_j$, $j\in\{1,2,3\}$ are the} components of the vectors $\textbf{u}, \textbf{v}, \textbf{s}_3$, respectively, and $$\tilde{T}_1:=\tilde{c}_4\frac{v_2s_3-v_3s_2}{\det(\textbf{T})}, \hspace*{0.5cm} \tilde{T}_2:=\tilde{c}_4\frac{s_2u_3-s_3u_2}{\det(\textbf{T})}, \hspace*{0.5cm} \tilde{T}_3:=\tilde{c}_4\frac{u_2v_3-u_3v_2}{\det(\textbf{T})}.$$

If $u_3=v_3=0$, then the nonlinear term $\mathbf{h}_{2}$ in \eqref{h_2_xi} is of second order modulus form:
\begin{equation}\textbf{h}_2(\xi_{1},\xi_{2},\xi_{3})=
	s_3\abs{s_3}\left( \tilde{T}_1, \tilde{T}_2, 0\right)^\mathsf{T}
	\xi_{3}\abs{\xi_{3}},
	\label{h_3_d0}
\end{equation}
where $\det(\textbf{T})\neq 0$ implies $s_3\neq 0$.
Here we need no further theory as we can directly solve {\eqref{system_xi}}: the equation for $\xi_{3}$ reads $\dot \xi_{3} = \lambda_3 \xi_{3}$ so that periodic solutions require $\xi_{3}(t)\equiv0$, i.e., $\xi_{3}(0)=0$.
The remaining system for $\xi_{1},\xi_{2}$ is then the purely linear part
\begin{equation*}
	\begin{pmatrix}
		\dot{\xi_{1}}\\
		\dot{\xi_{2}}
	\end{pmatrix}=\begin{pmatrix}
		\mu & \omega\\
		-\omega & \mu
	\end{pmatrix}\begin{pmatrix}
		\xi_{1}\\
		\xi_{2}
	\end{pmatrix},
\end{equation*}
and consists of periodic solutions (except the origin) for $\mu=0$. The unique branch of bifurcating periodic solutions is thus vertical, i.e., has $\mu=0$ constant.

Next, we consider the case when one of $u_{3}, v_{3}$ is nonzero. In order to simplify the nonlinear term, we apply a rotation $\mathbf{R}_{\theta}$ about the $\xi_{3}$-axis with angle $\theta$, 
which keeps the Jordan normal form matrix invariant, and in the new variables {$(v,w,u)^\mathsf{T}=\mathbf{R}_{\theta}^{-1}(\xi_1,\xi_2,\xi_3)^\mathsf{T}$, in particular $\xi_{3}=u$,} the nonlinear term from 
\eqref{h_2_xi} reads
\begin{align}
	\abs{u_3(v\cos{\theta}-\w\sin{\theta})+v_3(v\sin{\theta}+\w\cos{\theta})+s_3u} 
	= \abs{\tilde{d} v +\w(v_3\cos{\theta}-u_3\sin{\theta})+s_3u},
	\label{abs_eqs}
\end{align}
where $\tilde{d}=u_3\cos{{\theta}}+v_3\sin{{\theta}}$.
We select $\theta$ to simplify \eqref{abs_eqs}: if $u_3\neq 0$ we choose $\theta=\tilde{\theta}=\arctan\left(\frac{v_3}{u_3}\right)$ such that the coefficient of $\w$ in \eqref{abs_eqs} vanishes, i.e., $v_3\cos{\tilde{\theta}}-u_3\sin{\tilde{\theta}} = 0$. Note that $\tilde{d}\neq 0$ since otherwise $v_{3}\tan\tilde\theta=-u_{3}$, but $\tan \tilde\theta=\frac{v_3}{u_3}$, so both expressions together yield $v_{3}^2=-u_{3}^{2}$ and thus $u_{3}=v_{3}=0$ (which has been discussed above). 
If $u_3=0$ and $v_3\neq 0$ we choose $\theta=\tilde{\theta}=\arctan\left(-\frac{u_3}{v_3}\right)$ such that the coefficient of $v$ vanishes, i.e., $u_3\cos{\tilde{\theta}}+v_3\sin{\tilde{\theta}} = 0$, and the following computation is analogous.

Hence, in case $u_3\neq 0$, system \eqref{system_xi} becomes  
\begin{equation}
	\begin{pmatrix}
		\dot{v}\\
		\dot{\w}\\
		\dot{u}
	\end{pmatrix}=
	\mathbf{A}
	\begin{pmatrix}
		v\\
		\w\\
		u
	\end{pmatrix}+\textbf{h}_3(v,\w,u), \quad
	\label{after_rotation}
	\textbf{h}_3(v,\w,u)=
	\begin{pmatrix}
		\tilde{T}_1\cos{\tilde{\theta}}+\tilde{T}_2\sin{\tilde{\theta}}\\
		-\tilde{T}_1\sin{\tilde{\theta}}+\tilde{T}_2\cos{\tilde{\theta}}\\
		\tilde{T}_3
	\end{pmatrix}
	[[\tilde{d}v+s_3u]].
\end{equation} 
Notably, since $\tilde{d}\neq 0$, the nonlinear term is of second order modulus form for $s_3=0$, and we consider this degenerate situation first; as mentioned, the case $u_{3}=0, v_{3}\neq 0$ is analogous.

If $s_3=0$ (which means that the third component of the third eigenvector of the matrix 
$\mathbf{T}$ is zero) the nonlinear term in \eqref{after_rotation} is of second order modulus form. We can write system (\ref{after_rotation}) in the notation of system (\ref{3DAbstractSystem}):
\begin{equation}
	\begin{pmatrix}
		\dot{u}\\
		\dot{v}\\
		\dot{\w}
	\end{pmatrix} = \begin{pmatrix}
		c_1u + h_{11}v\abs{v}\\
		\mu v - \omega \w + a_{11}v\abs{v}\\
		\omega v + \mu \w + b_{11}v\abs{v}
	\end{pmatrix},
\end{equation}
where we changed $\omega$ to $-\omega$ and set $c_1:=\lambda_3$, $h_{11}:=\tilde{T}_3\tilde{d}|\tilde{d}|$, $a_{11}:=\left( \tilde{T}_1\cos{\tilde{\theta}}+\tilde{T}_2\sin{\tilde{\theta}} \right)\tilde{d}|\tilde{d}|$ and $b_{11}:=\left( -\tilde{T}_1\sin{\tilde{\theta}}+\tilde{T}_2\cos{\tilde{\theta}} \right)\tilde{d}|\tilde{d}|$. Since $s_{3}=0$, we have $a_{11}=0$ by choice of $\tilde{\theta}$, which implies $\sigma_{_\#}=0$. Furthermore, $\sigma_2=0$ holds so that Theorem \ref{2ndPart} does not apply. 
However, at $\mu=0$ we have 
$ \ddot{v} = -\omega^2v-\omega b_{11}v\abs{v} = -\frac{\D}{\D v}P$ with potential energy
$$ P(v) = \frac{\omega^2}{2}v^2+\frac{\omega b_{11}}{3}v^2\abs{v}, $$
which is globally convex if $\omega b_{11}\geq 0$ and otherwise convex in an interval around zero and concave outside of it. In both cases there is a vertical branch of periodic solutions, which is either unbounded or bounded by heteroclinic orbits.

\medskip
Let us now come back to \eqref{after_rotation} for $s_{3}\neq 0$, where the nonlinearity is of the form $\mathbf{h}_3=(h_{31}, h_{32}, h_{33})^\mathsf{T}[[\tilde{d}v+s_3u]]$. We first note that in the cylindrical coordinates from \eqref{e:cylindrical0} with the rescaled $u=r\tu$ {for $r\neq 0$} we have
\begin{align*}
	\dot r &= \mu r + r^{2}[[\tilde d \cos(\varphi) + s_{3}\tu]](h_{31}\cos(\varphi) + h_{32}\sin(\varphi)),\\
	\dot \varphi &= \omega + r[[\tilde d \cos(\varphi) + s_{3}\tu]](h_{32}\cos(\varphi) - h_{31}\sin(\varphi)),\\
	\dot\tu &= \lambda_{3}\tu + \tilde{T}_{3} r[[\tilde d \cos(\varphi) + s_3\tu]].
\end{align*}
Following the notation of the proof of Theorem~\ref{t_per_orb} we have the estimate $|\tu_{\infty}| = \calO(r_{\infty})$ and together with the expansion of the $[[\cdot]]$ terms from proof of Theorem \ref{Thm_Gen_Lin_Part}, we can write  
\[
\dot r = \mu r + r^{2}[[{\tilde{d}}\cos(\varphi)]](h_{31}\cos(\varphi) + h_{32}\sin(\varphi)) + \calO(r^{2} r_{\infty}).
\]
In the notation of Proposition~\ref{Thm_Gen}, in this case $\chi_{2}(\varphi)= [[{\tilde{d}}\cos(\varphi)]](h_{31}\cos(\varphi) + h_{32}\sin(\varphi))$, and according to Corollary \ref{hot2D} the bifurcating branch ist given by \eqref{General_Result} with 
\[
\int_{0}^{2\pi} \chi_{2}(\varphi) \D \varphi =  {\tilde{d}|\tilde{d}|} h_{31} \int_{0}^{2\pi} \cos^{2}(\varphi)|\cos(\varphi)| \D \varphi
= \frac 8 3 {\tilde{d}|\tilde{d}|} h_{31} = \frac 8 3{|\tilde{d}|} \frac{{\tilde{d}}s_3\tilde{c}_{4}}{\det(\mathbf{T})}.
\]
Since $\tilde{d}\neq 0$ the direction of bifurcation is determined by the sign of ${\tilde{d}}s_3\tilde{c}_{4}\det(\mathbf{T})$.
Note that {$\tilde{d}$, $s_3$, $\det(\mathbf{T})$ are independent of $\tilde{c}_4$, and} ${\tilde{d}}s_3\tilde{c}_{4}\det(\mathbf{T})=0$ requires $s_{3}=0$ as discussed above, or $\tilde{c}_{4}=0$, which implies vanishing nonlinearity. Thus, in all degenerate cases the branch is vertical and we have proven the following.

\begin{theorem}\label{t:shym}
	Any Hopf bifurcation in \eqref{System1Beregi} yields either a vertical branch of periodic solutions, or is super- or subcritical as in {Proposition \ref{Thm_Gen}}. 
	Using the above notation, the branch is vertical if and only if $\tilde{d}s_{3}\tilde{c}_{4}=0$, where $\tilde{d}=0$ means $u_{3}=v_{3}=0$. The bifurcation is supercritical if ${\tilde{d}}s_3\tilde{c}_{4}\det(\mathbf{T})<0$ and subcritical for positive sign. {In particular, reversing the sign of $\tilde{c}_4$ switches the criticality of the bifurcation.}
\end{theorem}

This conclusion is consistent with the results in \cite{SBeregi}.

\section{Discussion}
In this {paper} we have analyzed Hopf bifurcations in mildly nonsmooth systems with piecewise smooth nonlinearity for which standard center manifold reduction and normal form computations cannot be used. By averaging and a direct approach we have derived explicit analogues of Lyapunov coefficients and have discussed some codimension-one degeneracies as well as the modified scaling laws.
In an upcoming paper we will apply these results to models for controlled ship maneuvering, where stabilization by p-control induces a Hopf bifurcation. 

We believe this is an interesting class of equations from a theoretical as well as applied viewpoint, arising in a variety of models for macroscopic laws that lack smoothness in the nonlinear part. Among the perspectives, there is an analysis of normal forms for coefficients for other bifurcations in these models, such as 
Bogdanov-Takens points. Particularly interesting is the impact on scaling laws, including exponentially small effects 
for smooth vector fields.

\appendix
\section{Appendix}
\subsection{Averaging}
\label{NIT}
\subsubsection*{Near identity transformation}

For completeness, we provide details for the essentially standard normal form transformation used in the proof of Theorem \ref{t_averaging}.  
We set $f(x,\varphi,\epsilon) := 
\frac{m}{\omega} x + \frac{\chi_2}{\omega}  x^2 + \epsilon\left( \frac{\chi_3}{\omega} - \frac{\chi_2\Omega_1}{\omega^2} \right) x^3$, which is the leading part of the right-hand side of \eqref{x_form_for_av} up to a factor of $\epsilon$. We write $f$ as the sum of its mean, $\bar{f}$, and its oscillating part, $\tilde{f}$,
\begin{align*}
	\bar{f}(x) &= \frac{1}{2\pi}\int_0^{2\pi} f(x,\varphi,0)\mathrm{d}\varphi =  \frac{m}{\omega}x +\frac{2}{3\pi\omega}\sigma_{_\#}x^2,\\
	\tilde{f}(x,\varphi,\epsilon) &= f(x,\varphi,\epsilon)-\bar{f}(x) 
	= \left( \frac{\chi_2}{\omega} - \frac{2}{3\pi\omega}\sigma_{_\#} \right)x^2+\epsilon\left( \frac{\chi_3}{\omega} - \frac{\chi_2\Omega_1}{\omega^2} \right)x^3.
\end{align*}
Hence, $f(x,\varphi,\epsilon) = \bar{f}(x) + \tilde{f}(x,\varphi,\epsilon)$ and \eqref{x_form_for_av} reads
\begin{equation}\label{e:fsplit}
	x' = \epsilon \big(\bar{f}(x)+\tilde{f}(x,\varphi,\epsilon)\big) + \epsilon^2\mathcal{O}\left( \epsilon x^4+\abs{m}x^2 \right).
\end{equation}
The near-identity transformation which we will use has smooth coefficients $w_1, w_2$ and is of the form $$ x = y + \epsilon w_1(y,\varphi,\epsilon) + \epsilon^2 w_2(y,\varphi,\epsilon).$$
Differentiating this equation with respect to $\varphi$, using the right-hand side of \eqref{e:fsplit} for $x'$ and rearranging terms gives
$$ {y}' = \epsilon \big(\bar{f}(x)+\tilde{f}(x,\varphi,\epsilon)\big) + \epsilon^2\mathcal{O}\left( \epsilon x^4+\abs{m}x^2 \right) -\epsilon\left( \frac{\partial w_1}{\partial\varphi}+D_yw_1 y' \right) -\epsilon^2\left( \frac{\partial w_2}{\partial\varphi}+D_yw_2 y' \right). $$
Further rearrangements yield
$$ (1+\epsilon D_yw_1+\epsilon^2D_yw_2) {y}' = \epsilon \big(\bar{f}(x)+\tilde{f}(x,\varphi,\epsilon)\big) + \epsilon^2\mathcal{O}\left( \epsilon x^4+\abs{m}x^2 \right) -\epsilon\frac{\partial w_1}{\partial\varphi} -\epsilon^2\frac{\partial w_2}{\partial\varphi},$$
and inverting the first factor on the left and expanding in terms of $\epsilon$ we obtain 
\begin{equation}\label{e:y}
	\begin{aligned}
		{y}' = &\Big( 1-\epsilon D_yw_1+\mathcal{O}(\epsilon^2) \Big) \epsilon \left( \bar{f}(x)+\tilde{f}(x,\varphi,\epsilon) -\frac{\partial w_1}{\partial\varphi} -\epsilon\frac{\partial w_2}{\partial\varphi} \right)\\ 
		&+ \epsilon^2\mathcal{O}\left( \epsilon x^4+\abs{m}x^2 \right).
	\end{aligned}
\end{equation}
Concerning $f$ and the error term, directly using the transformation to $y$ we expand 
\begin{align*}
	\bar{f}(x) &= 
	\bar{f}(y)+D_y\bar{f}(y)\epsilon w_1+\mathcal{O}(\epsilon^2),\\ 
	\tilde{f}(x,\varphi,\epsilon) &= 
	\tilde{f}(y,\varphi,0) + D_y\tilde{f}(y,\varphi,0)\epsilon w_1 + \frac{\partial\tilde{f}}{\partial\epsilon}(y,\varphi,0)\epsilon +\mathcal{O}(\epsilon^2),
\end{align*}
and note that $\epsilon^2\mathcal{O}\left( \epsilon x^4+\abs{m}x^2 \right) =\mathcal{O}\left(\epsilon^3 + \epsilon^2\abs{m}y^2 \right)$. 
Substituting into \eqref{e:y} gives
\begin{align*}
	{y}' =& \Big( 1-\epsilon D_yw_1 \Big) \epsilon \left( \bar{f}(y + \epsilon w_1 + \epsilon^2 w_2)+\tilde{f}(y + \epsilon w_1 + \epsilon^2 w_2,\varphi,\epsilon) -\frac{\partial w_1}{\partial\varphi} -\epsilon\frac{\partial w_2}{\partial\varphi} \right) \\
	&+ \mathcal{O}\left(\epsilon^3 + \epsilon^3 y^4 + \epsilon^2\abs{m}y^2 \right)\\
	=& \epsilon\left( \bar{f}(y)+\tilde{f}(y,\varphi,0)-\frac{\partial w_1}{\partial\varphi} \right) 
	+ \epsilon^2\left( D_y f(y,\varepsilon,0)w_1(y,\varepsilon,0)+ \frac{\partial\tilde{f}}{\partial\epsilon}(y,\varphi,0)-\frac{\partial w_2}{\partial\varphi} \right) \\
	& -\epsilon^2 D_yw_1(y,\varphi,0)\left( \bar{f}(y)+\tilde{f}(y,\varphi,0)-\frac{\partial w_1}{\partial\varphi} \right) + \mathcal{O}(\epsilon^3).
\end{align*}
In order to cancel $\varphi$-dependent terms and obtain \eqref{truncated_av}, we choose periodic $w_1(y,\varphi,\epsilon), w_2(y,\varphi,\epsilon)$ as follows. Firstly, since 
$\int_0^{2\pi}\tilde{f}(y,\varphi,0)\D\varphi = \left( \frac{1}{\omega}\int_0^{2\pi}\chi_2(\varphi)\D\varphi - \frac{4}{3\omega}\sigma_{_\#} \right)y^2 = 0$,
we define $w_1$ via 
\[
\frac{\partial w_1}{\partial\varphi} = \tilde{f}(y,\varphi,0).
\] 
Secondly, for $w_2$ we take the ansatz 
\begin{equation}\label{w2}
	\frac{\partial w_2}{\partial\varphi} = D_yf(y,\varphi,0)w_1(y,\varphi,0)+\frac{\partial\tilde{f}}{\partial\epsilon}(y,\varphi,0)-D_yw_1(y,\varphi,0)\bar{f}(y)-\bar{f}_2(y),
\end{equation}
where we determine the function $\bar{f}_2(y)$ such that $\int_0^{2\pi}\frac{\partial w_2}{\partial\varphi}\D\varphi=0$ holds. To simplify the presentation of the computations of $\bar{f}_2(y)$, we define
\[
A(s) := \frac{\chi_2(s)}{\omega}-\frac{2\sigma_{_\#}}{3\pi\omega},
\]
and write then the expressions from the right-hand side of \eqref{w2} as
\begin{align*}
	D_yf(y,\varphi,0)w_1(y,\varphi,0) &= \frac{m}{\omega}\int_{0}^{\varphi}A(s)\dd s\, y^2 + \frac{2\chi_2(\varphi)}{\omega}\int_0^\varphi A(s)\dd s\, y^3, \\[1pt]
	\frac{\partial\tilde{f}}{\partial\epsilon}(y,\varphi,0) &= \left(\frac{\chi_3(\varphi)}{\omega}-\frac{\chi_2(\varphi)\Omega_1(\varphi)}{\omega^2}\right) y^3, \\[1pt]
	D_yw_1(y,\varphi,0)\bar{f}(y) &= \frac{2m}{\omega}\int_{0}^{\varphi}A(s)\dd s\, y^2 + \frac{4\sigma_{_\#}}{3\pi\omega}\int_{0}^{\varphi}A(s)\dd s\, y^3.
\end{align*}
Next, arranging the coefficients for $y^2$ and $y^3$, the integral of $\bar{f}_2(y)$ becomes
\begin{align*}
	\int_0^{2\pi}\bar{f}_2(y)\dd\varphi =& \int_0^{2\pi}\left( D_yf(y,\varphi,0)w_1(y,\varphi,0) +  \frac{\partial\tilde{f}}{\partial\epsilon}(y,\varphi,0) 
	-D_yw_1(y,\varphi,0)\bar{f}(y)
	\right)\dd\varphi \\ 
	=& -\frac{m}{\omega}\int_{0}^{2\pi}\int_0^\varphi A(s)\dd s\,\D\varphi\, y^2 \\
	&+ \Bigg[ \frac{2}{\omega}\int_{0}^{2\pi}\chi_2(\varphi)\int_{0}^{\varphi}A(s)\dd s\,\D\varphi 
	+ \frac{1}{\omega}\int_{0}^{2\pi} \chi_3(\varphi)\D\varphi \\
	&- \frac{1}{\omega^2}\int_{0}^{2\pi} \chi_2(\varphi)\Omega_1(\varphi) \D\varphi 
	- \frac{4\sigma_{_\#}}{3\pi\omega}\int_{0}^{2\pi}\int_0^\varphi A(s)\dd s\,\D\varphi 
	\Bigg] y^3,
\end{align*}
where the expression with $m$ gives a term of order $\calO(\abs{m}y^2)$. Substituting $A(s)$ and computing the corresponding integrals yield
\begin{equation*}
	\int_0^{2\pi}\bar{f}_2(y)\dd\varphi = \left(\frac{\pi}{4\omega}S_c + \frac{\pi}{4\omega^2}S_q + \frac{1}{\omega^2}\sigma_2 
	\right)y^3 + \calO\left(\sigma_{_\#}y^3+\abs{m}y^2\right),
\end{equation*}
which gives $\bar{f}_2(y)$ in \eqref{averaging_integrals}.

\subsubsection*{Integration of some periodic functions with absolute values}

Here we explain the computation of the integrals which yield the formulas for \eqref{averaging_integrals}. 

We show $\int_0^{2\pi}\chi_2(\varphi)\D\varphi=\frac{4}{3}\sigma_{_\#}$ in order to prove the first equality in \eqref{averaging_integrals}. 
On the one hand, the smooth terms of $\chi_2(\varphi)$ have clearly zero integral over $2\pi$ due to symmetry. On the other hand, all nonsmooth terms have the following feature: $a_{ij}$ and $b_{ij}$ are always multiplied by $c$ and $s$ respectively, times a common symmetric odd or even function, which implies that only one of the coefficients for each pair of $(i,j)$ will be nonzero after integrating over $2\pi$. For instance, $\int_0^{2\pi} c\abs{c}(a_{11}c+b_{11}s) \D\varphi = \frac{8a_{11}}{3}$ and $\int_0^{2\pi} s\abs{c}(a_{21}c+b_{21}s) \D\varphi = \frac{4b_{21}}{3}$. Finally, the factor $2$ for $a_{11}$ and $b_{22}$ is due to the product of purely $c$ or $s$: $\int_0^{2\pi} c^2\abs{c} \D\varphi = \int_0^{2\pi} s^2\abs{s} \D\varphi = 2 \int_0^{2\pi} c^2\abs{s} \D\varphi = 2\int_0^{2\pi} s^2\abs{c} \D\varphi$.

Next, we turn to the second equality in \eqref{averaging_integrals}.
Similarly as before, the integral of $\chi_3(\varphi)$ over $2\pi$ yields $\frac{\pi}{4}S_c$ since the integral of the arising mixed products between $c$ and $s$ vanish. Furthermore, different prefactors compared with $\sigma_{_\#}$ occur since the power of $c$ and $s$ are distinct as well: $\int_0^{2\pi} c^4 \D\varphi = \frac{3}{4} \int_0^{2\pi} c^2 \D\varphi = \frac{3}{4} \int_0^{2\pi} s^2 \D\varphi$. Moreover, $\int_0^{2\pi} \chi_2(\varphi)\Omega_1(\varphi) \D\varphi = -\frac{\pi}{4}S_q + \sigma_2$, where $S_q$ comes from integrating the smooth terms and $\sigma_2$ {from integrating} the others, which give the products between $a_i, b_j$, and $a_{ij}, b_{kl}$, respectively.

\subsection{$3$D system}
\subsubsection{Computation of $\gamma_{ij}$, $\delta_{ij}$ of system \eqref{UE_RE}}
\label{3D_gammas_deltas}
{In this subsection} we present the functions $\gamma_{ij}$, $\delta_{ij}$ of system \eqref{UE_RE}, used for the proof of Theorem \ref{thm3D}.

We first note that \eqref{UP-a} simplifies since $ \Psi_u(0,0,\varphi) = \partial_r\Psi_u(0,0,\varphi) = 0 $, and to ease notation we define
\begin{align*}
	p_1&:= \partial_u\Psi_u(0,0,\varphi) = \frac{c_1}{\omega}, &p_2(\varphi)&:= \frac{1}{2}\partial_u^2\Psi_u(0,0,\varphi) = \frac{c_2\omega-c_1\Omega_0(\varphi)}{\omega^2}, \\
	p_3(\varphi)&:= \frac{1}{2}\partial_r^2\Psi_u(0,0,\varphi) = \frac{\Upsilon(\varphi)}{\omega}, &p_4(\varphi)&:= \partial_{ur}^2\Psi_u(0,0,\varphi) = { \frac{\cos(\varphi) c_3 + \sin(\varphi) c_4}{\omega} }-\frac{c_1\Omega_1(\varphi)}{\omega^2}.
\end{align*}

Similarly, in \eqref{UP-b} we have $\Psi_r(0,0,\varphi) = \partial_u\Psi_r(0,0,\varphi) = \frac{1}{2}\partial_u^2\Psi_r(0,0,\varphi) = 0$, and we define 
\begin{align*}
	k_1&:=\partial_r\Psi_r(0,0,\varphi) = \frac{\mu}{\omega}, \quad
	k_2(\varphi):= \frac{1}{2}\partial_r^2\Psi_r(0,0,\varphi) = \frac{\chi_2(\varphi)\omega-\mu\Omega_1(\varphi)}{\omega^2}, \\ k_3(\varphi)&:= \partial_{ur}^2\Psi_r(0,0,\varphi) = \frac{\chi_1(\varphi)\omega-\mu\Omega_0(\varphi)}{\omega^2}.
\end{align*}

The expressions for $\gamma_{ij}$ and $\delta_{ij}$ follow from solving the ODEs that arise upon substituting (\ref{UE_RE}) into (\ref{UP-a}) and (\ref{UP-b}), and matching the coefficients of the powers of $u_0$ and $r_0$. We start with {the equations and initial conditions to obtain} $\gamma_{01}, \gamma_{10}, \delta_{01}, \delta_{10}$:
\begin{alignat*}{3}
	\gamma_{01}'&=p_1\gamma_{01},\hspace*{0.15cm} \gamma_{01}(0)=0 &&\Rightarrow \gamma_{01}\equiv 0, \hspace*{0.6cm} 
	&\delta_{01}'&=k_1 \delta_{01},\hspace*{0.15cm} \delta_{01}(0)=1 \Rightarrow \delta_{01}(\varphi)=e^{k_1\varphi}, \\
	\gamma_{10}'&=p_1\gamma_{10},\hspace*{0.15cm} \gamma_{10}(0)=1 &&\Rightarrow \gamma_{10}(\varphi)=e^{p_1\varphi}, \hspace*{0.6cm} 
	&\delta_{10}'&=k_1 \delta_{10},\hspace*{0.15cm} \delta_{10}(0)=0 \Rightarrow \delta_{10}\equiv 0.
\end{alignat*}
Using these, we solve {the corresponding equations} for the remaining coefficients as follows:
\begin{align*}
	\gamma_{20}'=p_1\gamma_{20}+p_2(\varphi)\gamma_{10}^2,\,\,\, \gamma_{20}(0)=0 &\Rightarrow \gamma_{20}=  \int_0^\varphi e^{p_1(\varphi+s)}p_2(s)\D s, \\
	\gamma_{02}'=p_1\gamma_{02}+p_3(\varphi)\delta_{01}^2,\,\,\, \gamma_{02}(0)=0 &\Rightarrow \gamma_{02}(\varphi)= \int_0^\varphi e^{p_1(\varphi-s)+2k_1s}p_3(s)\D s, \\
	\gamma_{11}'=p_1\gamma_{11}+p_4(\varphi)\gamma_{10}\delta_{01},\,\,\, \gamma_{11}(0)=0 &\Rightarrow \gamma_{11}(\varphi)= \int_0^\varphi e^{p_1\varphi+k_1s}p_4(s)\D s,
\end{align*}
\begin{align*}\delta_{20}'=k_1\delta_{20},\,\,\, \delta_{20}(0)=0 &\Rightarrow \delta_{20}\equiv 0, \\
	\delta_{02}'=k_1 \delta_{02}+k_2(\varphi)\delta_{01}^2 ,\,\,\, \delta_{02}(0)=0 &\Rightarrow \delta_{02}(\varphi)= \int_0^\varphi e^{k_1(\varphi+s)}k_2(s)\D s, \\
	\delta_{11}'=k_1\delta_{11}+k_3(\varphi)\gamma_{10}\delta_{01},\,\,\, \delta_{11}(0)=0 &\Rightarrow \delta_{11}(\varphi)= \int_0^\varphi e^{k_1\varphi+p_1s}k_3(s)\D s.
\end{align*}
Since we aim to solve the boundary value problem $0 = u(2\pi) - u(0)$, $0 = r(2\pi) - r(0)$, let $\ogamma_{ij}:=\gamma_{ij}(2\pi)-\gamma_{ij}(0)$, $\odelta_{ij}:=\delta_{ij}(2\pi)-\delta_{ij}(0)$, $\forall i,j\geq 0$. Direct computation of $\ogamma_{10}$ and $\ogamma_{20}$ gives the expressions in \eqref{ogammas_odeltas}. For the other functions, we consider the corresponding integrals and Taylor expand in $\mu=0$, which results in $\ogamma_{11}$, $\odelta_{01}$, $\odelta_{02}$ and $\odelta_{11}$ shown in \eqref{ogammas_odeltas}. For illustration of details omitted, we next present the full derivation for the explicit form of $\ogamma_{02}$. 
Taylor expansion of $e^{s\frac{2\mu}{\omega}}$ in $\mu=0$ 
and rearranging the terms in the integral of $\gamma_{02}$ gives
\begin{equation}\label{ogamma_02}
	\ogamma_{02} = \frac{1}{\omega}e^{\frac{2\pi c_1}{\omega}} \int_0^{2\pi}e^{-s\frac{c_1}{\omega}}\Upsilon(s) \D s + \frac{2\mu}{\omega^2}e^{\frac{2\pi c_1}{\omega}} \int_0^{2\pi} s e^{-s\frac{c_1}{\omega}}\Upsilon(s) \D s + \calO\left(\mu^2\right).
\end{equation}
We compute the two integrals in \eqref{ogamma_02} separately. The first one readily expands {in $c_1=0$} as 
\begin{align*}
	\frac{1}{\omega}&e^{\frac{2\pi c_1}{\omega}} \int_0^{2\pi}e^{-s\frac{c_1}{\omega}}\Upsilon(s) \D s = \frac{\omega}{c_1(c_1^2+4\omega^2)} \Bigg[ 2\left(e^{\frac{3\pi}{2\omega}c_1}-e^{\frac{\pi}{2\omega}c_1}\right)(c_1h_{21}-2h_{11}\omega)\\ &+ \left(e^{\frac{2\pi}{\omega}c_1}-2e^{\frac{\pi}{\omega}c_1}+1\right)(c_1h_{12}+2h_{22}\omega) + \left(e^{\frac{2\pi}{\omega}c_1}-1\right)\left(c_1h_{21}+c_1c_5+\frac{c_1^2h_{11}}{\omega}+2h_{11}\omega\right) \Bigg]\\
	&= \frac{1}{c_1^2+4\omega^2} \Bigg[ {2\pi}(c_1h_{21}-2h_{11}\omega) 
	+ 2\pi\left(c_1h_{21}+c_1c_5+\frac{c_1^2h_{11}}{\omega}+2h_{11}\omega\right) \Bigg] + \calO\left(c_1^2\right).
\end{align*}
In particular, it vanishes for $c_1=0$.

For the second integral of \eqref{ogamma_02} we proceed similarly. Its explicit expression initially reads
\begin{align*}
	\frac{2\mu}{\omega^2}e^{\frac{2\pi c_1}{\omega}} \int_0^{2\pi} s e^{-s\frac{c_1}{\omega}}\Upsilon(s) \D s &= -\frac{4\mu}{\omega c_1^2(c_1^2+4\omega^2)^2} \Bigg[ \omega^2c_1^3(6h_{11}\pi-h_{12}) +\omega c_1^4\pi(c_5+h_{21}) + c_1^5h_{11}\pi \\
	&+ \frac{\omega}{2}\Big\{ (-16h_{11}\omega^4-12c_1^2h_{11}\omega^2+4c_1^3h_{21}\omega)\left(e^{\frac{\pi}{2\omega}c_1}-e^{\frac{3\pi}{2\omega}c_1}\right)\\
	&+ \pi(-8c_1h_{11}\omega^3+4c_1^2h_{21}\omega^3+c_1^4h_{21}-2h_{11}c_1^3\omega)\left(3e^{\frac{\pi}{2\omega}c_1}-e^{\frac{3\pi}{2\omega}c_1}\right) \Big\} \\
	&+ \omega\Bigg\{ (3c_1^2h_{22}\omega^2+c_1^3h_{12}\omega+4h_{22}\omega^4)\left(-e^{\frac{2\pi}{\omega}c_1}+2e^{\frac{\pi}{\omega}c_1}\right)\\
	&+ \left(\omega^44h_{11}+\omega^2 c_1^2h_{11}+\omega c_1^3(h_{21}+c_5)+\frac{1}{2}c_1^4h_{11}\right)\left(-e^{\frac{2\pi}{\omega}c_1}+1\right) \Bigg\} \\
	&+ (\omega^44\pi c_1h_{22}+\omega^34\pi c_1^2h_{12}+\omega^22\pi c_1^3h_{22}+\omega\pi c_1^4h_{12})\left(e^{\frac{\pi}{\omega}c_1}-1\right) \\
	&- 4\omega^5h_{22}+\omega^48c_1 h_{11}\pi + \omega^3c_1^2( 4\pi(c_5+h_{21}c_1)-3h_{22} )\Bigg].
\end{align*}
Expanding again the exponential functions in $c_1=0$ and simplifying coefficients, at $c_1=0$ we obtain
\begin{equation*}
	\frac{2\mu}{\omega^2}e^{\frac{2\pi c_1}{\omega}} \int_0^{2\pi} s e^{-s\frac{c_1}{\omega}}\Upsilon(s) \D s = -\mu\frac{\pi}{\omega^2}\big( 2h_{21} + c_5 + \pi h_{22} \big).
\end{equation*}

\subsubsection{Computation of $\delta_{03}$ of system \eqref{UE_RE_Second}}
\label{3D_gammas_deltas_second}
Similar to the previous subsection, we present the function $\delta_{03}$ of system \eqref{UE_RE_Second}, used in the proof of Corollary \ref{Second:c:3D}.

Taylor {expanding the right-hand sides of \eqref{eq_u_r}} in $(u,r)=(0,0)$ up to fourth order gives the following; we omit the dependence of $\Psi_u$ and $\Psi_r$ on  $(u,r,\varphi)$ at $(0,0,\varphi)$  to simplify the notation: 
\begin{subequations}
	\begin{align} 
		\begin{split} \label{UP-a3}
			u' =& \Psi_u + \partial_u\Psi_uu + \partial_r\Psi_ur + \frac{1}{2}\partial_u^2\Psi_uu^2 + \frac{1}{2}\partial_r^2\Psi_ur^2 + \partial_{ur}^2\Psi_uur \\
			&+ \frac{1}{3!}\partial_u^3\Psi_uu^3 + \frac{1}{3!}\partial_r^3\Psi_ur^3 + \frac{1}{2}\partial_{u^2r}^3\Psi_uu^2r + \frac{1}{2}\partial_{ur^2}^3\Psi_uur^2 + \mathcal{O}\left(4\right),
		\end{split}\\[5pt]
		\begin{split} \label{UP-b3}
			r' =&  \Psi_r + \partial_u\Psi_ru + \partial_r\Psi_rr + \frac{1}{2}\partial_u^2\Psi_ru^2 + \frac{1}{2}\partial_r^2\Psi_rr^2 + \partial_{ur}^2\Psi_rur  \\
			&+ \frac{1}{3!}\partial_u^3\Psi_ru^3 + \frac{1}{3!}\partial_r^3\Psi_rr^3 + \frac{1}{2}\partial_{u^2r}^3\Psi_ru^2r + \frac{1}{2}\partial_{ur^2}^3\Psi_rur^2 + \mathcal{O}\left(4\right),
		\end{split}
	\end{align}
\end{subequations}
where $\frac{1}{3!} \partial_u^3\Psi_r = 0$. We also set  
$k_4(\varphi):= \frac{1}{3!} \partial_r^3\Psi_r = \frac{-\omega\chi_2(\varphi)\Omega_1(\varphi)+\mu\Omega_1(\varphi)^2}{\omega^3}.$

\medskip
Substituting (\ref{UE_RE_Second}) into (\ref{UP-a3}) and (\ref{UP-b3}) and matching the coefficients of the powers of $u_0$ and $r_0$ we get to solve a set of ODEs in order to obtain the expressions for $\gamma_{ij}$, $\delta_{ij}$ for $i,j$ such that $i+j=3$, which are rather lengthy. We show $\delta_{03}$, which is the only one required for the leading order analysis in the degenerate case $\sigma_{_\#}=0$, $c_1\neq 0$ in Corollary~\ref{Second:c:3D}:
\begin{equation}\label{delta_03}
	\begin{aligned}
		\delta_{03}'&=k_4\delta_{01}^3+2k_2\delta_{01}\delta_{02}+k_3\delta_{01}\gamma_{02}+k_1\delta_{03},\,\,\, \delta_{03}(0)=0 \Rightarrow \\
		\delta_{03}(\varphi) &= e^{k_1\varphi}\int_0^\varphi e^{-k_1s}\delta_{01}(s)\big[2k_2(s)\delta_{02}(s)+k_3(s)\gamma_{02}(s)+k_4(s)\delta_{01}(s)^2\big] \D s.
	\end{aligned}
\end{equation}

\section*{Acknowledgments}
The authors are grateful to Alan Champneys (University of Bristol), Ivan Ovsyannikov (University of Hamburg) and Martin Rasmussen (Imperial College London) for fruitful discussions. This research has been supported by the Deutsche Forschungsgemeinschaft (DFG, German Research Foundation) within the framework of RTG ``$\pi^3$: Parameter Identification - Analysis, Algorithms, Applications'' - Projektnummer 281474342/GRK2224/1.

\bibliographystyle{siam}
{\small \bibliography{arxiv_stein_macher}}

\begin{thebibliography}{10}

\bibitem{Amari1977}
{\sc S.-i. Amari}, {\em Dynamics of pattern formation in lateral-inhibition
  type neural fields}, Biological Cybernetics, 27 (1977), pp.~77--87.

\bibitem{InitialPaper}
{\sc M.~Apri, N.~Banagaaya, J.~v.~d. Berg, R.~Brussee, D.~Bourne, T.~Fatima,
  F.~Irzal, J.~D.~M. Rademacher, B.~Rink, F.~Veerman, and S.~Verpoort}, {\em
  Analysis of a model for ship manoeuvering}, in Proceedings of the
  Seventy-ninth European Study Group Mathematics with Industry, R.~Planque,
  S.~Bhulai, J.~Hulshof, W.~Kager, and R.~T.O., eds., Vrije Universiteit
  Uitgeverij, 2012, pp.~83--116.

\bibitem{IntegralManifold}
{\sc B.~Aulbach and T.~Wanner}, {\em Integral manifolds for {C}arath\'{e}odory
  type differential equations in {B}anach spaces}, in Six lectures on dynamical
  systems, World Sci. Publ., River Edge, NJ, 1996, pp.~45--119.

\bibitem{SBeregi}
{\sc S.~Beregi, D.~Takács, and C.~H\H{o}s}, {\em Nonlinear analysis of a
  shimmying wheel with contact-force characteristics featuring higher-order
  discontinuities}, Nonlinear Dynamics, 90 (2017), pp.~877--888.

\bibitem{BuzMedTor2018}
{\sc C.~A. Buzzi, J.~C. Medrado, and J.~Torregrosa}, {\em Limit cycles in
  4-star-symmetric planar piecewise linear systems}, Journal of Differential
  Equations, 268 (2020), pp.~2414--2434.

\bibitem{ChowHale}
{\sc S.-N. Chow and J.~K. Hale}, {\em Methods of bifurcation theory}, vol.~251
  of Grundlehren der Mathematischen Wissenschaften [Fundamental Principles of
  Mathematical Science], Springer-Verlag, New York, 1982.

\bibitem{CoddLev}
{\sc E.~A. Coddington and N.~Levinson}, {\em Theory of ordinary differential
  equations}, McGraw-Hill Book Company, 1955.

\bibitem{CollGasullProhens}
{\sc B.~Coll, A.~Gasull, and R.~Prohens}, {\em Degenerate {H}opf bifurcations
  in discontinuous planar systems}, Journal of Mathematical Analysis and
  Applications, 253 (2001), pp.~671--690.

\bibitem{Coombes2005}
{\sc S.~Coombes}, {\em Waves, bumps, and patterns in neural field theories},
  Biological Cybernetics, 93 (2005), pp.~91--108.

\bibitem{BookAlan}
{\sc M.~di~Bernardo, C.~J. Budd, A.~R. Champneys, and P.~Kowalczyk}, {\em
  Piecewise-smooth dynamical systems}, vol.~163 of Applied Mathematical
  Sciences, Springer-Verlag London, 2008.

\bibitem{ReviewAlan}
{\sc M.~di~Bernardo, C.~J. Budd, A.~R. Champneys, P.~Kowalczyk, A.~B. Nordmark,
  G.~Olivar~Tost, and P.~T. Piiroinen}, {\em Bifurcations in nonsmooth
  dynamical systems}, SIAM Review, 50 (2008), pp.~629--701.

\bibitem{auto}
{\sc E.~J. Doedel, A.~R. Champneys, T.~F. Fairgrieve, Y.~A. Kuznetsov,
  B.~Sandstede, and X.~J. Wang}, {\em Auto97 : Software for continuation and
  bifurcation problems in ordinary differential equations}, Technical report,
  California Institute of Technology, Pasadena CA 91125,  (1998).

\bibitem{Filippov1988}
{\sc A.~Filippov}, {\em Differential Equations with Discontinuous Righthand
  Sides}, vol.~18, Springer Netherlands, 1988.

\bibitem{FossenHandbook}
{\sc T.~I. Fossen}, {\em Handbook of marine craft hydrodynamics and motion
  control}, John Wiley \& Sons, Ltd., 2011.

\bibitem{GasTorr2003}
{\sc A.~Gasull and J.~Torregrosa}, {\em Center-focus problem for discontinuous
  planar differential equations}, International Journal of Bifurcation and
  Chaos in Applied Sciences and Engineering, 13 (2003), pp.~1755--1765.

\bibitem{Guckenheimer}
{\sc J.~Guckenheimer and P.~Holmes}, {\em Nonlinear oscillations, dynamical
  systems, and bifurcations of vector fields}, vol.~42 of Applied Mathematical
  Sciences, Springer-Verlag, New York, 1983.

\bibitem{Harris2015}
{\sc J.~Harris and B.~Ermentrout}, {\em Bifurcations in the {W}ilson-{C}owan
  equations with nonsmooth firing rate}, SIAM Journal on Applied Dynamical
  Systems, 14 (2015), pp.~43--72.

\bibitem{Hartman}
{\sc P.~Hartman}, {\em Ordinary differential equations}, vol.~38 of Classics in
  Applied Mathematics, Society for Industrial and Applied Mathematics,
  Philadelphia, PA, 2002.

\bibitem{BookRasmussen}
{\sc P.~E. Kloeden and M.~Rasmussen}, {\em Nonautonomous dynamical systems},
  vol.~176 of Mathematical Surveys and Monographs, American Mathematical
  Society, 2011.

\bibitem{Kunze2000}
{\sc M.~Kunze}, {\em Non-Smooth Dynamical Systems}, vol.~1744 of Lecture Notes
  in Mathematics, Springer-Verlag, Berlin Heidelberg, 2000.

\bibitem{KuepperHoshamWeiss2013}
{\sc T.~K\"{u}pper, H.~Hosham, and D.~Weiss}, {\em Bifurcation for Non-smooth
  Dynamical Systems via Reduction Methods}, vol.~35, Springer Basel, 2013,
  pp.~79--105.

\bibitem{KuepperHosham2010}
{\sc T.~K\"{u}pper and H.~A. Hosham}, {\em Reduction to invariant cones for
  non-smooth systems}, Mathematics and Computers in Simulation, 81 (2011),
  pp.~980--995.

\bibitem{KuepperMoritz}
{\sc T.~K\"{u}pper and S.~Moritz}, {\em Generalized {H}opf bifurcation for
  non-smooth planar systems}, Philos. Trans. R. Soc. Lond. Ser. A Math. Phys.
  Eng. Sci., 359 (2001), pp.~2483--2496.

\bibitem{Kuznetsov}
{\sc Y.~A. Kuznetsov}, {\em Elements of applied bifurcation theory}, vol.~112
  of Applied Mathematical Sciences, Springer-Verlag, New York, second~ed.,
  1998.

\bibitem{Leine}
{\sc R.~I. Leine}, {\em Bifurcations of equilibria in non-smooth continuous
  systems}, Physica D: Nonlinear Phenomena, 223 (2006), pp.~121--137.

\bibitem{LlibreEtAl2017}
{\sc J.~Llibre, D.~D. Novaes, and C.~A. Rodrigues}, {\em Averaging theory at
  any order for computing limit cycles of discontinuous piecewise differential
  systems with many zones}, Physica D: Nonlinear Phenomena, 353--354 (2017),
  pp.~01--10.

\bibitem{LlibrePonce}
{\sc J.~Llibre, M.~Ord\'{o}\~{n}ez, and E.~Ponce}, {\em On the existence and
  uniqueness of limit cycles in planar continuous piecewise linear systems
  without symmetry}, Nonlinear Anal. Real World Appl., 14 (2013),
  pp.~2002--2012.

\bibitem{NonsmoothSurvey2012}
{\sc O.~Makarenkov and J.~S.~W. Lamb}, {\em Dynamics and bifurcations of
  nonsmooth systems: A survey}, Physica D. Nonlinear Phenomena, 241 (2012),
  pp.~1826--1844.

\bibitem{Marsden76}
{\sc J.~E. Marsden and M.~McCracken}, {\em The {H}opf bifurcation and its
  applications}, Applied Mathematical Sciences, Springer-Verlag New York, 1976.

\bibitem{Averaging}
{\sc J.~A. Sanders, F.~Verhulst, and J.~Murdock}, {\em Averaging methods in
  nonlinear dynamical systems}, vol.~59 of Applied Mathematical Sciences,
  Springer-Verlag, New York, second~ed., 2007.

\bibitem{Simpson2019}
{\sc D.~J.~W. Simpson}, {\em Twenty {H}opf-like bifurcations in
  piecewise-smooth dynamical systems}, arXiv Preprint, arXiv:1905.01329 (2019).

\bibitem{TianThesis}
{\sc Y.~Tian}, {\em Bifurcation of limit cycles in smooth and non-smooth
  dynamical systems with normal form computation}, PhD thesis, The University
  of Western Ontario, Canada, 2014.

\bibitem{ToxopeusPaper}
{\sc S.~Toxopeus}, {\em Deriving mathematical manoeuvring models for bare ship
  hulls using viscous flow calculations}, Journal of Marine Science and
  Technology, 14 (2009), pp.~30--38.

\bibitem{KuepperHoshamWeiss2012}
{\sc D.~Weiss, T.~K\"{u}pper, and H.~Hosham}, {\em Invariant manifolds for
  nonsmooth systems}, Physica D: Nonlinear Phenomena, 241 (2012),
  pp.~1895--1902.

\bibitem{ZouKuepper2005}
{\sc Y.~Zou and T.~K\"{u}pper}, {\em Generalized {H}opf bifurcation emanated
  from a corner for piecewise smooth planar systems}, Nonlinear Analysis:
  Theory, Methods \& Applications, 62 (2005), pp.~1--17.

\bibitem{ZouKuepperBeyn}
{\sc Y.~Zou, T.~K\"upper, and W.-J. Beyn}, {\em Generalized {H}opf bifurcation
  for planar filippov systems continuous at the origin}, Journal of Nonlinear
  Science, 16 (2006), pp.~159--177.

\end{thebibliography}

\end{document}